\documentclass[12pt,CJK]{article}
\usepackage{amssymb}
\usepackage{}
\usepackage{stmaryrd}
\usepackage{amsfonts,amsmath,amssymb,amsthm,mathrsfs}
\usepackage{latexsym}
\usepackage{graphicx}
\usepackage{indentfirst}
\usepackage{color}
\usepackage{fancyhdr}
\usepackage[colorlinks,linktocpage,linkcolor=blue]{hyperref}

\theoremstyle{plain}
\newtheorem{thm}{Theorem}[section]
\newtheorem{definition}{Definition}
\newtheorem{lemma}[thm]{Lemma}
\newtheorem{corollary}[thm]{Corollary}
\newtheorem{remark}[thm]{Remark}
\numberwithin{equation}{section}

\theoremstyle{remark}

\def\Xint#1{\mathchoice
  {\XXint\displaystyle\textstyle{#1}}%
  {\XXint\textstyle\scriptstyle{#1}}%
  {\XXint\scriptstyle\scriptscriptstyle{#1}}%
  {\XXint\scriptscriptstyle\scriptscriptstyle{#1}}%
  \!\int}
\def\XXint#1#2#3{{\setbox0=\hbox{$#1{#2#3}{\int}$}
  \vcenter{\hbox{$#2#3$}}\kern-.5\wd0}}

\def\dashint{\Xint-}

\title{\textbf{Quantitative Estimates in Homogenization of Parabolic Systems of Elasticity\\
in Lipschitz Cylinders}}

\author{
Qiang Xu,\thanks{Corresponding author.}
\thanks{Email: xuqiang@math.pku.edu.cn.}
\quad Shulin Zhou
\thanks{Email: szhou@math.pku.edu.cn.}\\
School of Mathematical Sciences, Peking University, \\
Beijing, 100871, PR China. }

\begin{document}
\allowdisplaybreaks
\maketitle
\begin{abstract}
This paper is devoted to establish an almost sharp error estimate $O(\varepsilon\ln(1/\varepsilon))$ in $L^2$-norm for homogenization of parabolic systems of elasticity with initial-Dirichlet conditions in a Lipschitz cylinder.
To achieve the goal, with the parabolic distance function being a weight, we first develop
some new weighted-type inequalities for the smoothing operator
at scale $\varepsilon$ in terms of t-anisotropic Sobolev spaces, and then reduce
all the problems to three kinds of estimate for the homogenized system, in which a weighted-type Caccioppoli's inequality on time-layer has been found. Throughout the paper,
we do not require any smoothness on coefficients compared to
the arguments investigated by C.Kenig, F. Lin and Z. Shen in \cite{SZW2}. This study
can be considered to be a further development of \cite{GZS} and \cite{QX2}.
\end{abstract}

\section{Introduction and main results}

In recent years, J. Geng and Z. Shen  in \cite{GZ,GZS} have made some significant developments in quantitative homogenization of parabolic systems
with time-dependent periodic coefficients, such as the uniform $W^{1,p}$, H\"older and interior
Lipschitz estimates, as well as a sharp $L^2$ convergence rate. Meanwhile, for parabolic systems only involving spatial-dependent periodic coefficients,
a sharp $L^2$ error estimate has also been obtained by Yu. Meshkova and T. Suslina in \cite{MYMTS}.
However, all the results in previous references were merely established for smooth cylinders.
In this paper, we manage to study the nonsmooth case.

We begin by stating the initial-boundary value problems that we will investigate and
sketching our main results.
Let $\Omega\subset\mathbb{R}^d$ with $d\geq 3$ be a bounded Lipschitz domain. For $T$ satisfying
$ T\in(0,\infty)$,
we define the parabolic cylinder as $\Omega_T= \Omega\times(0,T]$, and the
lateral boundary of $\Omega_T$ as $S_T = \partial\Omega\times(0,T]$,
while the parabolic boundary of $\Omega_T$ is written by
$\partial_p\Omega_T = \overline{\Omega}_T\setminus\Omega_T$.

For given data $F$, $h$ and $g$ specified in some proper spaces,
we consider the following parabolic system of elasticity with a initial-Dirichlet condition:
\begin{equation*}
(\text{DP}_\varepsilon)\left\{\begin{aligned}
\big(\partial_t + \mathcal{L}_\varepsilon\big)(u_\varepsilon)
& = F &~&\text{in}~ ~\Omega_T,\\
u_\varepsilon & = g &~&\text{on}~ S_T,\\
u_\varepsilon & = h &~&\text{on}~ \Omega\times\{t=0\},
\end{aligned}\right.
\end{equation*}
where $\varepsilon>0$ is a small parameter, and
\begin{equation*}
 \mathcal{L}_\varepsilon = -\text{div}(A(x/\varepsilon,t/\varepsilon^2))
=-\frac{\partial}{\partial x_i}\bigg[a_{ij}^{\alpha\beta}
\Big(\frac{x}{\varepsilon},\frac{t}{\varepsilon^2}\Big)\frac{\partial}{\partial x_j}\bigg].
\end{equation*}
(Einstein's convention for summation is used throughout.)

Let $A(y,\tau) = \big(a_{ij}^{\alpha\beta}(y,\tau)\big)$ with
$1\leq i,j,\alpha,\beta\leq d$ be real and satisfy two hypotheses.
\begin{enumerate}
  \item Elasticity: there holds
\begin{equation}\label{c:1}
  \begin{array}{c}
   a_{ij}^{\alpha\beta}(y,\tau)=  a_{ji}^{\beta\alpha}(y,\tau)
   =  a_{\alpha j}^{i\beta}(y,\tau) \\
 \mu_1 |\xi|^2 \leq a_{ij}^{\alpha\beta}(y,\tau)\xi_i^\alpha\xi_j^\beta\leq \mu_2 |\xi|^2
  \end{array}
\end{equation}
for any  $(y,\tau)\in\mathbb{R}^{d+1}$ and symmetric matrix $\xi=(\xi_i^\alpha)\in \mathbb{R}^{d\times d}$, where $\mu_1,\mu_2>0$.
  \item Periodicity:  for $(z,s)\in \mathbb{Z}^{d+1}$ and $(y,\tau)\in\mathbb{R}^{d+1}$,
\begin{equation}\label{c:2}
 A(y+z,\tau+s) = A(y,\tau).
\end{equation}
\end{enumerate}

We now state the first result.
\begin{thm}\label{thm:1.2}
Suppose that $A$ satisfies $\eqref{c:1}$ and $\eqref{c:2}$. Let
$F\in L^{2}(0,T;(H^{-1}(\Omega))^d)$, $h\in (L^2(\Omega))^d$ and
$g\in (H^{\frac{1}{2},\frac{1}{4}}(S_T))^d$.
Then
for a family of weak solutions
$u_\varepsilon\in L^2(0,T;(H^1(\Omega))^d)$ with
$\partial_t u_\varepsilon\in L^2(0,T;(H^{-1}(\Omega))^d)$ to $(\emph{DP}_\varepsilon)$,
there holds $u_\varepsilon \to u_0$ strongly in $(L^2(\Omega_T))^d$, as
$\varepsilon\to 0$,
where $u_0$ satisfies the homogenized system:
\begin{equation*}
(\emph{DP}_0)\left\{\begin{aligned}
\big(\partial_t + \mathcal{L}_0\big)(u_0)
& = F &~&\emph{in}~ ~\Omega_T,\\
u_0 & = g &~&\emph{on}~ S_T,\\
u_0 & = h &~&\emph{on}~ \Omega\times\{t=0\},
\end{aligned}\right.
\end{equation*}
in a weak sense,
and $\mathcal{L}_0 = \emph{div}(\widehat{A}\nabla)$ is an operator with the constant coefficient specified in
$\eqref{eq:2.7}$.
\end{thm}

Here $F\in L^{2}(0,T;(H^{-1}(\Omega))^d)$ means its component
$F^\alpha\in L^{2}(0,T;H^{-1}(\Omega))$ with $\alpha = 1,\cdots,d$, and its
definition may be found in \cite[pp.374]{LCE}.
By the same convention, the notation $(L^2(\Omega))^d$, $(H^{\frac{1}{2},\frac{1}{4}}(S_T))^d$,
and $L^2(0,T;(H^1(\Omega))^d)$ represent
the corresponding product spaces, in which
the space $H^{\frac{1}{2},\frac{1}{4}}(S_T)$
is a Hilbert space collecting functions
with one half of a spatial derivative and one quarter of a time derivative in $L^2(S_T)$
(see \cite[pp.502]{MC}), and
the definition of $L^2(0,T;H^1(\Omega))$ is stated in \cite[pp.301]{LCE}.
We remark that there is little effort made to distinguish vector-valued functions or function spaces from their real-valued counterparts in the paper.

Up to the first Korn inequality stated in \cite[pp.371]{VSO},
the proof of Theorem $\ref{thm:1.2}$ is quite similar to that given
for \cite[Theorem 3.1]{GZ} in the case of $g=0$, and we just outline it in Section $\ref{section:2}$. Also, on account of \cite[Remark 3.2]{GZ}
the homogenized system $(\text{DP}_0)$ is still a parabolic system with the constant coefficient  satisfying
the same elasticity condition $\eqref{c:1}$.
The above result is just a qualitative investigation, and this kind research may trace back to
1970s, which was summarized in the monograph \cite[pp.140]{ABJLGP}. In the paper,
we will seek for a sharply quantitative estimate on rate of convergence between $u_\varepsilon$ and
$u_0$ in $L^2(\Omega_T)$, and
the following theorem gives the main result.

\begin{thm}\label{thm:1.1}
Suppose that $A$ satisfies the conditions $\eqref{c:1}$ and $\eqref{c:2}$.
Given $F\in (L^2(\Omega_T))^d$, $g\in ({_0H^{1,1/2}(S_T)})^d$ and $h\in (H_0^1(\Omega))^d$, let
$u_\varepsilon$ and $u_0$ be the weak solutions of the initial-Dirichlet problems
$(\emph{DP}_\varepsilon)$ and $(\emph{DP}_0)$, respectively. Then we have
\begin{equation}\label{pri:1.1}
\|u_\varepsilon - u_0\|_{L^2(\Omega_T)}
\leq C\varepsilon\ln(1/\varepsilon)
\Big\{\|F\|_{L^2(\Omega_T)}
+\|g\|_{H^{1,1/2}(S_T)} + \|h\|_{H^1(\Omega)}\Big\},
\end{equation}
where $C$ depends only on $\mu_1, \mu_2, d, T$ and $\Omega$.
\end{thm}

The symbol ${_0H^{1,1/2}(S_T)}$ denotes a Sobolev space of functions
with one spatial derivative and half of a time derivative in $L^2(S_T)$,
requiring its element to vanish on $\partial\Omega\times\{t=0\}$
(see \cite[pp.353]{RB}). This may be viewed as a
compatibility condition between the lateral data $g$ and the initial data $h$.

The convergence rate estimate $\eqref{pri:1.1}$ is almost sharp,
which may be interpreted as an operator error estimate sometimes.
Compared to the recent result
obtained in \cite[Theorem 1.1]{GZS}, the estimate $\eqref{pri:1.1}$ owns two
conspicuous advantages. One is that the result is established for a
Lipschitz cylinder, the other is that
the estimate is fully based upon the given data, especially permitting
a lower regularity assumption on the lateral data $g$.
On the other hand, the estimate $\eqref{pri:1.1}$ is quite similar to that developed
for elliptic systems with Dirichlet or Neumann boundary conditions in \cite[Theorems 1.1,1.2]{QX2},
which seems to be reasonable if we think of the elliptic system as the stable case of the parabolic one. However, handling parabolic systems proved to be much complicated,
and we have to establish some new weighted-type estimates with a parabolic distance function being
a weight, such as Lemmas $\ref{lemma:2.5},\ref{lemma:2.1}$ and $\ref{lemma:2.6}$.
Meanwhile, some new techniques designed for the so-called time-layer type estimates have also been developed in Lemmas $\ref{lemma:3.2}$ and $\ref{lemma:4.1}$.
Such the estimates similar to $\eqref{pri:1.1}$ have been intensively studied during the past ten years for elliptic operators, parabolic equations and Stokes systems in periodic homogenization theory,
and without attempting to be exhaustive we refer the reader to
\cite{SACS,SZ,MAFHL,ABJLGP,BMSHSTA,GZS,GZS1,GG1,GG2,G,SZW2,MYMTS,SZW12,SZW20,TS2,TS,QX2,QX3,QXS1,ZVVPSE} and references therein for more results.
We end this paragraph by mention that
the source of the main ideas directly come from the references \cite{GZS,QX2}, originally from C. Kenig, F. Lin, Z. Shen and T. Suslina in \cite{SZW2,SZW12,TS}.

So, it is instructive to sketch the main procedures before giving the detailed proof.
Inspired from \cite[Theorem 2.2]{GZS},
we construct the approximating of $u_\varepsilon$ as follows
\begin{equation}
w_\varepsilon = u_\varepsilon - u_0
-\varepsilon\chi_j(x/\varepsilon,t/\varepsilon^2)S_\varepsilon
K_\varepsilon(\Psi_{[4\varepsilon^2,2\varepsilon]}\nabla_j u_0)
-\varepsilon^2 E_{l(d+1)j}(x/\varepsilon,t/\varepsilon^2)
\nabla_l S_\varepsilon
K_\varepsilon(\Psi_{[4\varepsilon^2,2\varepsilon]}\nabla_j u_0),
\end{equation}
where $\chi_j$ and $E_{l(d+1)j}$ with $1\leq j,l\leq d$, are known as correctors and dual correctors in Subsection $\ref{subsection2.2}$, and they had been well studied in \cite{GZ,GZS}.
Here $S_\varepsilon$ and $K_\varepsilon$ is the smoothing operators given
in Definition $\ref{def:1}$, as successors of the so-called Steklov smoothing operator
originally applied to the homogenization problems by V.V Zhikov and S.E. Pastukhova in \cite{ZVVPSE}.
The notation $\Psi_{[4\varepsilon^2,2\varepsilon]}$ is a cut-off function whose description will be given later. Then, we can find an equation that $w_\varepsilon$ satisfies
(see Lemma $\ref{lemma:3.4}$), and this is the starting point of the proof of Theorem $\ref{thm:1.1}$.
For ease of statement, it is fine to assume
$\|F\|_{L^2(\Omega_T)} + \|g\|_{H^{1,1/2}(S_T)}+\|h\|_{H^1(\Omega)} = 1$
by the linearity of $(\text{DP}_\varepsilon)$ and $(\text{DP}_0)$. Roughly speaking, the proof will be reduced to two
steps. The first one is based upon the energy inequality, which shows
\begin{equation}\label{pri:1.3}
\big\|\nabla w_\varepsilon\big\|_{L^2(\Omega_T)} = O(\varepsilon^{1/2}).
\end{equation}
The second one relies on duality methods, by which we may establish
\begin{equation}\label{pri:1.4}
\big\|w_\varepsilon\big\|_{L^2(\Omega_T)} = O(\varepsilon\ln(1/\varepsilon)).
\end{equation}

At a glimpse, the methods look similar to the aforementioned ones as in \cite{GZS}.
However, the calculations related to nonsmooth cylinders turn to be
much involved. So, some related tricks are necessary to be explained. Before proceeding further, it is better to introduce some geometric notation to simplify the later statements,
and they will be shown in Figures $\ref{pic:1.1}$ and $\ref{pic:1.2}$ to make them be apprehended at a glance.
\begin{itemize}
  \item $S_r = \big\{x\in\Omega:\text{dist}(x,\partial\Omega)= r\big\}$
  denotes the level set of $\Omega$.
  \item $r_0$ is the diameter of $\Omega$, and $r_{00}=\max\{r>0:B(x,r)\subset\Omega,
  \forall x\in S_r\}$ denotes the the internal diameter, where $B(x,r)$ is an open ball
   in $\mathbb{R}^d$ with center $x$ and radius $r>0$, and we call $c_0 = r_{00}/10$ the layer constant.
  \item $P(X,r) = \big\{Y\in\mathbb{R}^{d+1}: \text{d}(X,Y)<r\big\}$ is known as a parabolic cube with the center $X$ and radius $r>0$, where the capital letters $X=(x,t)$ and $Y=(y,s)$
      are used to represent some points in the parabolic cylinder $\Omega_T$, and
      $\text{d}(X,Y)=|x-y|+|t-s|^{1/2}$ is the so-called parabolic distance.
\includegraphics[width=6in]{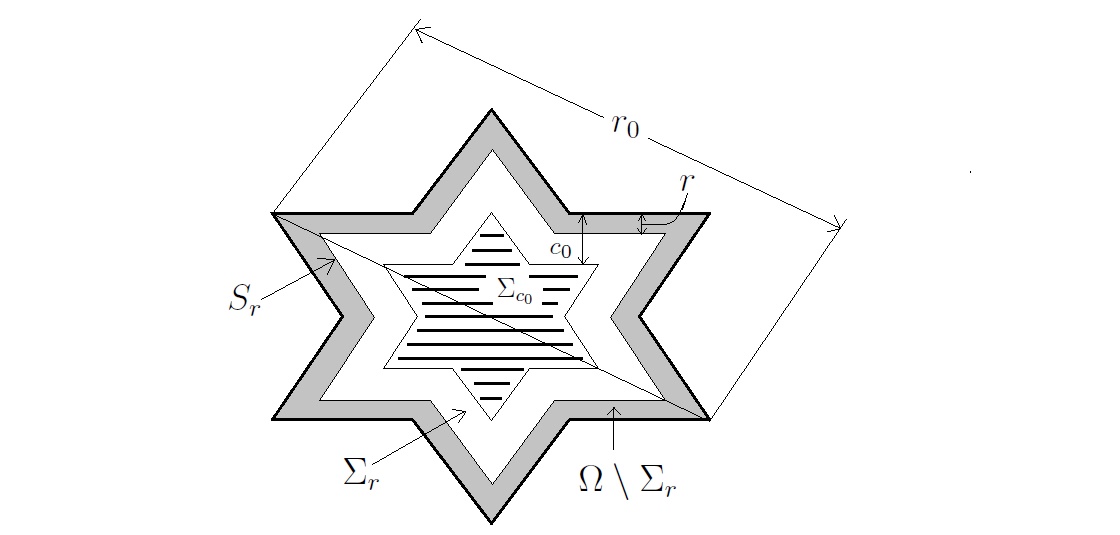}
\makeatletter\def\@captype{figure}\makeatother
\caption{aerial view of $\Omega_T$}\label{pic:1.1}
\vspace{0.5cm}
\includegraphics[width=6in]{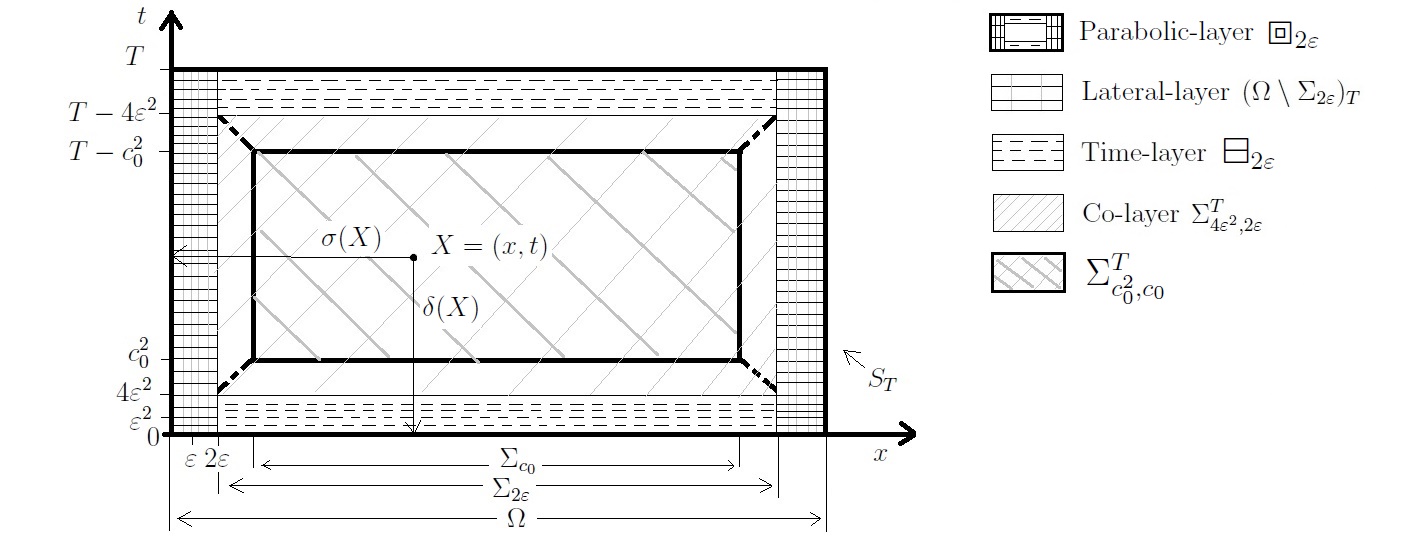}
\makeatletter\def\@captype{figure}\makeatother
\caption{sectional view of $\Omega_T$}\label{pic:1.2}
  \item $\Sigma_{r^2,r}^T = \Sigma_{r}\times [r^2,T-r^2]$,
      where $\Sigma_{r}
      = \big\{x\in\Omega:\text{dist}(x,\partial\Omega) \geq r\big\}$,
      which is interpreted as being the parabolic co-layer of $\Omega_T$, and
      $\kappa\Sigma_{r^2,r}^T = \Sigma_{\kappa r^2,\kappa r}^T$ is regarded as an expansion
      of $\Sigma_{r^2,r}^T$ with the factor $\kappa\in(0,1)$,
      or as a shrink with $\kappa > 1$.
  \item $\boxbox_{\kappa r} = \Omega_T\setminus \kappa\Sigma_{r^2,r}^T$ is
  known as a parabolic-layer of $\Omega_T$, which is composed of two parts:
  \begin{description}
    \item[lateral-layer] $(\Omega\setminus\Sigma_{\kappa r})_T
  = \Omega\setminus\Sigma_{\kappa r}\times (0,T]$,
    \item[time-layer] $\boxminus_{\kappa r}
    =  \boxbox_{\kappa r}\setminus(\Omega\setminus\Sigma_{\kappa r})_T
    = \Sigma_{\kappa r}\times(0,\kappa r^2]\cup
    \Sigma_{\kappa r}\times(T-\kappa r^2,T]$.
  \end{description}
 \item For any
$X=(x,t)\in \Omega_T$, the distance between $X$ and $\partial\Omega_T$ is denoted by
\begin{equation}\label{eq:2.4}
\delta(x,t)=\delta(X) = \text{d}(X,\partial\Omega_T) = \min\Big\{\text{dist}(x,\partial\Omega),t^{\frac{1}{2}},
(T-t)^{\frac{1}{2}}\Big\}.
\end{equation}
and the one between $X$ and $S_T$ is written by
\begin{equation}\label{eq:2.5}
 \sigma(x,t)=\sigma(X) = \text{d}(X,S_T) = \text{dist}(x,\partial\Omega).
\end{equation}
\item Let $\Psi_{[r^2,r]}\in C_0^{2,1}(\Omega_T)$ be a cut-off function such that
\begin{equation}
\Psi_{[r^2,r]} = \left\{
                 \begin{array}{ll}
                   1, & \text{in}~\Sigma_{2r^2,2r}^T, \\
                   0, & \text{on}~\Omega_T\setminus\Sigma_{r^2,r}^T,
                 \end{array}
               \right.
\quad\text{and}\quad
\left\{
  \begin{array}{l}
  |\nabla\Psi_{[r^2,r]}|\leq C/r, \\
  |\partial_t \Psi_{[r^2,r]}|
+ |\nabla^2 \Psi_{[r^2,r]}|\leq C/r^2 ,
  \end{array}
\right.
\end{equation}
where $C_0^{2,1}(\Omega_T)$ is the set of all continuous functions with compacted support in
$\Omega_T$ having continuous derivatives $\nabla_i u, \nabla^2_{ij}u$ and $\partial_t u$ (see Subsection $\ref{subsection:2.0}$).
\end{itemize}

In fact, to estimate $\eqref{pri:1.3}$, it suffices to show
\begin{equation}\label{pri:1.5}
 \big\|\nabla u_0\big\|_{L^2((\Omega\setminus\Sigma_{4\varepsilon})_T)}
+  \sup_{\varepsilon^2<t<T}\Big(
\int_{t-\varepsilon^2}^t\int_{\Sigma_{2\varepsilon}}|\nabla u_0|^2 dxdt\Big)^{1/2}
= O(\varepsilon^{1/2}),
\end{equation}
and
\begin{equation}\label{pri:1.6}
\max\Big\{\|\nabla^2 u_0\|_{L^2(\Sigma_{4\varepsilon^2,2\varepsilon}^T)},
~\|\partial_t u_0\|_{L^2(\Sigma_{4\varepsilon^2,2\varepsilon}^T)}\Big\}
= O(\varepsilon^{-1/2}).
\end{equation}

According to the region of the above integrals (see Figure $\ref{pic:1.2}$),
the estimate $\eqref{pri:1.5}$ may be regarded as ``a lateral-layer type estimate $+$ a time-layer type one'', while we think of $\eqref{pri:1.6}$ as a co-layer type estimate, where co-layer means
the complementary layer for short.
All these estimates can not be directly derived,
since $g\not=0$ on $S_T$ and there is no hope of transferring it to the initial or
source term due to the less regularity assumption on $\Omega_T$.
However, owing to the linearity of $(\text{DP}_0)$,
it may be divided into a homogeneous part with nonzero lateral data and a nonhomogeneous part with zero one. By an extension technique,
the solution of related nonhomogeneous system (denoted by $v$ for example) holds some regularity estimates for $\partial_t v$ and $\nabla^2 v$ in a larger smooth cylinder.
Note that the co-area formula is still valid for space variables, by which the
lateral-layer and time-layer type estimates for $v$ could be reduced to bound
the following quantity
\begin{equation*}
  \int_{S_r} |\nabla v| dS_r
\end{equation*}
uniformly for $r\in[0,c_0]$, and this will be done through the trace theorem.
Actually, the width of the layer is merely $\varepsilon$ or $\varepsilon^2$, which
is one of places where a half order of convergence rates is born.
In addition, one
may even derive a better co-layer estimate for $v$, which is
\begin{equation*}
\max\Big\{\|\nabla^2 v\|_{L^2(\Sigma_{4\varepsilon^2,2\varepsilon}^T)},
~\|\partial_t v\|_{L^2(\Sigma_{4\varepsilon^2,2\varepsilon}^T)}\Big\}
= O(1).
\end{equation*}
We mention that the above approaches have already been developed by Z. Shen in \cite{SZW12}
regarding to an elliptic system of elasticity.

The hard part is the homogeneous one with nonzero lateral data, whose solution
is represented by $\bar{w}$ for the occasion.
The existence of $\bar{w}$ is a long but interesting story which had been brilliantly accomplished by Z. Shen in \cite{SZW22}, and it guarantees that the previous detaching works legally.
Also, we strongly recommend R. Brown's work \cite{RB} for this field.
Compared to the elliptic cases, the main difficulty will soon emerge in the time-layer type estimate
\begin{equation*}
\sup_{\varepsilon^2<t<T}\int_{t-\varepsilon^2}^t\int_{\Sigma_{2\varepsilon}}|\nabla u_0|^2 dxdt
= O(\varepsilon),
\end{equation*}
which may promptly be put down to the following estimates
\begin{equation*}
\int_{t-\varepsilon^2}^t\int_{\Sigma_{2\varepsilon}\setminus\Sigma_{4\varepsilon}}|u_0|^2 dxds
= O(\varepsilon^3),
\quad
\int_{t-\varepsilon^2}^t\int_{\Sigma_{2\varepsilon}}
|u_0|^2 dxds
= O(\varepsilon^2)
\quad
\text{and}\quad
\int_{t-\varepsilon^2}^t\int_{\Sigma_{2\varepsilon}}
|\partial_t u_0|^2 dxds
= O(1),
\end{equation*}
by using a Caccioppoli's inequality in Lemma $\ref{lemma:3.3}$.
The crucial ingredient is that in virtue of nontangential maximal functions,
we can control the behavior of $\bar{w}$ near $S_T$ in a time layer
(see Figure $\ref{pic:1.3}$). Here we define the maximal function of $\bar{w}$ as
\begin{equation*}
 (\bar{w})^*(x,t) = \sup_{(y,s)\in \Upsilon(x,t)}\big|\bar{w}(y,s)\big|
\end{equation*}
where $\Upsilon(x,t)$ is the parabolic nontangential approach region defined for
$(x,t)\in S_T$ by
\begin{equation*}
 \Upsilon(x,t) = \big\{(y,s):|y-x|+|s-t|^{1/2}<(1+N)\text{dist}(y,\partial\Omega)\big\}
\cap\Omega_T.
\end{equation*}
The parameter $N$ is an arbitrary positive number which will be fixed throughout this paper.
\begin{center}
\includegraphics[width=5.5in]{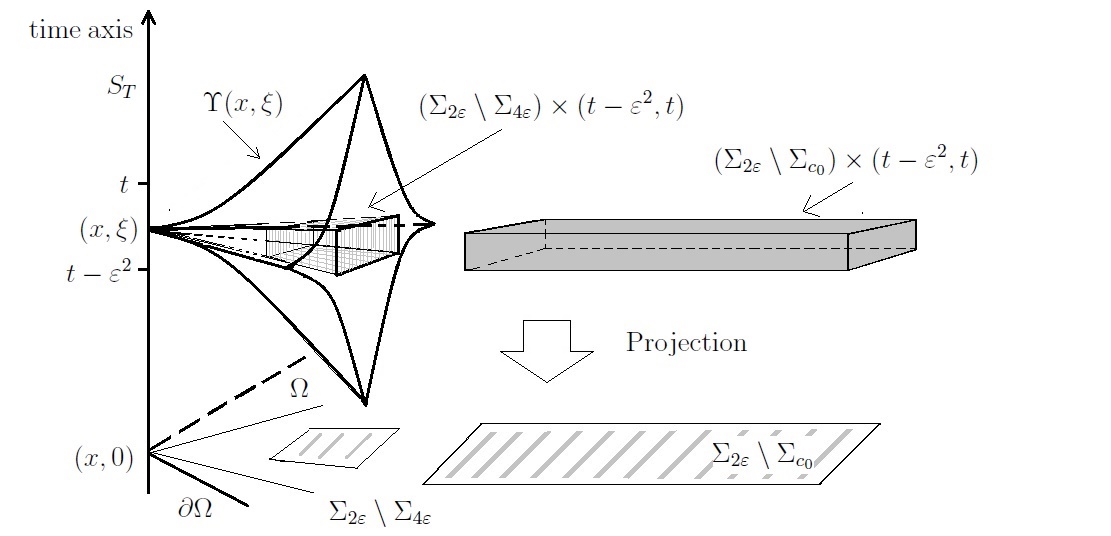}
\makeatletter\def\@captype{figure}\makeatother
\caption{parabolic nontangential approach region and time-layer regions}\label{pic:1.3}
\end{center}

Precisely, a subtle fact observed in Figure $\ref{pic:1.3}$ will be frequently used in the
later sections, which is
\begin{equation}\label{f:1.1}
\int_{t-\varepsilon^2}^t\int_{\Sigma_{2\varepsilon}\setminus\Sigma_{4\varepsilon}}|\bar{w}|^2 dxds
\leq C\varepsilon^3\int_{\partial\Omega}|(\bar{w})^*(\cdot,\xi)|^2dx
\end{equation}
and
\begin{equation}\label{f:1.2}
\int_{t-\varepsilon^2}^t\int_{\Sigma_{2\varepsilon}\setminus\Sigma_{c_0}}|\nabla\bar{w}|^2 dxds
\leq C\varepsilon^2\int_{\partial\Omega}|(\nabla\bar{w})^*(\cdot,\xi)|^2dx
\end{equation}
where $\xi\in(t-\varepsilon^2,t)$, and $C$ is independent of $\xi$.
Then, integrating both sides of the above inequalities with respect to $\xi$ from $0$ to $T$,
the left-hand sides of $\eqref{f:1.1}$ and $\eqref{f:1.2}$ will be controlled by the quantities
$\|(\bar{w})^*\|_{L^2(S_T)}$ and $\|(\nabla\bar{w})^*\|_{L^2(S_T)}$, which will further be
determined by the given lateral data $g$ (see \cite[Theorem 4.2.1]{SZW22}). Here
we always divide $\Sigma_{r}$ into $\Sigma_{r}\setminus\Sigma_{c_0}$ and $\Sigma_{c_0}$, where
the region $\Sigma_{c_0}$ will be good part for related calculations in general.

Next, we will show some important observations on the co-layer type estimates $\eqref{pri:1.6}$. Again, we only focus ourselves on the estimate of
\begin{equation}\label{f:1.3}
\bigg(\int_{4\varepsilon^2}^{T-4\varepsilon^2}
\int_{\Sigma_{2\varepsilon}}|\nabla^2 \bar{w}|^2 dxdt\bigg)^{1/2} = O(\varepsilon^{-1/2}),
\end{equation}
whereas it is not hard to verify
$\|\nabla^2 v\|_{L^2(\Omega_T)}= O(1)$. To do so, we consider the following pointwise estimate
\begin{equation*}
\big|\nabla^2 \bar{w}(X)\big|^2 \leq
\frac{C}{[\sigma(X)]^2}\dashint_{P(X,\sigma(X)/4)\cap\Omega_T}|\nabla \bar{w}(Y)|^2 dY
\end{equation*}
for any $X=(x,t)\in (\Sigma_{2\varepsilon}\setminus\Sigma_{c_0})\times(4\varepsilon^2,T-4\varepsilon^2)$,
which may be found in \cite[pp.1148-1149]{SchW}. Since there holds the following relationship
between a parabolic ball and a parabolic nontangential approach region:
\begin{equation*}
 P(X,\sigma(X)/4)\cap\Omega_T \subsetneqq  \Upsilon(x^\prime,t),
\end{equation*}
where $x^\prime\in\partial\Omega$ is the point such that
$|x^\prime-x|=\sigma(X)$. Hence we have the following estimate
\begin{equation}\label{f:1.4}
\int_{4\varepsilon^2}^{T-4\varepsilon}
\int_{\Sigma_{2\varepsilon}\setminus\Sigma{c_0}}
\frac{1}{[\sigma(X)]^2}
\dashint_{P(Y,\delta(P)/4)}|\nabla\bar{w}|^2 dY dX
\leq C\int_0^T\int_{\partial\Omega}|(\nabla\bar{w})^*(\cdot,t)|^2 dSdt
\int_{2\varepsilon}^{c_0}\frac{dr}{r^2}.
\end{equation}
This together with the parabolic nontangential maximal estimate
\cite[Theorem 4.2.1]{SZW22} leads to the desired estimate $\eqref{f:1.3}$.
Consequently, the main procedures in the proof of $\eqref{pri:1.3}$ have been introduced to the reader. We must mention that such the aforemention techniques have already been in
Z. Shen's recent work \cite{SZW12} for elliptic cases.

Innovations originally come from managing to improve the estimate $\eqref{f:1.4}$.
It is natural to think of the distance function $\delta$ as a weight to
increase some integrability in the right-hand side of $\eqref{f:1.4}$, as a result of the fact that
$\delta/\sigma\leq 1$. Although this weight may lead to some better estimates, such as
\begin{equation*}
\int_{4\varepsilon^2}^{T-4\varepsilon^2}
\int_{\Sigma_{2\varepsilon}}|\nabla^2 \bar{w}|^2 \delta(x,t)dxdt
 = O(\ln(c_0/\varepsilon))
\quad\text{and}\quad
\int_{\boxbox_{2\varepsilon}}|\nabla u_0|^2 \delta(x,t)dxdt
= O(\varepsilon),
\end{equation*}
it also arises other intractable problems. One of them is to bound the following quantity
\begin{equation*}
\int_{4\varepsilon^2}^{T-4\varepsilon^2}\int_{\Sigma_{2\varepsilon}}
|\nabla u_0|^2 [\delta(x,t)]^{-1}dxdt
\end{equation*}
by $O(\log_2(1/\varepsilon))$,
which urges us to find a weighted Caccioppoli's inequality in a time-layer region
(see Lemma $\ref{lemma:4.1}$). Beyond this, we require that
the weight functions $\delta^{\pm1}$ can pass through the smoothing operators
$S_\varepsilon$ and $K_\varepsilon$ freely,
which has been summarized in Lemmas $\ref{lemma:2.1}$, $\ref{lemma:2.6}$ and
$\ref{lemma:2.5}$. As far as the authors have known, they are new established in this paper.
Therefore, in technical point of view, the order of $\varepsilon$ in the estimate $\eqref{pri:1.4}$ will come from two sources.
One is straightforwardly from the duality method as J. Geng and Z. Shen did in \cite{GZS}, the
other is actually attributed to the weight function $\delta$.
Since the duality method has been well illustrated in \cite{GZS,QX2}, we do not
repeat here.

Up to now, we have shown the main tricks related to the estimates $\eqref{pri:1.3}$ and
$\eqref{pri:1.4}$, and so to $\eqref{pri:1.1}$.
We mention that the estimate $\eqref{pri:1.3}$ may play a fundamental part in
further quantitative estimates, such as uniform H\"older estimates and $W^{1,p}$
estimates with $1<p\leq\infty$. This is an active field and
some of them have been established
through compactness methods (see \cite{GZS1}).
We also highly recommend \cite{SZW12} for recent developments in periodic homogenization theory, as well as \cite{SACS,SZ} for a non-periodic setting.

We end this section by two remarks.
\begin{remark}
\emph{We emphasis that the expression
$S_\varepsilon K_\varepsilon(\Psi_{[4\varepsilon^2,2\varepsilon]}\nabla_j u_0)$
in $w_\varepsilon$ can not be
replaced by $S_\varepsilon(\Psi_{[4\varepsilon^2,2\varepsilon]}\nabla u_0)$ in \cite{GZS},
even though we are able to establish the weighted-type estimates for $\nabla u_0$
in $\Sigma_{\varepsilon^2,\varepsilon}^T$ (see Corollary $\ref{cor:2.1}$).
In concrete calculations, $K_\varepsilon$ will serve as a role in eliminating  one spatial derivative, by reason of that there is no good way of bounding derivatives of third order.
Note that there naturally hold global regularity estimates for
$\|\partial_t u_0\|_{L^2(\Omega_T)}$ and $\|\nabla^2 u_0\|_{L^2(\Omega_T)}$
provided $\partial\Omega\in C^{1,1}$ as in \cite{GZS,MYMTS}.
By contrast, for a Lipschitz cylinder, we have to rely on some subtle arguments mentioned before.}
\end{remark}

\begin{remark}
\emph{We point out that
the arguments developed in this paper can be extended to other
initial-boundary problems, and to the parabolic operators with lower order terms.
The crucial estimates actually relies on the symmetry assumption on $\mathcal{L}_\varepsilon$,
while
the methods for getting rid of it have been studied in recent work \cite{GX},
which will possibly illuminate the sharp uniform estimate with regard to smooth cylinders.}
\end{remark}

The paper is organized as follows. Secton $\ref{section:2}$ is mainly to show the weighted-type estimates
for the smoothing operator at scale $\varepsilon$ in terms of t-anisotropic Sobolev spaces.
Section $\ref{section:3}$ is designed to establish the estimate $\eqref{pri:1.3}$ and the proof of Theorem
$\ref{thm:1.1}$ will be presented in Section $\ref{section:4}$.

\section{Preliminaries}\label{section:2}

\subsection{Notation}\label{subsection:2.0}

We first introduce notation for derivatives.
\begin{enumerate}
  \item $\nabla u = (\nabla_1,\cdots,\nabla_d)$ is the gradient of $u$ with respect to spatial variable, where $\nabla_i u = \partial u/\partial x_i$ denotes the $i^{\text{th}}$ spatial derivative of $u$. $\nabla^2 u = (\nabla^2_{ij}u)_{d\times d}$ denotes the Hessian matrix
      of $u$, where $\nabla^2_{ij} u = \frac{\partial^2 u}{\partial x_i\partial x_j}$.
  \item $\partial_t u = \partial u/\partial t$ briefly
      represents the derivative of $u$ with respect to the time variable.
\end{enumerate}

The following notation represents function spaces and weighted-type norms.
\begin{enumerate}
  \item The Sobolev space $W_2^{1,1}(\Omega_T) = H^1(\Omega_T)$ is the Banach space consisting
  of the elements of $L^2(\Omega_T)$ having weak derivatives of the forms $\partial_t u$ and
  $\nabla_i u$ with $i=1,\cdots,d$.
  The space $W_2^{1,0}(\Omega_T) = L^2(0,T;H^1(\Omega))$ is a proper one for weak solutions, and
  $W_{2,\text{loc}}^{2,1}(\Omega_T)$ presents the function space
  $W_{2}^{2,1}(\Omega_T)$ in a local sense. These function spaces can be found in \cite{LCE,OAL}.
  \item The weighted-type norms are defined by
\begin{equation}\label{eq:2.6}
 \big\|f\big\|_{L^2(\Sigma_{r^2,r}^T;\omega)}
= \int_{r^2}^{T-r^2}\int_{\Sigma_{r}}|f(x,t)|^2\omega(x,t)dxdt,
\end{equation}
where the weight function $\omega$ may be chosen from $\delta$ and $\delta^{-1}$.
\end{enumerate}

\subsection{$L^2$ theory}

\begin{thm}\label{thm:2.1}
Suppose that $A$ satisfies $\eqref{c:1}$. Let $F\in L^2(0,T; (H^{-1}(\Omega))^d)$ and
$g\in (H^{\frac{1}{2},\frac{1}{4}}(S_T))^d$ with $h\in (L^2(\Omega))^d$. Then there exists
a unique weak solution $u_\varepsilon\in L^2(0,T;(H^1(\Omega))^d)\cap L^\infty(0,T;(L^2(\Omega))^d)$
to $(\emph{DP}_\varepsilon)$ satisfying the uniform energy estimate
\begin{equation}\label{pri:2.18}
\sup_{0\leq t\leq T}\|u_\varepsilon\|_{L^2(\Omega)}
+ \big\|\nabla u_\varepsilon\big\|_{L^2(\Omega_T)}
\leq C\Big\{\big\|F\big\|_{L^2(0,T;H^{-1}(\Omega))}+\big\|h\big\|_{L^2(\Omega)}
+ \big\|g\big\|_{H^{1/2,1/4}(S_T)}\Big\},
\end{equation}
where $C$ depends on $\mu_1,\mu_2,d,T$ and $\Omega$.
\end{thm}

\begin{proof}
We first prove the existence of weak solution $u_\varepsilon$.
Following the notation from \cite{MC}, we define
$\tilde{H}^{1,\frac{1}{2}}(\Omega_T) =
\{u\in H^{1,\frac{1}{2}}(\Omega\times(-\infty,T]):u=0~\text{for}~t<0\}$.
On account of \cite[Theorem 2.9]{MC} and \cite[Theorem 3.4]{MC}, there
exists $G^\alpha\in \tilde{H}^{1,\frac{1}{2}}(\Omega_T)$ such that
$\|G^\alpha\|_{\tilde{H}^{1,1/2}(\Omega_T)}\leq C\|g^\alpha\|_{H^{\frac{1}{2},\frac{1}{4}}(S_T)}$, and
\begin{equation*}
 \partial_t G^\alpha - \Delta G^\alpha = 0 \quad\text{in}~\Omega_T,
\qquad G^\alpha= g^\alpha \quad\text{on}~S_T,\quad \text{and} \quad G^\alpha = 0 \quad\text{on}~\Omega\times\{t=0\}.
\end{equation*}

Let $z_\varepsilon = u_\varepsilon - G$, where $G=(G^1,\cdots, G^d)$, and we have
\begin{equation}\label{f:2.3}
\left\{\begin{aligned}
\big(\partial_t + \mathcal{L}_\varepsilon\big)(z_\varepsilon)
&= F + \text{div}(\nabla G - A_\varepsilon\nabla G) &\quad&\text{in}~\Omega_T, \\
z_\varepsilon
&= 0 &\quad&\text{on}~S_T,\\
z_\varepsilon
&= h &\quad&\text{on}~\Omega\times\{t=0\},
\end{aligned}\right.
\end{equation}
where $A_\varepsilon(x,t) = A(x/\varepsilon,t/\varepsilon^2)$, and
we use the fact that $\partial_t G = \Delta G$ in $\Omega_T$.
Then the source term in $\eqref{f:2.3}$ belongs to $L^2(0,T;(H^{-1}(\Omega))^d)$, bounded by
$\|F\|_{L^2(0,T;H^{-1}(\Omega))}+\|g\|_{H^{\frac{1}{2},\frac{1}{4}}(S_T)}$.
Therefore, the existence of weak solution
$u_\varepsilon$ to $(\text{DP}_\varepsilon)$ is reduced to
finding a weak solution $z_\varepsilon$ for
$\eqref{f:2.3}$, and it has been done by \cite[Theorem 3, pp.378]{LCE}.
The uniqueness of the weak solution $u_\varepsilon$ may be
easily derived by the energy inequality $\eqref{pri:2.18}$, and this is what we do in next step.

For the equation $\eqref{f:2.3}$, it follows from \cite[Lemma 2.1, Chapter III]{OAL} that
\begin{equation*}
\big\|\nabla z_\varepsilon\big\|_{L^2(\Omega_T)}
\leq C\Big\{\big\|F\big\|_{L^2(0,T;H^{-1}(\Omega))}
+\big\|g\big\|_{H^{1/2,1/4}(S_T)}
+\big\|h\big\|_{L^2(\Omega)}\Big\},
\end{equation*}
where we need to employ the elasticity condition $\eqref{c:1}$ coupled with the first Korn inequality (see \cite[pp.371]{VSO}),
and this implies
\begin{equation}\label{pri:2.3}
 \|\nabla u_\varepsilon\|_{L^2(\Omega_T)}
\leq C\Big\{\big\|F\big\|_{L^2(0,T;H^{-1}(\Omega))}
+\big\|g\big\|_{H^{1/2,1/4}(S_T)}
+\big\|h\big\|_{L^2(\Omega)}\Big\}.
\end{equation}
From this estimate, we know that $\partial_t u_\varepsilon\in L^2(0,T;(H^{-1}(\Omega))^d)$, and
this together with \cite[Theorem 3, pp.303]{LCE} and the estimate $\eqref{pri:2.3}$ leads to
\begin{equation*}
\begin{aligned}
\sup_{0\leq t\leq T}\|u_\varepsilon\|_{L^2(\Omega)}
&\leq C\Big\{\|\nabla u_\varepsilon\|_{L^2(\Omega_T)}
+\|\partial_tu_\varepsilon\|_{L^2(0,T;H^{-1}(\Omega))}\Big\}\\
&\leq C\Big\{\big\|F\big\|_{L^2(0,T;H^{-1}(\Omega))}
+\big\|g\big\|_{H^{1/2,1/4}(S_T)}
+\big\|h\big\|_{L^2(\Omega)}\Big\},
\end{aligned}
\end{equation*}
where we use the equation
$\partial_t u_\varepsilon = F - \mathcal{L}_\varepsilon(u_\varepsilon)$ and $\eqref{pri:2.3}$ in
the second step. We have completed the proof.
\end{proof}

\begin{flushleft}
\textbf{Proof of Theorem $\ref{thm:1.2}$.}
The proof is quite similar to that given for \cite[Theorem 3.6]{GZS1} in the case of $g=0$, which
follows from the estimate $\eqref{pri:2.18}$ and Tartar's test function methods
(it actually does not involve any boundary condition or initial data). Thus, without a proof, we
straightforwardly show the following facts:
\end{flushleft}
\begin{equation}\label{f:2.4}
\begin{array}{c}
u_\varepsilon \rightharpoonup u_0 \quad \text{weakly~in}~L^2(0,T;(H^1(\Omega))^d), \\
\partial_tu_\varepsilon \rightharpoonup \partial_t u_0 \quad
\text{weakly~in}~L^2(0,T;(H^{-1}(\Omega))^d), \\
A_\varepsilon \nabla u_\varepsilon
\rightharpoonup \widehat{A}u_0 \quad \text{weakly~in}~(L^2(\Omega_T))^d,
\end{array}
\end{equation}
where $u_0$ satisfies $\partial_t u_0 -\text{div}(\widehat{A}\nabla u_0) = F$ in $\Omega_T$.
Then we plan to verify $u_0 = g$ on $S_T$ in a trace sense, and $u_0(x,0) = h(x)$ for a.e.
$x\in \Omega$. It follows from $\eqref{f:2.4}$ together with the Aubin-Lions-Simon theorem that
\begin{equation}\label{f:2.10}
\begin{aligned}
u_\varepsilon \rightarrow u_0 \quad &\text{strongly~in}~(L^2(\Omega_T))^d.
\end{aligned}
\end{equation}
Also, in view of \cite[Theorem 3, pp.303]{LCE},
we have $u_\varepsilon,u_0\in C([0,T];(L^2(\Omega))^d)$. Our now task is to verify $u_0 = h$
on $\Omega\times\{t=0\}$. Let $\psi=\varphi_x\varphi_t$ be a test function,
where $\varphi_x\in C_0^1(\Omega)$ and $\varphi_t\in C^1([0,T])$ satisfying
$\varphi_t(0) = 1$ and $\varphi(T) = 0$.
By reusing $\eqref{f:2.4}$, we have
\begin{equation*}
\begin{aligned}
 -\int_{\Omega_T}\big(A_\varepsilon\nabla u_\varepsilon - \widehat{A}\nabla u_0\big)\cdot\nabla\psi dxdt
&= \int_0^T \big<\partial_t(u_\varepsilon-u_0),\psi\big>dt \\
& = - \int_{\Omega_T}\big(u_\varepsilon - u_0\big)\partial_t\psi dt
+ \big<u_\varepsilon-u_0,\psi\big>\bigg|_{t=0}^T.
\end{aligned}
\end{equation*}
This gives
\begin{equation*}
\big<h-u_0,\varphi_x\big> =
\int_{\Omega_T}\big(u_\varepsilon - u_0\big)\partial_t\psi dt
-\int_{\Omega_T}\big(A_\varepsilon\nabla u_\varepsilon - \widehat{A}\nabla u_0\big)\cdot\nabla\psi dxdt \to 0,\quad\text{as}~\varepsilon\to0,
\end{equation*}
where we employ $\eqref{f:2.4}$ and $\eqref{f:2.10}$ in the last step.
The desired result directly follows from the arbitrary choosing $\varphi_x\in C_0^1(\Omega)$.
The next step is to show $u_0 = g$ on $S_T$. Owing to $\eqref{f:2.10}$ and $\eqref{f:2.4}$, we can  derive $u_\varepsilon\to u_0$ strongly in $L^2(S_T)$, just by noting
\begin{equation*}
\int_{S_T}|u_\varepsilon - u_0|^2 dxdt
\leq C\Big\{\int_{\Omega_T} |u_\varepsilon-u_0|^2 dxdt
+\|u_\varepsilon-u_0\|_{L^2(\Omega_T)}\|\nabla u_\varepsilon\|_{L^2(\Omega_T)}
\Big\}.
\end{equation*}
This implies $u_0 = g$ on $S_T$ in the trace sense, and we end the proof here.
\qed

\subsection{Correctors and its properties}\label{subsection2.2}
Let $Y=(0,1]^{d+1}\simeq \mathbb{R}^{d+1}/\mathbb{Z}^{d+1}$. Define the correctors
$\chi_j^\beta(y,\tau) = (\chi_j^{\gamma\beta}(y,\tau))$ associated with the
parabolic system $(\text{PD}_\varepsilon)$ by the following cell problem:
\begin{equation}
\left\{\begin{aligned}
& \big(-\partial_\tau + \mathcal{L}_1\big)(\chi_j^\beta+ P_j^\beta) = 0 \quad \text{in}~Y,\\
& ~\chi_j^\beta(y,\tau) \text{is~1-periodic~in}~(y,\tau),\\
&~\dashint_Y\chi_j^\beta dyd\tau = 0, \quad \text{with}~j,\beta = 1,\cdots,d,
\end{aligned}\right.
\end{equation}
where $P_j^\beta(y) = y_je^\beta$, and $e^\beta = (0,\cdots,1,\cdots,0)$ with 1 in the
$\beta^{\text{th}}$ position.
Since there is no boundary term produced by taking integration by parts, it follows from energy inequality \cite[pp.139]{OAL} that
\begin{equation}
  \big\|\nabla\chi\big\|_{L^2(Y)} \leq C(\mu_1,\mu_2,d).
\end{equation}
By asymptotic expansion arguments the homogenized operator is given by
$\partial_t + \mathcal{L}_0 = \partial_t -\text{div}(\widehat{A}\nabla)$, where
$\widehat{A} = (\hat{a}_{ij}^{\alpha\beta})$ and
\begin{equation}\label{eq:2.7}
\hat{a}_{ij}^{\alpha\beta} = \int_Y \Big[a_{ij}^{\alpha\beta}(y,\tau)
+a_{ik}^{\alpha\gamma}\frac{\partial\chi_j^{\gamma\beta}}{\partial y_k}(y,\tau)\Big]dyd\tau
\end{equation}
(see \cite{ABJLGP,GZ,GZS}).

\begin{lemma}\label{lemma:2.3}
Let $1\leq j\leq d$ and $1\leq\alpha,\gamma\leq d$, and
\begin{equation}\label{eq:2.1}
b_{ij}^{\alpha\gamma}(y,\tau)
= \hat{a}_{ij}^{\alpha\gamma}
- a_{ij}^{\alpha\gamma}(y,\tau)
- a_{ik}^{\alpha\beta}(y,\tau)\frac{\partial\chi_j^{\beta\gamma}}{\partial y_k}(y,\tau),
\qquad b_{(d+1)j}^{\alpha\gamma}(y,\tau) = \chi_j^{\alpha\gamma}(y,\tau),
\end{equation}
where $y=x/\varepsilon$ and $\tau = t/\varepsilon^2$. Then the quantity $b_{ij}^{\alpha\gamma}$
with $i=1,\cdots,d+1$ satisfies
two properties:
\begin{equation}\label{eq:2.3}
\emph{(i)} \quad \dashint_{Y} b_{ij}^{\alpha\gamma}(y,\tau)dyd\tau = 0;
\qquad\text{and}\qquad
\emph{(ii)} \quad \sum_{i=1}^{d+1}\frac{\partial b_{ij}^{\alpha\gamma}}{\partial y_i} =0.
\qquad\text{and}\qquad
\end{equation}
Moreover, there exists $E_{kij}^{\alpha\gamma}\in H_{per}^1(Y)$ such that
\begin{equation}\label{eq:2.2}
b_{ij}^{\alpha\gamma} = \sum_{k=1}^{d+1}\frac{\partial}{\partial y_k}\big\{E_{kij}^{\alpha\gamma}\big\},
\qquad
E_{kij}^{\alpha\gamma} = - E_{ikj}^{\alpha\gamma},
\qquad\text{and}\qquad
\|E_{kij}^{\alpha\gamma}\|_{L^2(Y)}\leq C,
\end{equation}
where $C$ depends only on $\mu$ and $d$.
\end{lemma}

\begin{proof}
See \cite[Lemma 2.1]{GZS}.
\end{proof}

\subsection{Smoothing operator and its properties}

\begin{definition}\label{def:1}
\emph{Fix $\eta\in C_0^\infty(B(0,1/2))$ with $\int_{\mathbb{R}^d}\eta dx = 1$. Define a smoothing
operator associated with the spatial variable as
\begin{equation}
 K_\varepsilon(f)(x,t) = \int_{\mathbb{R}^d}\eta_\varepsilon(x-z)f(z,t)dz,
\end{equation}
where $\eta_\varepsilon(x)=\varepsilon^{-d}\eta(x/\varepsilon)$.
Let $\zeta\in C_0^\infty(P(0,1/2))$ satisfy $\int_{\mathbb{R}^{d+1}}\zeta dx = 1$.
Define a parabolic smoothing operator as
\begin{equation}
 S_\varepsilon(f)(x,t) = \int_{\mathbb{R}^{d+1}}\zeta_\varepsilon(x-z,t-s)f(z,s)dzds,
\end{equation}
where $\zeta_\varepsilon(x,t)=\varepsilon^{-d-2}\zeta(x/\varepsilon,t/\varepsilon^2)$.}
\end{definition}

\begin{lemma}
Let $f\in L^p(\mathbb{R}^{d+1})$ with $1\leq p<\infty$, then for any $\varpi\in L^p_{per}(Y)$ we have
\begin{equation}\label{pri:2.1}
\begin{aligned}
\big\|\varpi(\cdot/\varepsilon,\cdot/\varepsilon^2)S_\varepsilon(f)\big\|_{L^p(\mathbb{R}^{d+1})}
&\leq C\|\varpi\|_{L^p(Y)}\|f\|_{L^p(\mathbb{R}^{d+1})}, \\
\big\|\varpi(\cdot/\varepsilon,\cdot/\varepsilon^2)
\nabla S_\varepsilon(f)\big\|_{L^p(\mathbb{R}^{d+1})}
&\leq C\varepsilon^{-1}\|\varpi\|_{L^p(Y)}\|f\|_{L^p(\mathbb{R}^{d+1})}, \\
\big\|\varpi(\cdot/\varepsilon,\cdot/\varepsilon^2)\partial_tS_\varepsilon(f)\big\|_{L^p(\mathbb{R}^{d+1})}
&\leq C\varepsilon^{-2}\|\varpi\|_{L^p(Y)}\|f\|_{L^p(\mathbb{R}^{d+1})},
\end{aligned}
\end{equation}
where $C$ depends on $\zeta$ and $d$.
\end{lemma}

\begin{proof}
See \cite[Lemma 3.3]{GZS} and \cite[Remark 3.4]{GZS}.
\end{proof}

\begin{lemma}\label{lemma:2.2}
Let $X=(x,t)\in \Sigma_{\varepsilon^2,\varepsilon}^T$,
and $\delta(X)$ be given in $\eqref{eq:2.4}$. Then for any $Y\in P_r(X)$, we have
\begin{equation}\label{pri:2.11}
 \big|\delta(X) - \delta(Y)\big|\leq \emph{d}(X,Y) \leq r.
\end{equation}
Moreover, if we define
\begin{equation}
 S_{i,\varepsilon}(\delta)(x,t) = \int_{\mathbb{R}^{d+1}}
 \big|\nabla_i\zeta_\varepsilon(x-y,t-\tau)\big|\delta(y,\tau)dyd\tau
\end{equation}
with $i=0,1,\cdots,d,d+1$,
where $\nabla_0\zeta = \zeta$, and $\nabla_{d+1}\zeta=\partial_t\zeta$.
Then there holds
\begin{equation}\label{pri:2.4}
\big|S_{i,\varepsilon}(\delta)(x,t)\big|
\leq C\delta(x,t)
\quad\text{and}\quad
\big|S_{i,\varepsilon}(\delta^{-1})(x,t)\big|\leq C[\delta(x,t)]^{-1},
\end{equation}
where $C$ depends on $\zeta$ and $d$.
\end{lemma}

\begin{remark}
\emph{Note that there must be $S_{0,\varepsilon} = S_\varepsilon$ by definition.
So the symbol $S_{0,\varepsilon}$ is only used to simplify the statements of the lemma,
and will not appear in any other place.}
\end{remark}

\begin{proof}
The estimate $\eqref{pri:2.11}$ is easily observed, and we provide a proof for the sake of completeness. Let $X_0=(x_0,t_0)\in\partial\Omega_T$ be the point such that
$\delta(X) = |x-x_0| + |t-t_0|^{\frac{1}{2}}$, and we remark that either $x=x_0$ or $t=t_0$.
According to the definition of distance function (see $\eqref{eq:2.4}$), it is not hard
to see that for any $Y\in P_\varepsilon(X)$ with $Y=(y,s)$,
\begin{equation*}
\delta(Y) - \delta(X)
\leq |y-x_0| + |s-t_0|^{\frac{1}{2}} - |y-x_0| - |t-t_0|^{\frac{1}{2}}
\leq |y-x| + |s-t|^{\frac{1}{2}} = \text{d}(X,Y)\leq r,
\end{equation*}
and interchanging the variable $X$ and $Y$ leads to the same type inequality. This implies
the desired estimate $\eqref{pri:2.11}$.

We now proceed to prove the first estimate in $\eqref{pri:2.4}$,
the main idea is to quantify the difference between $\delta$ and $S_{i,\varepsilon}(\delta)$. It is clear to see that
\begin{equation}\label{f:2.1}
\begin{aligned}
\Big|S_{i,\varepsilon}(\delta)(x,t)\Big|
&\leq
\int_{\mathbb{R}^{d+1}}\big|\nabla_i\zeta_\varepsilon(x-y,t-s)\big|\big|\delta(y,s)
-\delta(x,t)\big| dyds
+ C\delta(x,t)\\
&\leq C\Big\{\varepsilon + \delta(x,t)\Big\},
\end{aligned}
\end{equation}
where we use the estimate $\eqref{pri:2.11}$ in the last step. Since $\delta(x,t)\geq \varepsilon$,
we have already proved the first estimate of $\eqref{pri:2.4}$. An argument similar to
the one used in $\eqref{f:2.1}$ will show the second one in $\eqref{pri:2.4}$, and we are done.
\end{proof}

\begin{lemma}\label{lemma:2.4}
Assume $f\in W^{1,0}_2(\mathbb{R}^{d+1})$.
Then we have
\begin{equation}\label{pri:2.17}
\begin{aligned}
\big\|K_\varepsilon(f)\big\|_{L^2(\mathbb{R}^{d+1})}
&\leq C\big\|f\big\|_{L^2(\mathbb{R}^{d+1})},\\
\big\|\nabla K_\varepsilon(f)\big\|_{L^2(\mathbb{R}^{d+1})}
&\leq C\varepsilon^{-1}\big\|f\big\|_{L^2(\mathbb{R}^{d+1})},
\end{aligned}
\end{equation}
as well as,
\begin{equation}\label{pri:2.13}
\big\|f-K_\varepsilon(f)\big\|_{L^2(\mathbb{R}^{d+1})}
\leq C\varepsilon \big\|\nabla f\big\|_{L^2(\mathbb{R}^{d+1})},
\end{equation}
where $C$ depends on $d$ and $\eta$. Moreover, if $f\in W_2^{2,1}(\mathbb{R}^{d+1})$, then
there holds
\begin{equation}\label{pri:2.14}
\big\|f-S_\varepsilon(f)\big\|_{L^2(\mathbb{R}^{d+1})}
\leq C\varepsilon \big\|\nabla f\big\|_{L^2(\mathbb{R}^{d+1})}
+ C\varepsilon^2\big\|\partial_t f\big\|_{L^2(\mathbb{R}^{d+1})},
\end{equation}
where $C$ depends on $d$ and $\zeta$.
\end{lemma}

\begin{proof}
The estimate $\eqref{pri:2.17}$ follows from the Plancherel theorem immediately,
while the estimates $\eqref{pri:2.13}$ and $\eqref{pri:2.14}$ essentially come from
the absolute continuity of the integral with respect to a small translation.
We will adopt the idea from \cite[Lemma 3.2]{GZS}, originally developed by Z. Shen \cite{SZW12}.
In view of the Plancherel theorem, the left-hand side of $\eqref{pri:2.13}$ is equal to
\begin{equation}\label{f:2.9}
 \int_{\mathbb{R}^{d+1}}|\widehat{f}(\xi,t)|^2 |1-\widehat{\eta}(\varepsilon\xi)|^2d\xi dt
\end{equation}
Since $\widehat{\eta}(0) = 1$, we have $|1-\widehat{\eta}(\varepsilon\xi)|
\leq C\varepsilon|\xi|$. Hence, the quantity $\eqref{f:2.9}$ may be controlled by
\begin{equation*}
  C\varepsilon^2  \int_{\mathbb{R}^{d+1}}|\widehat{f}(\xi,t)|^2 |\xi|^2 dt,
\end{equation*}
and this implies the desired estimate $\eqref{pri:2.13}$ through the Plancherel theorem again.
By the same token, the estimate $\eqref{pri:2.14}$ is based upon the following estimate
\begin{equation*}
  |\widehat{\zeta}(0,0)-\widehat{\zeta}(\varepsilon\xi,\varepsilon^2\rho)|
\leq C\Big\{\varepsilon|\xi|+\varepsilon^2|\rho|\Big\}
\end{equation*}
where the pair $(\xi,\rho)$ is in the phase space, produced by the spacial and
the time variable, respectively.
\end{proof}

Recall the definition of weighted-type norms $\eqref{eq:2.6}$, and the parabolic distance function
$\delta$ is defined in $\eqref{eq:2.4}$.

\begin{lemma}\label{lemma:2.5}
Let $f\in W^{1,0}_2(\Omega_T)$ be supported in $\Sigma_{\varepsilon^2,\varepsilon}^T$.
Then there holds
\begin{equation}\label{pri:2.8}
 \|K_\varepsilon(f)\|_{L^2(\Sigma_{\varepsilon^2,\varepsilon}^T;\omega)}
 \leq C\|f\|_{L^2(\Sigma_{\varepsilon^2,\varepsilon}^T;\omega)},
\end{equation}
\begin{equation}\label{pri:2.9}
 \|\nabla K_\varepsilon(f)\|_{L^2(\Sigma_{\varepsilon^2,\varepsilon}^T;\omega)}
 \leq C\varepsilon^{-1}\|f\|_{L^2(\Sigma_{\varepsilon^2,\varepsilon}^T;\omega)},
\end{equation}
for $\omega = \delta^{\pm 1}$, and
\begin{equation}\label{pri:2.10}
\|K_\varepsilon(f)-f\|_{L^2(\Sigma_{4\varepsilon^2,2\varepsilon}^T;\delta)}
\leq C\varepsilon
\|\nabla f\|_{L^2(\Sigma_{\varepsilon^2,\varepsilon}^T;\delta)},
\end{equation}
where $C$ depends on $d$ and $\eta$.
\end{lemma}

\begin{proof}
We only show the proof in the case of $\omega =\delta$, and the other case follows from the same way.
To show $\eqref{pri:2.8}$, we fix $t\in [4\varepsilon^2,T-4\varepsilon^2]$, and it follows
from \cite[Lemma 3.2]{QX2} that
\begin{equation}
\int_{\Sigma_{2\varepsilon}} \big|K_\varepsilon(f)(x,t)\big|^2\delta(x,t) dx
\leq C\int_{\Sigma_{2\varepsilon}} \big|f(x,t)\big|^2\delta(x,t) dx,
\end{equation}
where $C$ is independent of $t$, and integrating the above inequality with respect to $t$ from $4\varepsilon^2$ to
$T-4\varepsilon^2$ leads to the desired estimate $\eqref{pri:2.8}$. Denote
\begin{equation*}
 \tilde{K}_{i,\varepsilon}(\delta)(x,t)
 = \int_{\mathbb{R}^d}|\nabla_i\eta_\varepsilon(x-y)|\delta(y,t)dy,
\end{equation*}
where $i=1,\cdots,d$. Then an argument similar to the one used in Lemma $\ref{lemma:2.2}$ shows that
\begin{equation}\label{f:2.8}
\big|\tilde{K}_{i,\varepsilon}(\delta)(x,t)\big|\leq C\delta(x,t),
\end{equation}
since for fixed $t$, there still holds
\begin{equation*}
  |\delta(x,t) -\delta(y,t)|\leq |x-y|.
\end{equation*}
Hence, the estimate $\eqref{f:2.8}$ shall lead to the estimate $\eqref{pri:2.9}$.
The proof will not be fully included here, and the similar details will be found in Lemma $\ref{lemma:2.1}$.

By the same token, for any fixed $t\in [\varepsilon^2,T-\varepsilon^2]$, it follows from
\cite[Lemma 3.3]{QX2} that
\begin{equation*}
\int_{\Sigma_{2\varepsilon}}\big|f(x,t)-K_\varepsilon(f)(x,t)\big|^2\delta(x,t)dxdt
\leq C\varepsilon \int_{\Sigma_{\varepsilon}}\big|\nabla f(x,t)\big|^2\delta(x,t)dxdt,
\end{equation*}
where $C$ is independent of $t$, and this implies $\eqref{pri:2.10}$ is true. We have
completed the proof.
\end{proof}

\begin{lemma}[Weighted-type inequality I]\label{lemma:2.1}
Let $f\in L^2(\Omega_T)$ be supported in $\Sigma_{\varepsilon^2,\varepsilon}^{T}$. Then
for any $\varpi\in L_{per}^2(Y)$, there holds
\begin{equation}\label{pri:2.5}
\begin{aligned}
\bigg(\int_{\varepsilon^2}^{T-\varepsilon^2}
\int_{\Sigma_{\varepsilon}}\big|\varpi(x/\varepsilon,t/\varepsilon^2)
&\nabla S_\varepsilon(f)(x,t)\big|^2
\big[\delta(x,t)\big]^{\pm1}dxdt\bigg)^{1/2} \\
&\leq C\varepsilon^{-1}\|\varpi\|_{L^2(Y)}
\bigg(\int_{\varepsilon^2}^{T-\varepsilon^2}
\int_{\Sigma_{\varepsilon}}\big|f(x,t)\big|^2\big[\delta(x,t)\big]^{\pm1}dxdt\bigg)^{1/2}
\end{aligned}
\end{equation}
and
\begin{equation}\label{pri:2.6}
\begin{aligned}
\bigg(\int_{\varepsilon^2}^{T-\varepsilon^2}
\int_{\Sigma_{\varepsilon}}\big|\varpi(x/\varepsilon,t/\varepsilon^2)
&\partial_t S_\varepsilon(f)(x,t)\big|^2
\big[\delta(x,t)\big]^{\pm1}dxdt\bigg)^{1/2} \\
&\leq C\varepsilon^{-2}\|\varpi\|_{L^2(Y)}
\bigg(\int_{\varepsilon^2}^{T-\varepsilon^2}
\int_{\Sigma_{\varepsilon}}\big|f(x,t)\big|^2
\big[\delta(x,t)\big]^{\pm 1}dxdt\bigg)^{1/2},
\end{aligned}
\end{equation}
where $C$ depends only on $d$ and $\zeta$.
\end{lemma}

\begin{proof}
Concerning the estimates $\eqref{pri:2.5}$ and $\eqref{pri:2.6}$,
the main idea has already been in \cite[Lemma 3.2]{QX2}, and we provide a proof for the sake of
the completeness. For any $(x,t)\in \Sigma_{\varepsilon^2,\varepsilon}^T$ and
$i=1,\cdots,d$, it follows
from Cauchy's inequality that
\begin{equation}\label{f:2.2}
\begin{aligned}
\Big|\nabla_i\int_{\mathbb{R}^{d+1}}
&\zeta_\varepsilon(x-z,t-s)f(z,s)dzds\Big|^2
\leq \varepsilon^{-2}
\bigg(\int_{\mathbb{R}^{d+1}}
\big|\nabla_i\zeta_\varepsilon(x-z,t-s)\big||f(z,s)|dzds\bigg)^{2} \\
&\leq \varepsilon^{-2}
S_{i,\varepsilon}(\delta^{\mp1})(x,t)
\int_{\mathbb{R}^{d+1}}
\big|\nabla_i\zeta_\varepsilon(x-z,t-s)\big||f(z,s)|^2[\delta(z,s)]^{\pm1}dzds\\
&\leq C\varepsilon^{-2}\big[\delta(x,t)\big]^{\mp1}
\int_{\mathbb{R}^{d+1}}
\big|\nabla_i\zeta_\varepsilon(x-z,t-s)\big||f(z,s)|^2[\delta(z,s)]^{\pm1}dzds,
\end{aligned}
\end{equation}
where we employ the estimate $\eqref{pri:2.4}$ in the last step.
To estimate $\eqref{pri:2.5}$, it suffices to show
\begin{equation}
\begin{aligned}
&\int_{\varepsilon^2}^{T-\varepsilon^2}
\int_{\Sigma_{\varepsilon}}\big|\varpi(x/\varepsilon,t/\varepsilon^2)
\nabla_i S_\varepsilon(f)(x,t)\big|^2
[\delta(x,t)]^{\pm1}dxdt \\
&\leq C\varepsilon^{-2}
\int_{\varepsilon^2}^{T-\varepsilon^2}
\int_{\Sigma_{\varepsilon}}
\big|\varpi(x/\varepsilon,t/\varepsilon^2)\big|^2
\int_{\mathbb{R}^{d+1}}
\big|\nabla_i\zeta_\varepsilon(x-z,t-s)\big||f(z,s)|^2[\delta(z,s)]^{\pm1}dzds dxdt\\
&\leq C\varepsilon^{-2}
\sup_{(x,t)\in\mathbb{R}^{d+1}}\int_{Q((x,t),1/2)}
\big|\varpi(z,s)\big|^2 dzds
\int_{\varepsilon^2}^{T-\varepsilon^2}
\int_{\Sigma_{\varepsilon}}
|f(x,t)|^2[\delta(x,t)]^{\pm1}dxdt\\
& \leq C\varepsilon^{-2}
\big\|\varpi\big\|_{L^2(Y)}^2
\int_{\varepsilon^2}^{T-\varepsilon^2}
\int_{\Sigma_{\varepsilon}}
|f(y,\tau)|^2[\delta(x,t)]^{\pm1}dxdt,
\end{aligned}
\end{equation}
where we use the estimate $\eqref{f:2.2}$ in the first inequality, and
the second one follows from the Fubini theorem, and the last step is due to $\varpi\in L^2_{per}(Y)$. Here the symbol $Q((x,t),1/2)$ means a cube in $\mathbb{R}^{d+1}$ with
$(x,t)$ being the center, with $1$ as the length of a side.

Adopting the same procedure as we did above, one can derive the estimate $\eqref{pri:2.6}$ without any real difficulty, which is based on the fact
\begin{equation}
\partial_t\int_{\mathbb{R}^{d+1}}
\zeta_\varepsilon(x-z,t-s)f(z,s)dzds
\leq \varepsilon^{-2}
\int_{\mathbb{R}^{d+1}}
\big|\partial_t\zeta_\varepsilon(x-z,t-s)\big||f(z,s)|dzds,
\end{equation}
and the estimate $\eqref{pri:2.4}$ in the case of $i=d+1$. Thus we may end the proof here.
\end{proof}

\begin{lemma}[Weighted-type inequality II]\label{lemma:2.6}
Let $f\in W_2^{2,1}(\Omega_T)$ be supported in $\Sigma_{\varepsilon^2,\varepsilon}^T$.
Then we have
\begin{equation}\label{pri:2.12}
\begin{aligned}
&\bigg(\int_{2\varepsilon^2}^{T-2\varepsilon^2}
\int_{\Sigma_{2\varepsilon}}
\big|f(x,t)
- S_\varepsilon(f)(x,t)\big|^2
\delta(x,t)dxdt\bigg)^{1/2} \\
&\leq C\varepsilon
\bigg(\int_{\varepsilon^2}^{T-\varepsilon^2}
\int_{\Sigma_{\varepsilon}}
\big|\nabla f\big|^2
\delta(x,t)dxdt\bigg)^{1/2}
+ C\varepsilon^2\bigg(\int_{\varepsilon^2}^{T-\varepsilon^2}
\int_{\Sigma_{\varepsilon}}
\big|\partial_t f\big|^2
\delta(x,t)dxdt\bigg)^{1/2},
\end{aligned}
\end{equation}
where $C$ depends on $d$ and $\zeta$.
\end{lemma}

\begin{proof}
Unlike the method used in the proof of Lemma $\ref{lemma:2.4}$, we will employ
the arguments similar to that shown in \cite[Lemma 3.3]{QX2}
to prove this lemma. Let $|y|\leq 1$
and $0<\tau\leq 1$, and it is not hard to see that
\begin{equation*}
\begin{aligned}
\Big|f(x,t)-f(x-\varepsilon y,t-\varepsilon^2\tau)\Big|^2
&\leq \varepsilon^2\int_0^1|\nabla f(x+(s-1)\varepsilon y,t)|^2 ds \\
&+ \varepsilon^4
\int_0^1|\partial_t f(x-\varepsilon y,t+(\theta-1)\varepsilon^2\tau)|^2 d\theta.
\end{aligned}
\end{equation*}
Then we have
\begin{equation}\label{f:2.5}
\begin{aligned}
&\int_{2\varepsilon^2}^{T-2\varepsilon^2}
\int_{\Sigma_{2\varepsilon}}\big|f(x,t)-f(x-\varepsilon y,t-\varepsilon^2\tau)\big|^2 \delta(x,t)dxdt\\
&\leq \varepsilon^2\underbrace{\int_{2\varepsilon^2}^{T-2\varepsilon^2}
\int_{\Sigma_{2\varepsilon}}\int_0^1 \big|\nabla f(x+(s-1)\varepsilon y,t)
\big|^2\delta(x,t)dsdxdt}_{I_1} \\
& + \varepsilon^4
\underbrace{\int_{2\varepsilon^2}^{T-2\varepsilon^2}
\int_{\Sigma_{2\varepsilon}}
\int_0^1\big|\partial_t f(x+\varepsilon y,t+(\theta-1)\varepsilon^2\tau)\big|^2\delta(x,t)d\theta dxdt}_{I_2}.
\end{aligned}
\end{equation}
For $I_1$, set $z=x+(s-1)\varepsilon y$. It is clear to see $z\in \Sigma_{\varepsilon}$, and
\begin{equation}\label{f:2.6}
\begin{aligned}
I_1 &= \int_{2\varepsilon^2}^{T-2\varepsilon^2}
\int_{\Sigma_{\varepsilon}}
\int_0^1\big|\nabla f(z,t)\big|^2\delta(z-(s-1)\varepsilon y,t)ds dz dt\\
&\leq  2\int_{\varepsilon^2}^{T-\varepsilon^2}
\int_{\Sigma_{\varepsilon}}
\big|\nabla f(z,t)\big|^2\delta(z,t)dzdt,
\end{aligned}
\end{equation}
due to the fact that $|\delta(z-(s-1)\varepsilon y,t)-\delta(z,t)|\leq \varepsilon\leq
\delta(z,t)$ for any $t\in [2\varepsilon^2,T-2\varepsilon^2]$. Concerning $I_2$, we
reset $z=x-\varepsilon y$ and $\lambda = t+(\theta-1)\varepsilon^2\tau$.
Thus $z\in\Sigma_{\varepsilon}$ and $\lambda\in[\varepsilon^2,T-\varepsilon^2]$,
and it follows from the estimate $\eqref{pri:2.11}$ that
\begin{equation*}
\Big|\delta(z+\varepsilon y,\lambda-(\theta-1)\varepsilon^2\tau)-
\delta(z,\lambda)\Big|\leq \varepsilon \leq \delta(z,\lambda).
\end{equation*}
By the same token, we have
\begin{equation}\label{f:2.7}
I_2 \leq 2\int_{\varepsilon^2}^{T-\varepsilon^2}\int_{\Sigma_{\varepsilon}}
\big|\partial_t f(z,t)\big|^2\delta(z,t) dzdt.
\end{equation}
Inserting $\eqref{f:2.6}$ and $\eqref{f:2.7}$ into the estimate $\eqref{f:2.5}$ leads to
\begin{equation*}
\begin{aligned}
&\int_{2\varepsilon^2}^{T-2\varepsilon^2}
\int_{\Sigma_{2\varepsilon}}\big|f(x,t)-f(x-\varepsilon y,t-\varepsilon^2\tau)\big|^2 \delta(x,t)dxdt \\
&\leq 2\varepsilon^2
\int_{\varepsilon^2}^{T-\varepsilon^2}
\int_{\Sigma_{\varepsilon}}
\big|\nabla f(z,t)\big|^2\delta(z,t)dzdt
+ 2\varepsilon^4
\int_{\varepsilon^2}^{T-\varepsilon^2}
\int_{\Sigma_{\varepsilon}}
\big|\partial_t f(z,t)\big|^2\delta(z,t)dzdt.
\end{aligned}
\end{equation*}
Then the rest part of the proof of $\eqref{pri:2.12}$ is based upon the Fubini theorem and
the estimate $\eqref{pri:2.4}$,
which will be found in \cite[Lemma 3.3]{QX2} and will not be reproduced here. We have completed
the whole proof.
\end{proof}

\begin{corollary}\label{cor:2.1}
Assume $f\in W_2^{2,1}(\mathbb{R}^{d+1})$. Then there holds
\begin{equation}\label{pri:2.15}
\big\|\nabla f - \nabla S_\varepsilon(f)\big\|_{L^2(\mathbb{R}^{d+1})}
\leq C\varepsilon\Big\{\big\|\nabla^2 f\big\|_{L^2(\mathbb{R}^{d+1})}
+ \big\|\partial_t f\big\|_{L^2(\mathbb{R}^{d+1})}\Big\}.
\end{equation}
Moreover, if $f$ is supported in $\Sigma_{\varepsilon^2,\varepsilon}^T$, then we have
\begin{equation}\label{pri:2.16}
\big\|\nabla f - \nabla S_\varepsilon(f)\big\|_{L^2(\Sigma_{4\varepsilon^2,2\varepsilon}^T;\delta)}
\leq C\varepsilon\Big\{\big\|\nabla^2 f\big\|_{L^2(\Sigma_{\varepsilon,\varepsilon}^T;\delta)}
+ \big\|\partial_t f\big\|_{L^2(\Sigma_{\varepsilon^2,\varepsilon}^T;\delta)}\Big\}.
\end{equation}
where $C$ depends on $\zeta,\eta$ and $d$.
\end{corollary}

\begin{proof}
We mention that the estimate $\eqref{pri:2.15}$ had been proven in \cite[Lemma 3.2]{GZS}
by the Plancherel theorem. Based upon the previous Lemma $\ref{lemma:2.4}$, we provide a new proof here, and this method could be applied to the estimate $\eqref{pri:2.16}$ as well.

In view of $\eqref{pri:2.13}$, it is not hard to see that
\begin{equation*}
 \|f\|_{L^2(\mathbb{R}^{d+1})} \leq C\varepsilon\|\nabla f\|_{L^2(\mathbb{R}^{d+1})}
+ \|K_\varepsilon(f)\|_{L^2(\mathbb{R}^{d+1})}.
\end{equation*}
Hence, from the above inequality, it follows that
\begin{equation*}
\begin{aligned}
\big\|\nabla f-\nabla S_\varepsilon(f)&\big\|_{L^2(\mathbb{R}^{d+1})}
\leq C\varepsilon\big\|\nabla (\nabla f- S_\varepsilon(\nabla f))\big\|_{L^2(\mathbb{R}^{d+1})}
+ \big\|K_\varepsilon(\nabla f)-  S_\varepsilon K_\varepsilon
(\nabla f)\big\|_{L^2(\mathbb{R}^{d+1})}\\
&\leq C\varepsilon\big\|\nabla^2 f\big\|_{L^2(\mathbb{R}^{d+1})}
+ C\varepsilon\big\|\nabla K_\varepsilon(\nabla f)\big\|_{L^2(\mathbb{R}^{d+1})}
+ C\varepsilon^2\big\|\partial_t K_\varepsilon(\nabla f)\big\|_{L^2(\mathbb{R}^{d+1})}\\
& = C\varepsilon\big\|\nabla^2 f\big\|_{L^2(\mathbb{R}^{d+1})}
+ C\varepsilon\big\| K_\varepsilon(\nabla^2 f)\big\|_{L^2(\mathbb{R}^{d+1})}
+ C\varepsilon^2\big\|\nabla K_\varepsilon(\partial_t f)\big\|_{L^2(\mathbb{R}^{d+1})} \\
&\leq C\varepsilon\Big\{\big\|\nabla^2 f\big\|_{L^2(\mathbb{R}^{d+1})}
+ \big\|\partial_t f\big\|_{L^2(\mathbb{R}^{d+1})}\Big\},
\end{aligned}
\end{equation*}
where we use the estimates $\eqref{pri:2.1}$ and $\eqref{pri:2.14}$ in the second inequality,
and the estimate $\eqref{pri:2.17}$ in the last one.

Proceeding as in the proof of the estimate $\eqref{pri:2.15}$, it is not hard to derive
the weighted-type one $\eqref{pri:2.16}$. All the requirements have been established except
the following estimate
\begin{equation*}
\|S_\varepsilon(\nabla^2 f)\|_{L^2(\Sigma_{\varepsilon^2,\varepsilon}^T;\delta)}
\leq C\|\nabla^2 f\|_{L^2(\Sigma_{\varepsilon^2,\varepsilon}^T;\delta)}.
\end{equation*}
However, it is could be easily acquired from Lemma $\ref{lemma:2.2}$, and we omit the details here.
The proof is now complete.
\end{proof}

\begin{remark}
\emph{Although Corollary $\ref{cor:2.1}$ has not been employed in this paper, there
are two reasons making us feel necessary to write it out. One is that we provide an idea
in the proof of $\eqref{pri:2.15}$, which actually suggests a new way for
\cite[Lemma 3.2]{GZS} to avoid using the Fourier transformation method.
The other is that from the proof of this corollary, it is clear to see why we
fix the undetermined function $\varphi$ in $\eqref{eq:3.1}$
by choosing $\varphi = S_\varepsilon K_\varepsilon(\Psi_{[4\varepsilon^2,2\varepsilon]}\nabla u_0)$
rather than
$S_\varepsilon(\Psi_{[4\varepsilon^2,2\varepsilon]}\nabla u_0)$ in the later sections.}
\end{remark}

\section{Convergence rates in $L^2(0,T;(H^1(\Omega))^d)$}\label{section:3}

\begin{lemma}\label{lemma:3.4}
Suppose that $u_\varepsilon,u_0\in L^2(0,T;(H^1(\Omega))^d)$ with
$\partial_t u_\varepsilon,\partial_t u_0\in L^2(0,T;(H^{-1}(\Omega))^{d})$ satisfy
\begin{equation}\label{eq:3.4}
\left\{\begin{aligned}
\frac{\partial u_\varepsilon}{\partial t} + \mathcal{L}_\varepsilon(u_\varepsilon)
& = \frac{\partial u_0}{\partial t} + \mathcal{L}_0(u_0) &\quad&\emph{in}\quad ~\Omega_T,\\
u_\varepsilon & = u_0 &\quad&\emph{on}\quad \partial_p\Omega_T.
\end{aligned}\right.
\end{equation}
Let $w_\varepsilon = (w_\varepsilon^\beta)$ with
\begin{equation}\label{eq:3.1}
w_\varepsilon^\beta = u_\varepsilon^\beta - u_0^\beta
-\varepsilon\chi_j^{\beta\gamma}(x/\varepsilon,t/\varepsilon^2)\varphi_j^\gamma
-\varepsilon^2 E_{l(d+1)j}^{\beta\gamma}(x/\varepsilon,t/\varepsilon^2)
\frac{\partial}{\partial x_l}\big\{\varphi_j^\gamma\big\},
\end{equation}
where $\varphi_j^\gamma\in W^{2,1}_2(\Omega_T)$ is supported in $\Omega_T$. Then we have
\begin{equation}\label{eq:3.2}
\left\{\begin{aligned}
\frac{\partial w_\varepsilon}{\partial t} + \mathcal{L}_\varepsilon(w_\varepsilon)
& = \emph{div}(\tilde{f}) &\quad&\emph{in}\quad ~\Omega_T,\\
w_\varepsilon & = 0 &\quad&\emph{on}\quad \partial_p\Omega_T,
\end{aligned}\right.
\end{equation}
where $\tilde{f} = (\tilde{f}_i^\alpha)$ with
\begin{equation}\label{eq:3.3}
\begin{aligned}
\tilde{f}_i^\alpha
&= \big[a_{ij}^{\alpha\beta}(y,\tau)-\widehat{a}_{ij}^{\alpha\beta}\big]
\Big[\frac{\partial u_0^\beta}{\partial x_j}-\varphi_j^\beta\Big] \\
& + \varepsilon \bigg[ a_{il}^{\alpha\beta}(y,\tau)\chi_j^{\beta\gamma}(y,\tau)
+ E_{lij}^{\alpha\gamma}(y,\tau)
+ a_{ik}^{\alpha\beta}(y,\tau)\frac{\partial E_{l(d+1)j}^{\beta\gamma}}{\partial y_k}(y,\tau)
\bigg]\frac{\partial\varphi_j^\gamma}{\partial x_l} \\
& + \varepsilon^2\Big[
a_{ik}^{\alpha\beta}(y,\tau)E_{l(d+1)j}^{\beta\gamma}(y,\tau)
\frac{\partial^2\varphi_j^\gamma}{\partial x_l\partial x_k}
- E_{i(d+1)j}^{\alpha\gamma}(y,\tau)\frac{\partial\varphi_j^\gamma}{\partial t}\Big],
\end{aligned}
\end{equation}
where $1\leq i,j,l,k\leq d$ and $1\leq \alpha,\beta,\gamma\leq d$,
and $y=x/\varepsilon$ with $\tau=t/\varepsilon^2$.
\end{lemma}

\begin{proof}
The proof may be found in \cite[Theorem 2.2]{GZS}, and we provide a proof for the sake of
the completeness. Observing the equation $\eqref{eq:3.4}$, we have
\begin{equation}\label{f:3.25}
\begin{aligned}
\frac{\partial w_\varepsilon}{\partial t} + \mathcal{L}_\varepsilon(w_\varepsilon)
&= \mathcal{L}_0(u_0) - \mathcal{L}_\varepsilon(u_0)
-\Big(\frac{\partial}{\partial t}+\mathcal{L}_\varepsilon\Big)
\big[\varepsilon\chi_j^{\cdot\gamma}(x/\varepsilon,t/\varepsilon^2)\varphi_j^\gamma\big]\\
&-\Big(\frac{\partial}{\partial t}+\mathcal{L}_\varepsilon\Big)
\big[\varepsilon^2E_{l(d+1)j}(x/\varepsilon,t/\varepsilon^2)
\frac{\partial}{\partial x_l}(\varphi_j^\gamma)\big] \\
& = -\underbrace{\frac{\partial}{\partial x_i}
\bigg\{ b_{ij}^{\cdot\beta}(y,\tau)\varphi_j^\beta \bigg\}}_{I_1}
-\underbrace{\frac{\partial}{\partial t}\big[\varepsilon\chi_j^{\cdot\gamma}(y,\tau)\varphi_j^\gamma\big]}_{I_2}
-\underbrace{\Big(\frac{\partial}{\partial t}+\mathcal{L}_\varepsilon\Big)
\big[\varepsilon^2 E_{l(d+1)j}^{\cdot\gamma}(y,\tau)
\frac{\partial\varphi_j^\gamma}{\partial x_l}\big]}_{I_3}\\
&-\frac{\partial}{\partial x_i}\bigg\{[\widehat{a}_{ij}^{\cdot\beta} - a_{ij}^{\cdot\beta}(y,\tau)]\big[\frac{\partial u_0^\beta}{\partial x_j}
- \varphi_j^\gamma\big]
-\varepsilon a_{ik}^{\cdot\beta}(y,\tau)\chi_j^{\beta\gamma}(y,\tau)
\frac{\partial\varphi_j^\gamma}{\partial x_k}\bigg\}
\end{aligned}
\end{equation}
where $b_{ij}^{\cdot\beta} = (b_{ij}^{\alpha\beta})$ is shown as in $\eqref{eq:2.1}$, and
$y=x/\varepsilon$ with $\tau = t/\varepsilon^2$. The last line of $\eqref{f:3.25}$ is
a good term, and our task is reduced to calculate $I_1$, $I_2$ and $I_3$. Recalling
Lemma $\ref{lemma:2.3}$, it is not hard to see that
\begin{equation}\label{f:3.26}
\begin{aligned}
~[I_1]^\alpha + [I_2]^\alpha
= \varepsilon\chi_j^{\alpha\beta}(y,\tau)\frac{\partial\varphi_j^\beta}{\partial t}
+ b_{ij}^{\alpha\beta}(y,\tau)\frac{\partial\varphi_j^\beta}{\partial x_i},
\end{aligned}
\end{equation}
where we use the equality (ii) in $\eqref{eq:2.3}$. Then in view of $\eqref{eq:2.2}$, the right-hand side
of $\eqref{f:3.26}$ may be rewritten by
\begin{equation}\label{f:3.27}
\begin{aligned}
&\qquad\qquad
\varepsilon\frac{\partial}{\partial y_{k^{\prime}}}\big[E_{k^\prime(d+1)j}^{\alpha\beta}\big]
\frac{\partial\varphi_j^\beta}{\partial t}
+ \frac{\partial}{\partial y_{k^{\prime}}}\big[E_{k^\prime ij}^{\alpha\beta}\big]
\frac{\partial\varphi_j^\beta}{\partial x_i}
+ \frac{\partial}{\partial \tau}\big[E_{(d+1)ij}^{\alpha\beta}\big]
\frac{\partial\varphi_j^\beta}{\partial x_i}\\
& = \varepsilon^2\frac{\partial}{\partial x_{k^{\prime}}}\big[E_{k^\prime(d+1)j}^{\alpha\beta}(y,\tau)\big]
\frac{\partial\varphi_j^\beta}{\partial t}
+\varepsilon\frac{\partial}{\partial x_{k^{\prime}}}
\big[E_{k^\prime ij}^{\alpha\beta}(y,\tau)\big]
\frac{\partial\varphi_j^\beta}{\partial x_i}
+ \varepsilon^2\frac{\partial}{\partial t}\big[E_{(d+1)ij}^{\alpha\beta}(y,\tau)\big]
\frac{\partial\varphi_j^\beta}{\partial x_i}
\end{aligned}
\end{equation}
where $k^\prime = 1,\cdots,d$, and we use the fact that $E_{(d+1)(d+1)j} = 0$ due to the
antisymmetry. Again, employing the antisymmetry of $E$, the second line of $\eqref{f:3.27}$
is equal to
\begin{equation*}
\begin{aligned}
\varepsilon^2\frac{\partial}{\partial x_{k^\prime}}\bigg\{E_{k^\prime(d+1)j}^{\alpha\beta}(y,\tau)
\frac{\partial\varphi_j^\beta}{\partial t}\bigg\}
+\varepsilon\frac{\partial}{\partial x_{k^{\prime}}}
\bigg\{E_{k^\prime ij}^{\alpha\beta}(y,\tau)
\frac{\partial\varphi_j^\beta}{\partial x_i}\bigg\}
+\varepsilon^2\frac{\partial}{\partial t}
\bigg\{E_{(d+1)ij}^{\alpha\beta}(y,\tau)\frac{\partial\varphi_j^\beta}{\partial x_i}\bigg\}
\end{aligned}
\end{equation*}
Thus, we have
\begin{equation*}
\begin{aligned}
I_1 + I_2 + I_3
:= \varepsilon^2\frac{\partial}{\partial x_{k^\prime}}\bigg\{E_{k^\prime(d+1)j}^{\cdot\beta}(y,\tau)
\frac{\partial\varphi_j^\beta}{\partial t}\bigg\}
+\varepsilon\frac{\partial}{\partial x_{k^{\prime}}}
\bigg\{E_{k^\prime ij}^{\cdot\beta}(y,\tau)
\frac{\partial\varphi_j^\beta}{\partial x_i}\bigg\}
+\varepsilon^2\mathcal{L}_\varepsilon\big[E_{l(d+1)j}^{\cdot\gamma}(y,\tau)
\frac{\partial\varphi_j^\gamma}{\partial x_l}\big]
\end{aligned}
\end{equation*}
and this together with the last line of $\eqref{f:3.25}$ implies the desired formula
$\eqref{eq:3.3}$. The proof is complete.
\end{proof}

\begin{thm}\label{thm:3.1}
Given $F\in (L^2(\Omega_T))^d$, $g\in ({_0H^{1,\frac{1}{2}}}(S_T))^d$
and $h\in (H^1_0(\Omega))^d$.
Let $u_\varepsilon,u_0\in L^2(0,T;(H^1(\Omega))^d)$ with
$\partial_tu_\varepsilon,\partial_tu_0\in L^2(0,T;(H^{-1}(\Omega))^d)$ be the weak solutions to
$(\emph{DP}_\varepsilon)$ and $(\emph{DP}_0)$, respectively. Then by setting
$\varphi_j^\gamma = S_\varepsilon K_\varepsilon(\Psi_{[4\varepsilon^2,2\varepsilon]} \nabla_ju_0^\gamma)$ in $\eqref{eq:3.1}$, we have
\begin{equation}\label{pri:3.0}
\|w_\varepsilon\|_{L^2(0,T;H^1_0(\Omega))}
\leq C\varepsilon^{1/2}\Big\{\|F\|_{L^2(\Omega_T)}
+\|g\|_{H^{1,1/2}(S_T)} + \|h\|_{H^1(\Omega)}\Big\},
\end{equation}
where $C$ depends only on $\mu_1, \mu_2, d,T$ and $\Omega$.
\end{thm}

\begin{lemma}\label{lemma:3.1}
Suppose $u_0\in (W^{2,1}_{2,\emph{loc}}(\Omega_T))^d$.
Let $w_\varepsilon$ be given as in $\eqref{eq:3.1}$ and
satisfy the problem $\eqref{eq:3.2}$.
Then by choosing $\varphi_j^\gamma =
S_\varepsilon K_\varepsilon(\Psi_{[4\varepsilon^2,2\varepsilon]}\nabla_ju_0^\gamma)$
in $\eqref{eq:3.1}$, we may have
\begin{equation}
\begin{aligned}
\|w_\varepsilon\|_{L^2(0,T;H_0^1(\Omega))}
&\leq C\bigg\{\sup_{\varepsilon^2<t<T} \Big(\int_{t-\varepsilon^2}^t
\int_{\Sigma_{2\varepsilon}}|\nabla u_0|^2 dxds\Big)^{\frac{1}{2}}
+\|\nabla u_0\|_{L^2((\Omega\setminus\Sigma_{4\varepsilon})_T)}\bigg\}\\
& + C\varepsilon\Bigg\{
\|\nabla^2 u_0 \|_{L^2(\Sigma_{4\varepsilon^2,2\varepsilon}^{T})}
+\|\partial_t u_0 \|_{L^2(\Sigma_{4\varepsilon^2;2\varepsilon}^T)}\Bigg\},
\end{aligned}
\end{equation}
where $C$ depends only on $\mu_1,\mu_2,d,T$ and $\Omega$.
\end{lemma}

\begin{proof}
Noting that $\varphi_j^\gamma =
S_\varepsilon K_\varepsilon(\Psi_{[4\varepsilon^2,2\varepsilon]}\nabla_ju_0^\gamma)$ in $\eqref{eq:3.1}$ such that $w_\varepsilon\in L^2(0,T;(H_0^1(\Omega))^d)$ satisfies $\eqref{eq:3.2}$, it follows from
the estimate $\eqref{pri:2.3}$ and the expression $\eqref{eq:3.3}$ that
\begin{equation}\label{f:3.1}
\begin{aligned}
\|w_\varepsilon\|_{L^2(0,T;H_0^1(\Omega))}
&\leq C\|\tilde{f}\|_{L^2(\Omega_T)} \\
&\leq
C\big\|\nabla u_0
- S_\varepsilon K_\varepsilon(\Psi_{[4\varepsilon^2,2\varepsilon]}\nabla u_0 )\big\|_{L^2(\Omega_T)} \\
& +C\varepsilon\big\|\varpi(\cdot/\varepsilon,\cdot/\varepsilon^2)\nabla S_\varepsilon K_\varepsilon
(\Psi_{[4\varepsilon^2,2\varepsilon]}\nabla u_0)\big\|_{L^2(\Omega_T)}\\
& +C\varepsilon^2 \big\|\varpi(\cdot/\varepsilon,\cdot/\varepsilon^2)\nabla^2 S_\varepsilon K_\varepsilon
(\Psi_{[4\varepsilon,2\varepsilon]}\nabla u_0)\big\|_{L^2(\Omega_T)}\\
& +C\varepsilon^2 \big\|\varpi(\cdot/\varepsilon,\cdot/\varepsilon^2)\partial_t S_\varepsilon K_\varepsilon
(\Psi_{[4\varepsilon^2,2\varepsilon]}\nabla u_0)\big\|_{L^2(\Omega_T)}
=: I_1 + I_2 + I_3 + I_4.
\end{aligned}
\end{equation}
To estimate $I_1$, we first notice
the following fact
\begin{equation*}
\begin{aligned}
\nabla u_0
- S_\varepsilon K_\varepsilon(\Psi_{[4\varepsilon^2,2\varepsilon]}\nabla u_0)
&= \Psi_{[4\varepsilon^2,2\varepsilon]}\nabla u_0
-  K_\varepsilon(\Psi_{[4\varepsilon^2,2\varepsilon]}\nabla u_0) \\
& + K_\varepsilon (\Psi_{[4\varepsilon^2,2\varepsilon]} \nabla u_0)
- S_\varepsilon K_\varepsilon(\Psi_{[4\varepsilon^2,2\varepsilon]}\nabla u_0)
+ (1 -\Psi_{[4\varepsilon^2,2\varepsilon]})\nabla u_0.
\end{aligned}
\end{equation*}
Hence, it follows from the estimates $\eqref{pri:2.13}$, $\eqref{pri:2.14}$
and $\eqref{pri:2.17}$  that
\begin{equation*}
\begin{aligned}
I_1&\leq C\Bigg\{\|
\Psi_{[4\varepsilon^2,2\varepsilon]}\nabla u_0
-  K_\varepsilon(\Psi_{[4\varepsilon^2,2\varepsilon]}\nabla u_0)\|_{L^2(\mathbb{R}^{d+1})}\\
& \qquad\qquad+ \|K_\varepsilon (\Psi_{[4\varepsilon^2,2\varepsilon]} \nabla u_0)
- S_\varepsilon K_\varepsilon(\Psi_{[4\varepsilon^2,2\varepsilon]}\nabla u_0)
\|_{L^2(\mathbb{R}^{d+1})}
+\|(1-\Psi_{[4\varepsilon^2,2\varepsilon]})\nabla u_0\|_{L^2(\Omega_T)}\Bigg\}\\
&\leq C\varepsilon
\bigg\{\|\nabla(\Psi_{[4\varepsilon^2,2\varepsilon]}\nabla u_0)\|_{L^2(\mathbb{R}^d)}
+ \varepsilon\|\partial_tK_\varepsilon(\Psi_{[4\varepsilon^2,2\varepsilon]}\nabla u_0)\|_{L^2(\mathbb{R}^d)}
\bigg\}
+C\|\nabla u_0\|_{L^2(\boxbox_{2\varepsilon})},
\end{aligned}
\end{equation*}
where we also use the fact that $\Psi_{[4\varepsilon^2,2\varepsilon]}$ is supported in
$\Sigma_{4\varepsilon^2,2\varepsilon}^T$, and
$\boxbox_{2\varepsilon}=\Omega_T\setminus \Sigma_{8\varepsilon^2,4\varepsilon}^T$.

Since there holds the following identity
\begin{equation}\label{f:3.28}
\begin{aligned}
\partial_t K_\varepsilon(\Psi_{[4\varepsilon^2,2\varepsilon]}\nabla u_0)
= K_\varepsilon(\partial_t\Psi_{[4\varepsilon^2,2\varepsilon]}\nabla u_0)
+ \nabla K_\varepsilon(\Psi_{[4\varepsilon^2,2\varepsilon]}\partial_t u_0)
- K_\varepsilon(\nabla\Psi_{[4\varepsilon^2,2\varepsilon]}\partial_t u_0),
\end{aligned}
\end{equation}
we then obtain
\begin{equation}\label{f:3.2}
\begin{aligned}
I_1 &\leq C\|\nabla u_0\|_{L^2(\boxbox_\varepsilon)}
+ C\varepsilon\Big\{\|\nabla^2 u_0\|_{L^2(\Sigma_{4\varepsilon^2,2\varepsilon}^T)}
+\|\partial_t u_0\|_{L^2(\Sigma_{4\varepsilon^2,2\varepsilon}^T)}\Big\}\\
&\leq C\Bigg\{\sup_{\varepsilon^2<t<T} \Big(\int_{t-\varepsilon^2}^t
\int_{\Sigma_{2\varepsilon}}|\nabla u_0|^2 dxds\Big)^{\frac{1}{2}}
+\|\nabla u_0\|_{L^2((\Omega\setminus\Sigma_{2\varepsilon})_T)}\Bigg\}\\
& + C\varepsilon\Bigg\{
\|\nabla^2 u_0 \|_{L^2(\Sigma_{4\varepsilon^2,2\varepsilon}^{T})}
+\|\partial_t u_0 \|_{L^2(\Sigma_{4\varepsilon^2;2\varepsilon}^T)}\Bigg\},
\end{aligned}
\end{equation}
where we use the estimates $\eqref{pri:2.17}$ in the first inequality again, and the second one is due to
\begin{equation}\label{f:3.3}
\|\nabla u_0\|_{L^2(\boxbox_\varepsilon)}
\leq C\bigg\{
\|\nabla u_0\|_{L^2((\Omega\setminus\Sigma_{2\varepsilon})_T)}
+ \sup_{\varepsilon^2<t<T} \Big(\int_{t-\varepsilon^2}^t
\int_{\Sigma_{2\varepsilon}}|\nabla u_0|^2 dxds\Big)^{\frac{1}{2}}\bigg\}.
\end{equation}
We now proceed to study $I_2$, and it follows from the first estimate in $\eqref{pri:2.1}$ that
\begin{equation}\label{f:3.4}
\begin{aligned}
I_2
\leq C\varepsilon\|\nabla(\Psi_{[4\varepsilon^2,2\varepsilon]}
\nabla u_0 )\|_{L^2(\mathbb{R}^{d+1})}
\leq C\Big\{\|\nabla u_0\|_{L^2((\Omega\setminus\Sigma_{2\varepsilon})_T)}
+ \varepsilon\|\nabla^2 u_0\|_{L^2(\Sigma_{4\varepsilon^2,2\varepsilon}^T)}\Big\}.
\end{aligned}
\end{equation}
By observing that
\begin{equation*}
\nabla^2S_\varepsilon K_\varepsilon(\Psi_{[4\varepsilon^2,\varepsilon]}\nabla u_0)
= \nabla S_\varepsilon K_\varepsilon(\nabla(\Psi_{[4\varepsilon^2,\varepsilon]}\nabla u_0)),
\end{equation*}
it is not hard to derive
\begin{equation}\label{f:3.5}
I_3 \leq C\varepsilon\|\nabla(\Psi_{[4\varepsilon^2,\varepsilon]}
\nabla u_0)\|_{L^2(\mathbb{R}^{d+1})}
\leq C\Big\{\|\nabla u_0\|_{L^2((\Omega\setminus\Sigma_{2\varepsilon})_T)}
+ \varepsilon\|\nabla^2 u_0\|_{L^2(\Sigma_{4\varepsilon^2,2\varepsilon}^T)}\Big\}
\end{equation}
from the second estimate in $\eqref{pri:2.1}$.
Based upon a similar fact that
\begin{equation}\label{f:3.29}
\partial_t S_\varepsilon K_\varepsilon(\Psi_{[4\varepsilon^2,2\varepsilon]}\nabla u_0)
= S_\varepsilon K_\varepsilon(\partial_t\Psi_{[4\varepsilon^2,2\varepsilon]}\nabla u_0)
+ \nabla S_\varepsilon K_\varepsilon(\Psi_{[4\varepsilon^2,2\varepsilon]}\partial_t u_0)
- S_\varepsilon K_\varepsilon(\nabla\Psi_{[4\varepsilon^2,2\varepsilon]}\partial_t u_0),
\end{equation}
using the estimates $\eqref{pri:2.1}$ and $\eqref{pri:2.17}$ again, we arrive at
\begin{equation}\label{f:3.6}
I_4 \leq  C\sup_{\varepsilon^2<t<T} \Big(\int_{t-\varepsilon^2}^t
\int_{\Sigma_{2\varepsilon}}|\nabla u_0|^2 dxds\Big)^{\frac{1}{2}}
+ C\varepsilon\|\partial_t u_0 \|_{L^2(\Sigma_{4\varepsilon^2;2\varepsilon}^T)}.
\end{equation}

Consequently, plugging $\eqref{f:3.2}$, $\eqref{f:3.4}$, $\eqref{f:3.5}$ and $\eqref{f:3.6}$
back into $\eqref{f:3.1}$ leads to the desired result, and we have completed the proof.
\end{proof}

\begin{lemma}\label{lemma:3.3}
Suppose $F\in (L^2(\Omega_T))^d$, $g\in ({_0H^{1,1/2}(S_T)})^d$ and
$h\in (H^1_0(\Omega))^d$.
Let $u_0\in (W_2^{1,0}(\Omega_T))^d\cap (W^{2,1}_{2,\emph{loc}}(\Omega_T))^d$ be a weak solution of
$(\emph{DP}_0)$.
Then for any $\varepsilon^2<t\leq T$, we have the following
interior estimates
\begin{equation}\label{pri:3.4}
\begin{aligned}
\int_{t-\varepsilon^2}^t\int_{\Sigma_{2\varepsilon}}|\nabla u_0|^2 dx ds
&\leq C\varepsilon^{-2}\int_{t-\varepsilon^2}^t
\int_{\Sigma_{2\varepsilon}\setminus\Sigma_{4\varepsilon}}|u_0|^2 dxds
+ C\bigg(\int_{t-\varepsilon^2}^t\int_{\Sigma_{2\varepsilon}}
|u_0|^2 dxds\bigg)^{\frac{1}{2}} \\
&\times\Bigg\{
\Big(\int_{t-\varepsilon^2}^t\int_{\Sigma_{2\varepsilon}}
|\partial_t u_0|^2 dxds\Big)^{\frac{1}{2}}
 +\Big(\int_{0}^T\int_{\Sigma_{2\varepsilon}}
|F|^2 dxds\Big)^{\frac{1}{2}}
\Bigg\},
\end{aligned}
\end{equation}
where $C$ depends on $\mu_1,\mu_2,d$, as well as,
\begin{equation}\label{pri:3.6}
  \int_0^T\int_{\Sigma_r} |\nabla^2 u_0|^2 dxdt
  \leq C_r \bigg\{\int_{\Omega_T} |\nabla u_0|^2 dxdt
  +\int_{\Omega_T} |F|^2 dxdt\bigg\},
\end{equation}
where $C_r$ will blow up as $r\to 0$. Moreover, there also holds a global estimate
\begin{equation}\label{pri:3.5}
\sup_{0\leq t\leq T}\Big(\int_{\Omega}|u_0|^2 dx \Big)^{\frac{1}{2}}
+ \Big(\int_0^T\int_{\Omega}|\nabla u_0|^2 dxdt\Big)^{\frac{1}{2}}
\leq C\Big\{\|F\|_{L^2(\Omega_T)}+\|h\|_{L^2(\Omega)}+\|g\|_{H^{1,1/2}(S_T)}\Big\},
\end{equation}
where $C$ depends on $\mu_1,\mu_2,d,T$ and $\Omega$.
\end{lemma}

\begin{remark}
\emph{In fact, the estimate $\eqref{pri:3.5}$ could be improved from the point of view
in the homogenization theory, and one may easily derive from the estimates $\eqref{pri:2.18}$
and $\eqref{f:2.10}$ that
\begin{equation*}
\sup_{0\leq t\leq T}\|u_0\|_{L^2(\Omega)}
+ \big\|\nabla u_0\big\|_{L^2(\Omega_T)}
\leq C\Big\{\big\|F\big\|_{L^2(0,T;H^{-1}(\Omega))}+\big\|h\big\|_{L^2(\Omega)}
+ \big\|g\big\|_{H^{1/2,1/4}(S_T)}\Big\},
\end{equation*}
by noting that $\|\cdot\|_{L^2(0,T;H^1(\Omega))}$ and
$\|\cdot\|_{L^\infty(0,T;L^2(\Omega))}$ are lower semi-continuous with
respect to the weak convergence and to the $\text{weak}^\ast$ convergence, respectively.}
\end{remark}

\begin{proof}
The estimate $\eqref{pri:3.4}$ is known as Caccioppoli's inequality.
Let $\psi^2u_0$ be a test function, where $\psi\in C_0^{1}(\Omega)$.
By the divergence theorem, we have
\begin{equation*}
\int_{\Omega} \psi^2 \partial_t u_0 u_0 dx
+ \int_{\Omega}\psi^2\widehat{A}\nabla u_0\cdot\nabla u_0 dx
+ 2\int_{\Omega}\psi\widehat{A}\nabla u_0\cdot\nabla\psi u_0 dx
= \int_\Omega \psi^2 Fu_0 dx.
\end{equation*}
On account of the elasticity assumption $\eqref{c:1}$ and Young's inequality, there holds
\begin{equation}\label{f:3.31}
\frac{\mu_1}{8}\int_{\Omega}\psi^2|\nabla u_0|^2 dx
\leq C\int_{\Omega}|\nabla\psi|^2|u_0|^2 dx
+ \int_{\Omega}\big| \psi^2 \partial_t u_0 u_0 \big|dx
+ \int_\Omega  \big| \psi^2 Fu_0 \big| dx,
\end{equation}
where we exactly carry out a following simple computation:
\begin{equation*}
\begin{aligned}
\int_{\Omega}\psi^2\widehat{A}\nabla u_0 \nabla u_0 dx
&\geq \frac{\mu_1}{4}\int_{\Omega}\psi^2|\nabla u_0 + (\nabla u_0)^T|^2 dx\\
&\geq \frac{\mu_1}{2}\int_{\Omega}\psi^2|\nabla u_0|^2 dx
- \mu_1\int_{\Omega} |\psi\nabla u_0||\nabla \psi u_0| dx,
\end{aligned}
\end{equation*}
and the notation $(\nabla u_0)^T$ denotes the transpose of $d\times d$ matrix $\nabla u_0$.

Here, we concretely choose $\psi=\psi_{2\varepsilon}$ to be the cut-off function, where
$\psi_{2\varepsilon} = 1$ in $\Sigma_{4\varepsilon}$, $\psi_{2\varepsilon} = 0$ outside $\Sigma_{2\varepsilon}$ and $|\nabla\psi_{2\varepsilon}|\leq C/\varepsilon$.
Hence, we obtain
\begin{equation*}
\int_{\Sigma_{2\varepsilon}}|\nabla u|^2 dx
\leq C\varepsilon^{-2}\int_{\Sigma_{2\varepsilon}\setminus\Sigma_{4\varepsilon}}
|u|^2 dx
+ \int_{\Sigma_{2\varepsilon}}\big|\partial_t u u \big|dx
+ \int_{\Sigma_{2\varepsilon}}  \big|Fu\big| dx,
\end{equation*}
and then take integral on the both sides from $t-\varepsilon^2$ to $t$. The desired estimate
$\eqref{pri:3.4}$ immediately follows from Cauchy's inequality.

The estimate $\eqref{pri:3.6}$ has been proved in \cite[Theorem 3.4.1]{WZQ} in detail,
so the proof will not be repeated here.

We now turn to address the estimate $\eqref{pri:3.5}$. Since
$\partial_t u_0 + \mathcal{L}_0(u_0) = F$ in $\Omega_T$,
taking $u_0$ as the text function and then integrating by parts, we have
\begin{equation*}
\frac{1}{2}\frac{\partial}{\partial t}
\Big(\int_{\Omega} u_0^2 dx\Big)
+ \int_\Omega \widehat{A}\nabla u_0 \nabla u_0 dx
= \int_\Omega Fu_0 dx + \int_{\partial\Omega} \frac{\partial u_0}{\partial \nu_0} u_0 dS.
\end{equation*}
Here we assume $\|(\partial u_0/\partial\nu_0)\|_{L^2(S_T)}<\infty$ for a movement to make the above
identity reasonable.
By the elasticity assumption $\eqref{c:1}$, we have
\begin{equation}\label{f:3.22}
\frac{\partial}{\partial t}
\Big(\int_{\Omega} u_0^2 dx\Big)
+ \mu_1\int_{\Omega}|\nabla u_0|^2 dx
\leq  2\int_{\Omega}|Fu_0|dx
+ 2\int_{\partial\Omega}\Big|\frac{\partial u_0}{\partial \nu_0}u_0\Big|dS
+ \mu_1\int_{\partial\Omega}|\nabla_{\text{tan}}u_0||u_0| dS,
\end{equation}
where the symbol $\nabla_{\tan} =
n_i\frac{\partial}{\partial x_\alpha} - n_\alpha\frac{\partial}{\partial x_i}$ denotes
the a tangential derivative.
Here we also employ the following Korn inequality:
\begin{equation*}
\int_{\Omega}|\nabla u_0 + (\nabla u_0)^T|^2dx
\geq 2\int_\Omega|\nabla u_0|^2dx - 2\int_{\partial\Omega}|\nabla_{\text{tan}}u_0||u_0|dS.
\end{equation*}
Then
it is fine to
split $u_0$ into $v$ and $w$, and they satisfy
\begin{equation*}
(\text{a}) \left\{\begin{aligned}
\partial_t v + \mathcal{L}_0(v) &= F &\quad&\text{in}\quad \Omega_T,\\
v & = 0 &~& \text{on}\quad S_T,\\
v & = h &~& \text{on}\quad \Omega_0,
\end{aligned}\right.
\qquad
(\text{b}) \left\{\begin{aligned}
\partial_t w + \mathcal{L}_0(w) &= 0 &\quad&\text{in}\quad \Omega_T,\\
w & = g &~& \text{on}\quad S_T,\\
w & = 0 &~& \text{on}\quad \Omega_0,
\end{aligned}\right.
\end{equation*}
respectively.
Thus, concerning $(\text{a})$, the estimate $\eqref{f:3.22}$ coupled with the Gronwall's inequality yields
\begin{equation}\label{f:3.23}
\max_{0\leq t\leq T}\|v\|_{L^2(\Omega)}
+ \|\nabla v\|_{L^2(\Omega_T)}
\leq C\Big\{\|h\|_{L^2(\Omega)}
+\|F\|_{L^2(\Omega_T)}\Big\}
\end{equation}
For $(\text{b})$, the estimate together with Cauchy's inequality and the trace theorem
gives
\begin{equation*}
\frac{\partial}{\partial t}
\Big(\int_{\Omega} w^2 dx\Big)
+ \mu_1\int_{\Omega}|\nabla w|^2 dx
\leq C\bigg\{\int_{\Omega} w^2 dx
+ \int_{\partial\Omega}\Big|\frac{\partial w}{\partial \nu_0}\Big|^2dS
+ \int_{\partial\Omega} |\nabla_{\text{tan}}w|^2 dS
+ \int_{\partial\Omega} |g|^2 dS\bigg\}.
\end{equation*}
From Gronwall's inequality, it follows that
\begin{equation}\label{f:3.24}
\begin{aligned}
\max_{0\leq t\leq T}\|w\|_{L^2(\Omega)}
+ \|\nabla w\|_{L^2(\Omega_T)}
&\leq C\Big\{\big\|\frac{\partial w}{\partial \nu_0}\big\|_{L^2(S_T)}
+\big\|\nabla_{\text{tan}}w\big\|_{L^2(S_T)}
+\big\|g\big\|_{L^2(S_T)}\Big\} \\
&\leq C\|w\|_{H^{1,1/2}(S_T)}
= C\|g\|_{H^{1,1/2}(S_T)}
\end{aligned}
\end{equation}
where we use the fact that $\|(\partial w/\partial\nu_0)\|_{L^2(S_T)} + \|w\|_{L^2(S_T)}
\thickapprox \|w\|_{H^{1,1/2}(S_T)}$ (see \cite[Lemma 4.3.13]{SZW22}) in the second step.

Consequently, the desired estimate $\eqref{pri:3.5}$ follows from the estimates
$\eqref{f:3.23}$ and $\eqref{f:3.24}$, and we completed the proof.
\end{proof}

\begin{lemma}\label{lemma:3.2}
Suppose that $A$ satisfies $\eqref{c:1}$ and $\eqref{c:2}$. Assume
$F\in (L^2(\Omega_T))^d$, $g\in ({_0H^{1,1/2}(S_T)})^d$ and $h\in (H^1_0(\Omega))^d$. Let $u_0\in L^2(0,T;(H^1(\Omega))^d)$ be
the weak solution of $(\emph{DP}_0)$. Then we have the lateral-layer type estimate
\begin{equation}\label{pri:3.1}
\|\nabla u_0\|_{L^2((\Omega\setminus\Sigma_{2\varepsilon})_T)}
\leq C\varepsilon^{1/2}\Big\{\|F\|_{L^2(\Omega_T)}+\|g\|_{H^{1,1/2}(S_T)}
+ \|h\|_{H^1(\Omega)}\Big\},
\end{equation}
and time-layer type estimate
\begin{equation}\label{pri:3.2}
\sup_{\varepsilon^2<t<T}\Big(\int_{t-\varepsilon^2}^t
\int_{\Sigma_{2\varepsilon}}|\nabla u_0|^2 dxds\Big)^{\frac{1}{2}}
\leq
C\varepsilon^{1/2}\Big\{\|F\|_{L^2(\Omega_T)}+\|g\|_{H^{1,1/2}(S_T)}
+ \|h\|_{H^1(\Omega)}\Big\},
\end{equation}
and co-layer type estimates
\begin{equation}\label{pri:3.3}
\max\Big\{\|\nabla^2 u_0\|_{L^2(\Sigma_{4\varepsilon^2,2\varepsilon}^T)},
~\|\partial_t u_0\|_{L^2(\Sigma_{4\varepsilon^2,2\varepsilon}^T)}\Big\}
\leq C\varepsilon^{-1/2}\Big\{\|F\|_{L^2(\Omega_T)}+\|g\|_{H^{1,1/2}(S_T)}
+ \|h\|_{H^1(\Omega)}\Big\},
\end{equation}
where $C$ depends at most on $\mu_1,\mu_2,d,T$ and $\Omega$.
\end{lemma}

\begin{proof}
The main idea in the proofs is analogous to that applied to elliptic operators, and we refer the
reader to \cite{SZW12,QX2} for the original idea.
We first address the spatial-layer type estimate
$\eqref{pri:3.1}$. Due to the linearity of the equation $(\text{DP}_0)$, we may divide
the solution $u_0$ into three parts, which means $u_0 = v+w+z$ and $v,w,z$ satisfy
the following equations (i), (ii), (iii), respectively.
\begin{equation*}
(\text{i})\left\{\begin{aligned}
\partial_t v + \mathcal{L}_0(v) &= \tilde{F} &~&\text{in}~ \tilde{\Omega}_T,\\
v & = 0 &~& \text{on}~ \tilde{S}_T,\\
v & = \tilde{h} &~& \text{on}~ \tilde{\Omega}\times\{t=0\},
\end{aligned}\right.\\
\end{equation*}
where $\tilde{\Omega}_T = \tilde{\Omega}\times(0,T]$, and $\tilde{\Omega}\supsetneqq \Omega$
is a new domain with $C^2$ boundary, and $\tilde{S}_T=\partial\tilde{\Omega}\times(0,T]$ denotes
the lateral of $\tilde{\Omega}_T$.
Here $\tilde{F}$ is a 0-extension to $\tilde{\Omega}$ such that
$\tilde{F} = F$ in $\Omega_T$ and $\tilde{F} =0$ in $\tilde{\Omega}_T\setminus\Omega_T$,
while $\tilde{h}$ is a $H^1_0$-extension to $\tilde{\Omega}$ satisfying
$\tilde{h}=h$ in $\Omega$ and
$\|\tilde{h}\|_{H^1_0(\tilde{\Omega})}\leq C\|h\|_{H^1_0(\Omega)}$.
\begin{equation*}
(\text{ii}) \left\{\begin{aligned}
\partial_t w + \mathcal{L}_0(w) &= 0 &\quad&\text{in}\quad \Omega_T,\\
w & = g &~& \text{on}\quad S_T,\\
w & = 0 &~& \text{on}\quad \Omega\times\{t=0\},
\end{aligned}\right.
\qquad
(\text{iii})
\left\{\begin{aligned}
\partial_t z + \mathcal{L}_0(z) &= 0 &\quad&\text{in}\quad \Omega_T,\\
z & = -v &~& \text{on}\quad S_T,\\
z & = 0 &~& \text{on}\quad \Omega\times\{t=0\}.
\end{aligned}\right.
\end{equation*}
We note that the existences of $w$ and $v$ have be shown in \cite[Theorem 4.2.1]{SZW22}, and the
second equality in (ii), as well as in (iii), should be understood in the sense of nontangential convergence.

Concerning the equation (i),
it follows from the regularity of initial-Dirichlet problem (see \cite[Section 7.1.3]{LCE}) that
\begin{equation*}
\sup_{0\leq t\leq T}\int_{\tilde{\Omega}}|\nabla v|^2 dx
+\int_0^T\int_{\tilde{\Omega}}\big(|\partial_t v|^2+|\nabla^2 v|^2\big) dxdt
\leq C\bigg\{\int_0^T\int_{\tilde{\Omega}}|\tilde{F}|^2 dxdt
+ \int_{\tilde{\Omega}}|\nabla \tilde{h}|^2 dx\bigg\}
\end{equation*}
and this together with the energy estimate
\begin{equation*}
\sup_{0\leq t\leq T}\int_{\tilde{\Omega}}| v|^2 dx
+
\int_0^T\int_{\tilde{\Omega}}|\nabla v|^2 dxdt
\leq C\bigg\{\int_0^T\int_{\tilde{\Omega}}|\tilde{F}|^2 dxdt
+ \int_{\tilde{\Omega}}|\tilde{h}|^2 dx\bigg\}
\end{equation*}
gives the following estimate
\begin{equation}\label{f:3.7}
\begin{aligned}
&\|v\|_{W^{2,1}_2(\Omega_T)}
\leq C\Big\{\|\tilde{F}\|_{L^2(\tilde{\Omega}_T)}+\|\tilde{h}\|_{H^1(\tilde{\Omega})}\Big\}
\leq C\Big\{\|F\|_{L^2(\Omega_T)}+\|h\|_{H^1(\Omega)}\Big\},\\
&\qquad\qquad\sup_{0\leq t\leq T}\int_{\Omega}
\big(|\nabla v|^2 +|v|^2\big) dx
\leq C\Big\{\|F\|_{L^2(\Omega_T)}^2+\|h\|_{H^1(\Omega)}^2\Big\},
\end{aligned}
\end{equation}
where $C$ depends at most on $\mu,d,m, T,\Omega$ and $\tilde{\Omega}$.
Recalling the definition of $S_r$, for any $r\in[0,c_0]$ and $t>0$, it follows from the trace theorem that
\begin{equation*}
\begin{aligned}
\int_{S_r}|\nabla v(\cdot,t)|^2dS
&\leq C\bigg\{\int_{\Sigma_r}|\nabla^2 v(\cdot,t)|^2 dx
+ \int_{\Sigma_r}|\nabla v(\cdot,t)|^2 dx\bigg\} \\
&\leq  C\bigg\{\int_{\Omega}|\nabla^2 v(\cdot,t)|^2 dx
+ \int_{\Omega}|\nabla v(\cdot,t)|^2 dx\bigg\},
\end{aligned}
\end{equation*}
where $C$ is independent of $r$ and $t$. By the co-area formula, we have
\begin{equation*}
\int_{\Omega\setminus\Sigma_{2\varepsilon}} |\nabla v(\cdot,t)|^2 dx
=\int_0^{2\varepsilon}\int_{S_r}|\nabla v(\cdot,t)|^2 dS_r dr
\leq C\varepsilon\bigg\{\int_{\Omega}|\nabla^2 v(\cdot,t)|^2 dx
+ \int_{\Omega}|\nabla v(\cdot,t)|^2 dx\bigg\}.
\end{equation*}
Integrating from $0$ to $T$ with respect to the time variable and then
taking the square root, it holds
\begin{equation}\label{f:3.10}
\begin{aligned}
\Big(\int_0^T\int_{\Omega\setminus\Sigma_{2\varepsilon}} |\nabla v|^2 dxdt\Big)^{\frac{1}{2}}
&\leq C\varepsilon^{1/2}\big\|\nabla v\big\|_{L^{2}(0,T;H^1(\Omega))} \\
&\leq C\varepsilon^{1/2}
\Big\{\|F\|_{L^2(\Omega_T)}+\|h\|_{H^1(\Omega)}\Big\},
\end{aligned}
\end{equation}
where we use the estimate $\eqref{f:3.7}$ in the last inequality.

We now proceed to study the equations (ii) and (iii).
Due to the work of Z. Shen's, it is well known that
\begin{equation}\label{f:3.8}
 \big\|\frac{\partial z}{\partial \nu_0}\big\|_{L^2(S_T)} + \|z\|_{L^2(S_T)}
 \thickapprox
 \big\|z\|_{H^{1,1/2}(S_T)}
\end{equation}
(see \cite[Lemma 4.3.13]{SZW22}), which may be derived from the so-called Rellich identity. By
the way, the original work
in the case of $\mathcal{L}_0 = -\Delta$
traces back to R. Brown (see \cite[Section 3]{RB}).
Hence, on account of \cite[Theorem 4.2.1]{SZW22} we have
\begin{equation}\label{f:3.9}
\begin{aligned}
\|(\nabla z)^*\|_{L^2(S_T)}
&+ \|(z)^*\|_{L^2(S_T)}
+ \|(\nabla w)^*\|_{L^2(S_T)}
+ \|(w)^*\|_{L^2(S_T)} \\
&\leq C\Big\{\|z\|_{H^{1,1/2}(S_T)}+\|g\|_{H^{1,1/2}(S_T)}\Big\}\\
&\leq C\Big\{\|\nabla v\|_{L^{2}(S_T)}
+ \|v\|_{L^{2}(S_T)}  +\|g\|_{H^{1,1/2}(S_T)}\Big\},
\end{aligned}
\end{equation}
where we employ the equivalence $\eqref{f:3.8}$ in the second inequality.
From the trace theorem in space variable, it is not hard to see that
\begin{equation*}
\begin{aligned}
\|\nabla v\|_{L^{2}(S_T)} + \|v\|_{L^{2}(S_T)}
&\leq C\Big\{\|\nabla v\|_{L^{2}(0,T;H^1(\Omega))} + \|v\|_{L^2(\Omega_T)}\Big\} \\
&\leq C\Big\{\|\nabla v\|_{L^{2}(0,T;H^1(\Omega))}
+ \sup_{0\leq t\leq T}\|v\|_{L^2(\Omega)}\Big\}
\leq C\Big\{\|F\|_{L^2(\Omega_T)}+\|h\|_{H^1(\Omega)}\Big\},
\end{aligned}
\end{equation*}
where we use the estimate $\eqref{f:3.7}$ in the last step. Thus, the third line of
$\eqref{f:3.9}$ will be controlled by
\begin{equation*}
C\Big\{\|F\|_{L^2(\Omega_T)}+\|h\|_{H^1(\Omega)} + \|g\|_{H^{1,1/2}(S_T)}\Big\}.
\end{equation*}

For the ease of the statement, let $\bar{w} = w+z$. Then we have known $u=v+\bar{w}$ and
\begin{equation}\label{f:3.13}
\|(\nabla \bar{w})^*\|_{L^2(S_T)}+\|(\bar{w})^*\|_{L^2(S_T)}
\leq C\Big\{\|F\|_{L^2(\Omega_T)}^2+\|h\|_{H^1(\Omega)}^2 + \|g\|_{H^{1,1/2}(S_T)}\Big\},
\end{equation}
and this implies
\begin{equation}\label{f:3.11}
\begin{aligned}
\Big(\int_0^T\int_{\Omega\setminus\Sigma_{2\varepsilon}}
|\nabla \bar{w}|^2 dxdt\Big)^{\frac{1}{2}}
&\leq C\varepsilon^{1/2}\big\|(\nabla\bar{w})^*\big\|_{L^2(S_T)} \\
&\leq C\varepsilon^{1/2}
\Big\{\|F\|_{L^2(\Omega_T)}+\|h\|_{H^1(\Omega)}
+ \|g\|_{H^{1,1/2}(S_T)}\Big\}.
\end{aligned}
\end{equation}

Thus, combining the estimates
$\eqref{f:3.10}$ and $\eqref{f:3.11}$,
we have proved the spatial-layer type estimate $\eqref{pri:3.1}$.

We now turn to investigate the time-layer type estimate $\eqref{pri:3.2}$.
Due to the estimate $\eqref{pri:3.4}$, the problem is reduced to estimate the following terms:
\begin{equation}\label{f:3.12}
\int_{t-\varepsilon^2}^t\int_{\Sigma_{2\varepsilon}\setminus\Sigma_{4\varepsilon}}|u_0|^2 dxds,
\qquad
\int_{t-\varepsilon^2}^t\int_{\Sigma_{2\varepsilon}}
|u_0|^2 dxds
\quad
\text{and}\quad
\int_{t-\varepsilon^2}^t\int_{\Sigma_{2\varepsilon}}
|\partial_t u_0|^2 dxds.
\end{equation}

The easiest one is
\begin{equation}\label{f:3.16}
\int_{t-\varepsilon^2}^t\int_{\Sigma_{2\varepsilon}}
|u_0|^2 dxds
\leq C\varepsilon^2\Big\{\|F\|_{L^2(\Omega_T)}+\|h\|_{H^1(\Omega)}
+ \|g\|_{H^{1,1/2}(S_T)}\Big\}
\end{equation}
where we employ the estimate $\eqref{pri:3.5}$. Then we address the first term
in $\eqref{f:3.12}$. Recalling $u_0 = v+\bar{w}$, there exists
$\xi\in(t-\varepsilon,t)$ such that
\begin{equation*}
\begin{aligned}
\int_{t-\varepsilon^2}^t\int_{\Sigma_{2\varepsilon}\setminus\Sigma_{4\varepsilon}}|u_0|^2 dxds
&\leq \int_{t-\varepsilon^2}^t\int_{\Omega\setminus\Sigma_{4\varepsilon}}|v|^2 dxds
+  \int_{t-\varepsilon^2}^t\int_{\Sigma_{2\varepsilon}\setminus\Sigma_{4\varepsilon}}
|\bar{w}|^2 dxds \\
&\leq C\varepsilon\int_{t-\varepsilon^2}^t\int_{\Omega}(|v|^2+|\nabla v|^2) dxds
+ C\varepsilon^3 \int_{\partial\Omega}|(\bar{w})^*(\cdot,\xi)|^2 dS \\
&\leq C\varepsilon^3\bigg\{\sup_{0\leq t\leq T}
\int_{\Omega}\big(|v|^2+|\nabla v|^2\big)dx
+ \int_{\partial\Omega}|(\bar{w})^*(\cdot,\xi)|^2dS\bigg\},
\end{aligned}
\end{equation*}
where we use the trace theorem for $v$ and the definition of the maximal function of $\bar{w}$
in the second inequality. Then integrating both sides of the above inequality with respect to
$\xi$ from $0$ to $T$, we have
\begin{equation}\label{f:3.17}
\begin{aligned}
\int_{t-\varepsilon^2}^t\int_{\Sigma_{2\varepsilon}\setminus\Sigma_{4\varepsilon}}|u_0|^2 dxds
\leq C\varepsilon^3
\Big\{\|F\|_{L^2(\Omega_T)}+\|h\|_{H^1(\Omega)}
+ \|g\|_{H^{1,1/2}(S_T)}\Big\},
\end{aligned}
\end{equation}
where we use the estimates $\eqref{f:3.7}$ and $\eqref{f:3.13}$.

Now, we focus on the last term in $\eqref{f:3.12}$. Noting that
$\partial_t u_0 + \mathcal{L}_0(u_0) = F$ in $\Omega_T$, we acquire
\begin{equation}\label{f:3.14}
\int_{t-\varepsilon^2}^t\int_{\Sigma_{2\varepsilon}}
|\partial_t u_0|^2 dxds
\leq C\bigg\{\int_{t-\varepsilon^2}^t\int_{\Sigma_{2\varepsilon}}
|\nabla^2 u_0|^2 dxds
+ \int_{0}^T\int_{\Omega}
|F|^2 dxds
\bigg\}.
\end{equation}
It suffices to estimate the first term in the right-hand side of $\eqref{f:3.14}$. Since using
$\eqref{f:3.7}$ leads to
\begin{equation}\label{f:3.15}
\begin{aligned}
\int_{t-\varepsilon^2}^t\int_{\Sigma_{2\varepsilon}}
|\nabla^2 u_0|^2 dxds
&\leq \int_{0}^T\int_{\Omega}
|\nabla^2 v|^2 dxds
+ \int_{t-\varepsilon^2}^t\int_{\Sigma_{2\varepsilon}}
|\nabla^2 \bar{w}|^2 dxds \\
&\leq C\Big\{\|F\|_{L^2(\Omega_T)}+\|h\|_{H^1(\Omega)}\Big\}
+\int_{t-\varepsilon^2}^t\int_{\Sigma_{2\varepsilon}}
|\nabla^2 \bar{w}|^2 dxds,
\end{aligned}
\end{equation}
the problem is reduced to estimate the last term of $\eqref{f:3.15}$.
And we obtain
\begin{equation*}
\begin{aligned}
\int_{t-\varepsilon^2}^t\int_{\Sigma_{2\varepsilon}}
|\nabla^2 \bar{w}|^2 dxds
&\leq \int_{t-\varepsilon^2}^t\int_{\Sigma_{2\varepsilon}\setminus\Sigma_{c_0}}
|\nabla^2 \bar{w}|^2 dxds
+ \int_{0}^{T}\int_{\Sigma_{c_0}}
|\nabla^2 \bar{w}|^2 dxds\\
&\leq C\bigg\{\varepsilon^{-2}
\int_{t-\frac{3}{2}\varepsilon^2}^{t+\frac{1}{2}\varepsilon^2}
\int_{\Sigma_{\varepsilon}\setminus\Sigma_{2c_0}}
|\nabla \bar{w}|^2 dxds
+ \int_{0}^T\int_{\Omega}
|\nabla \bar{w}|^2 dxds\bigg\}\\
&\leq C\bigg\{
\int_{\partial\Omega}
|(\nabla \bar{w})^*(\cdot,\xi)|^2 dS
+ \|g\|_{H^{1,1/2}(S_T)}^2\bigg\}
\end{aligned}
\end{equation*}
where $\xi\in (t-\varepsilon^2,t)$, and we use the interior estimates \cite[Lemma 1]{SchW} and $\eqref{pri:3.6}$ in the second step,
and the last one follows from the definition of nontangential maximal function
and the estimate $\eqref{pri:3.5}$.
By integrating with respect to $\xi$ from $0$ to $T$, it is shown that
\begin{equation*}
\begin{aligned}
\int_{t-\varepsilon^2}^t\int_{\Sigma_{2\varepsilon}}
|\nabla^2 \bar{w}|^2 dxds
\leq C\bigg\{
\int_0^T\int_{\partial\Omega}
|(\nabla \bar{w})^*(\cdot,\xi)|^2 dSd\xi
+ \|g\|_{H^{1,1/2}(S_T)}^2\bigg\}
\leq C\|g\|_{H^{1,1/2}(S_T)}^2,
\end{aligned}
\end{equation*}
and this together with the estimates $\eqref{f:3.14}$ and $\eqref{f:3.15}$ gives
\begin{equation}\label{f:3.18}
\int_{t-\varepsilon^2}^t\int_{\Sigma_{2\varepsilon}}
|\partial_t u_0|^2 dxds
\leq C\Big\{\|F\|_{L^2(\Omega_T)}+\|h\|_{H^1(\Omega)}
+ \|g\|_{H^{1,1/2}(S_T)}\Big\}.
\end{equation}

Hence, plugging the estimates $\eqref{f:3.16}, \eqref{f:3.17}$ and \eqref{f:3.18} back into
the estimate $\eqref{pri:3.4}$, we obtain
\begin{equation}
\sup_{\varepsilon^2< t< T}\int_{t-\varepsilon^2}^t
\int_{\Sigma_{2\varepsilon}}|\nabla u_0|^2 dx ds
\leq C\varepsilon\Big\{\|F\|_{L^2(\Omega_T)}+\|h\|_{H^1(\Omega)}
+ \|g\|_{H^{1,1/2}(S_T)}\Big\}
\end{equation}
and this verifies the desired time-layer-type estimate $\eqref{pri:3.2}$.

The rest of the proof is devoted to the so-called co-layer type estimate $\eqref{pri:3.3}$.
Since $\partial_t u_0 + \mathcal{L}_0(u_0) = F$ in $\Omega_T$, it is sufficient to prove the
estimate $\eqref{pri:3.3}$ for the quantity
\begin{equation*}
\int_{4\varepsilon^2}^{T-4\varepsilon^2}\int_{\Sigma_{2\varepsilon}}
|\nabla^2 u_0|^2 dxdt.
\end{equation*}
In view of $u_0 = v+\bar{w}$, the above one could be controlled by
\begin{equation}\label{f:3.19}
\begin{aligned}
&\qquad\qquad\qquad\qquad\qquad \int_{0}^{T}\int_{\tilde{\Omega}}
|\nabla^2 v|^2 dxdt
+ \int_{4\varepsilon^2}^{T-4\varepsilon^2}\int_{\Sigma_{2\varepsilon}}
|\nabla^2 \bar{w}|^2 dxdt\\
&\leq C\Big\{\|F\|_{L^2(\Omega_T)}+\|h\|_{H^1(\Omega)}\Big\}
+\underbrace{
\int_{4\varepsilon^2}^{T-4\varepsilon^2}\int_{\Sigma_{2\varepsilon}\setminus\Sigma_{c_0}}
|\nabla^2 \bar{w}|^2 dxdt}_{I_1}
+\underbrace{\int_{0}^{T}\int_{\Sigma_{c_0}}
|\nabla^2 \bar{w}|^2 dxdt}_{I_2}
\end{aligned}
\end{equation}
where we use the estimate $\eqref{f:3.7}$. From the interior estimate
$\eqref{pri:3.6}$
and the global estimate $\eqref{pri:3.5}$, it follows that
\begin{equation}\label{f:3.20}
I_2 \leq C\int_{0}^{T}\int_{\Omega}
|\nabla \bar{w}|^2 dxdt \leq C\|g\|_{H^{1,1/2}(S_T)}.
\end{equation}
We proceed to estimate $I_1$ by the method analogous to that used above.
However, we have to, in advance, remove one more order derivative
from $\nabla^2\bar{w}$, carefully. Due to the interior regularity estimate \cite[Lemma 1]{SchW} and
the Sobolev imbedding theorem (see \cite[pp.1148-1149]{SchW}),
there holds
\begin{equation}\label{f:3.30}
\big|\nabla^2 \bar{w}(X)\big|^2 \leq
\frac{C}{[\sigma(X)]^2}\dashint_{P(X,\sigma(X)/4)}|\nabla \bar{w}(Y)|^2 dY
\end{equation}
for any $X=(x,t)\in (\Sigma_{2\varepsilon}\setminus\Sigma_{c_0})\times(4\varepsilon^2,T-4\varepsilon^2)$,
where $\sigma(X) = \text{dist}(X,S_T)$.
We remark that the existence of $\bar{w}$ in fact comes from layer potential theory
concerning parabolic equations, the key idea from R. Brown \cite{RB} is that extending $\bar{w}$ to
a caloric function which is still caloric on $\Omega\times\mathbb{R}$. Roughly speaking,
since the estimate $\eqref{f:3.30}$ is just an interior estimate and the extension of $\bar{w}$
is still determined by given data $F,g$ and $h$,
here we may regard $\bar{w}$ as being a solution of $\partial_t\bar{w} + \mathcal{L}_0(\bar{w}) = 0$ in $\Omega\times \mathbb{R}$.

Hence, in view of the co-area formula, we obtain
\begin{equation}\label{f:3.21}
\begin{aligned}
I_2
&\leq C\int_{4\varepsilon^2}^{T-4\varepsilon}
\int_{\Sigma_{2\varepsilon}\setminus\Sigma{c_0}}
\frac{1}{[\sigma(X)]^2}
\dashint_{P(Y,\delta(P)/4)}|\nabla\bar{w}|^2 dY dX \\
&\leq C\int_0^T\int_{\partial\Omega}|(\nabla\bar{w})^*(\cdot,t)|^2 dSdt
\int_{2\varepsilon}^{c_0}\frac{dr}{r^2}\\
&\leq \frac{C}{\varepsilon}\int_{S_T}|(\nabla\bar{w})^*|^2 dSdt
\leq C\varepsilon^{-1}\Big\{
\|F\|_{L^2(\Omega_T)}
+ \|h\|_{H^1(\Omega)}+ \|g\|_{H^{1,1/2}(S_T)}^2\Big\},
\end{aligned}
\end{equation}
where we use the estimate $\eqref{f:3.13}$ in the last step. Combining the estimates
$\eqref{f:3.19}$, $\eqref{f:3.20}$ and $\eqref{f:3.21}$ gives
\begin{equation*}
 \bigg(\int_{4\varepsilon^2}^{T-4\varepsilon^2}
 \int_{\Sigma_{2\varepsilon}}|\nabla^2 u_0|^2 dxds\bigg)^{1/2}
 \leq C\varepsilon^{-1/2}\Big\{\|F\|_{L^2(\Omega_T)}
+ \|h\|_{H^1(\Omega)} +\|g\|_{H^{1,\frac{1}{2}}(S_T)} \Big\}.
\end{equation*}

   Until now, we have proved the desired estimate $\eqref{pri:3.3}$, and the whole proof is
complete.
\end{proof}

\begin{flushleft}
\textbf{Proof of Theorem $\ref{thm:3.1}$.}
The desired estimate $\eqref{pri:3.0}$ follows from
Lemmas $\ref{lemma:3.1}$ and $\ref{lemma:3.2}$.
\qed
\end{flushleft}

\section{Convergence rates in $L^2(\Omega_T)$}\label{section:4}

In order to accelerate the convergence rate, we shall employ the so-called duality methods. To do so, we first consider the adjoint initial-Dirichlet problems: given $\Phi\in (L^2(\Omega_T))^d$, let
$\phi_\varepsilon$ and $\phi_0$ be the weak solution to
\begin{equation*}
(\text{DP}_\varepsilon^*)\left\{\begin{aligned}
-\frac{\partial \phi_\varepsilon}{\partial t} + \mathcal{L}_\varepsilon^*(\phi_\varepsilon)
& = \Phi &~&\text{in}~ ~\Omega_T,\\
\phi_\varepsilon & = 0 &~&\text{on}~ S_T,\\
\phi_\varepsilon & = 0 &~&\text{on}~ \Omega\times\{t=T\},
\end{aligned}\right.
\quad
(\text{DP}_0^*)\left\{\begin{aligned}
-\frac{\partial \phi_0}{\partial t} + \mathcal{L}_\varepsilon^*(\phi_0)
& = \Phi &~&\text{in}~ ~\Omega_T,\\
\phi_0 & = 0 &~&\text{on}~ S_T, \\
\phi_0 & = 0 &~&\text{on}~ \Omega\times\{t=T\},
\end{aligned}\right.
\end{equation*}
respectively. Here $\mathcal{L}_\varepsilon^*$ is the adjoint operator of $\mathcal{L}_\varepsilon$.
By the symmetry condition $\eqref{c:2}$, let
$\tilde{\phi}_\varepsilon(x,t)=\phi_\varepsilon(x,T-t)$ and
$\tilde{\phi}_0(x,t)=\phi_0(x,T-t)$, which exactly right solve the initial-Dirichlet problems
\begin{equation*}
\left\{\begin{aligned}
\frac{\partial \tilde{\phi}_\varepsilon^\alpha}{\partial t}
-\frac{\partial}{\partial x_i}\Big\{a_{ij}^{\alpha\beta}\Big(\frac{x}{\varepsilon},
\frac{T-t}{\varepsilon^2}\Big)\frac{\partial \tilde{\phi}_\varepsilon^\beta}{\partial x_j}\Big\}
& = \Phi^\alpha &~&\text{in}~ ~\Omega_T,\\
\tilde{\phi}_\varepsilon & = 0 &~&\text{on}~ \partial_p\Omega_T
\end{aligned}\right.
\quad
\text{and}
\quad
\left\{\begin{aligned}
\frac{\partial \tilde{\phi}_0}{\partial t} + \mathcal{L}_0(\tilde{\phi}_0)
& = \Phi &~&\text{in}~ ~\Omega_T,\\
\tilde{\phi}_0 & = 0 &~&\text{on}~ \partial_p\Omega_T.
\end{aligned}\right.
\end{equation*}
We mention that $\chi_T^*$ and
$E_{T,l(d+1)j}^{*}$ are corresponding correctors and dual correctors
associated with $A(x/\varepsilon,(T-t)/\varepsilon^2)$, respectively.
It is clear to see that the adjoint problems will obey Theorems $\ref{thm:2.1}$, $\ref{thm:3.1}$ as well.

\begin{lemma}[Duality lemma]\label{lemma:4.3}
Let $w_\varepsilon$ be given in $\eqref{pri:3.1}$ by choosing
$\varphi_j^\gamma = S_\varepsilon
K_\varepsilon(\Psi_{[4\varepsilon^2,2\varepsilon]}\nabla_ju_0^\gamma)$,
where $u_\varepsilon,u_0$ are weak solution of $(\emph{PD}_\varepsilon)$ and
$(\emph{PD}_0)$, respectively. For any $\Phi\in (L^2(\Omega_T))^d$, we assume that
$\phi_\varepsilon,\phi_0\in L^2(0,T;(H^1_0(\Omega))^d)$ with
$\partial_t\phi_\varepsilon,\partial_t\phi_0\in L^2(0,T;(H^{-1}(\Omega))^d)$ are the weak solutions to
the adjoint problems $(\emph{PD}_\varepsilon^*)$ and $(\emph{PD}_0^*)$, respectively. Then we have
\begin{equation}\label{eq:4.1}
\int_{\Omega_T} w_\varepsilon \Phi dx dt
= -\int_{\Omega_T} \tilde{f}\cdot\nabla\phi_\varepsilon dx dt,
\end{equation}
where $\tilde{f}$ is shown in $\eqref{eq:3.3}$. Moreover, if we assume
\begin{equation}\label{eq:4.2}
\begin{aligned}
&\breve{w}_\varepsilon(x,t) =
\tilde{\phi}_\varepsilon - \tilde{\phi}_0
-\varepsilon\chi^{*}_{T}(x/\varepsilon,t/\varepsilon^2)S_\varepsilon K_\varepsilon(
\Psi_{[100\varepsilon^2,10\varepsilon]}\nabla\tilde{\phi}_0) \\
&\qquad\qquad\qquad
-\varepsilon^2 E_{T,l(d+1)j}^{*}(x/\varepsilon,t/\varepsilon^2)
\frac{\partial}{\partial x_l}
S_\varepsilon K_\varepsilon(
\Psi_{[100\varepsilon^2,10\varepsilon]}\nabla_j\tilde{\phi}_0),
\end{aligned}
\end{equation}
then there holds
\begin{equation}\label{pri:4.1}
\begin{aligned}
\bigg|\int_{\Omega_T} w_\varepsilon
&\Phi dxdt\bigg|
\leq C\bigg\{\|\nabla u_0\|_{L^2(\boxbox_{2\varepsilon})}
+\varepsilon\|\nabla^2 u_0\|_{L^2(\Sigma_{4\varepsilon^2,2\varepsilon}^T)}
+\varepsilon\|\partial_t u_0\|_{L^2(\Sigma_{4\varepsilon^2,2\varepsilon}^T)}\bigg\}\\
&\times
\bigg\{\|\nabla\breve{w}_\varepsilon\|_{L^2(\Omega_T)}
+ \|\nabla\phi_0\|_{L^2(\boxbox_{2\varepsilon})}
+ \varepsilon\|\nabla^2\phi_0\|_{L^2(\Sigma_{4\varepsilon^2,2\varepsilon}^T)}\bigg\}\\
& + C\big\|\nabla\phi_0\big\|_{L^2(\Sigma_{4\varepsilon^2,2\varepsilon}^T;\delta^{-1})}
 \Bigg\{
\|\nabla u_0\|_{L^2(\boxbox_{2\varepsilon};\delta)}
+\varepsilon\|\nabla^2 u_0\|_{L^2(\Sigma_{4\varepsilon^2,2\varepsilon}^T;\delta)}
+\varepsilon\|\partial_t u_0\|_{L^2(\Sigma_{4\varepsilon^2,2\varepsilon}^T;\delta)}\Bigg\},\\
\end{aligned}
\end{equation}
where $C$ depends on
$\mu_1,\mu_2,d,T$ and $\Omega$.
\end{lemma}

\begin{proof}
First, it is not hard to see that the equality $\eqref{eq:4.1}$ follows from integrating by parts
\begin{equation*}
\begin{aligned}
\int_{\Omega_T}w_\varepsilon \Phi dxdt
&= \int_0^T \big<w_\varepsilon,(-\partial_t+\mathcal{L}_\varepsilon^*)\phi_\varepsilon\big>dt \\
&= \int_0^T \big<(\partial_t+\mathcal{L}_\varepsilon)w_\varepsilon,\phi_\varepsilon\big>dt
+ \int_\Omega w_\varepsilon(x,T)\phi_\varepsilon(x,T)dx
- \int_\Omega w_\varepsilon(x,0)\phi_\varepsilon(x,0)dx \\
&= - \int_{\Omega_T}\tilde{f}\cdot\nabla\phi_\varepsilon dxdt,
\end{aligned}
\end{equation*}
where $w_\varepsilon,\phi_\varepsilon$ are weak solutions of $\eqref{eq:3.2}$ and $(\text{DP}_\varepsilon^*)$, respectively, and we employ the initial-boundary conditions
$w_\varepsilon = \phi_\varepsilon = 0$ on $S_T$ in the second step, and
$w_\varepsilon(x,0) = \phi_\varepsilon(x,T) = 0$ in the last one.

Let $\varpi$ denote its periodic parts for simplicity of
presentation. Thus, observing $\eqref{eq:3.3}$ we have
\begin{equation}
\begin{aligned}
\bigg|\int_{\Omega_T}w_\varepsilon \Phi dxdt\bigg|
&\leq C\int_{\Omega_T}
\big|\nabla u_0 - S_\varepsilon K_\varepsilon(\Psi_{[4\varepsilon^2,2\varepsilon]}\nabla u_0)\big|
\big|\nabla\phi_\varepsilon\big| dxdt\\
& +\varepsilon\int_{\Omega_T}
\big|\varpi(x/\varepsilon,t/\varepsilon^2)
\nabla S_\varepsilon K_\varepsilon(\Psi_{[4\varepsilon^2,2\varepsilon]}\nabla u_0)\big|
\big|\nabla\phi_\varepsilon\big| dxdt \\
& +\varepsilon^2
\int_{\Omega_T}
\big|\varpi(x/\varepsilon,t/\varepsilon^2)
\nabla^2 S_\varepsilon K_\varepsilon(\Psi_{[4\varepsilon^2,2\varepsilon]}\nabla u_0)\big|
\big|\nabla\phi_\varepsilon\big| dxdt \\
& +\varepsilon^2
\int_{\Omega_T}
\big|\varpi(x/\varepsilon,t/\varepsilon^2)
\partial_t S_\varepsilon K_\varepsilon(\Psi_{[4\varepsilon^2,2\varepsilon]}\nabla u_0)\big|
\big|\nabla\phi_\varepsilon\big| dxdt
=: I_1 + I_2 + I_3 + I_4.
\end{aligned}
\end{equation}

Before proceeding further, we want to show the main ideas on accelerating the convergence rates. The key step is to replace $\phi_\varepsilon(x,t)$ by
\begin{equation}\label{f:4.1}
\begin{aligned}
 \breve{w}_\varepsilon(x,T&-t)
+ \phi_0(x,t)
+\varepsilon\chi_{T}^{*}(x/\varepsilon,(T-t)/\varepsilon^2)
S_\varepsilon K_\varepsilon(\Psi_{[100\varepsilon^2,10\varepsilon]}\nabla\phi_0)(x,t) \\
& + \varepsilon^2 E_{T,(d+1)}^{*}(x/\varepsilon,(T-t)/\varepsilon^2)
\nabla S_\varepsilon K_\varepsilon(\Psi_{[100\varepsilon^2,10\varepsilon]}\nabla\phi_0)(x,t).
\end{aligned}
\end{equation}
Here
$\breve{w}$ is given by $\eqref{eq:4.2}$, and it follows from Theorem $\ref{thm:3.1}$ that
\begin{equation}\label{pri:4.5}
\|\breve{w}\|_{L^2(\Omega_T)}\leq C\varepsilon^{1/2}\|\Phi\|_{L^2(\Omega_T)}.
\end{equation}

Observing $\eqref{f:4.1}$ again, all the terms will produce $O(\varepsilon^{1/2})$
except for the second term $\phi_0$ in a co-layer type estimate,
and this naturally arouse the distance function playing a role as a weight function in the following calculation.

To estimate $I_1$, we divide it into two parts:
\begin{equation*}
\begin{aligned}
\underbrace{\int_{\Omega_T}
\big|\Psi_{[4\varepsilon^2,2\varepsilon]}
\nabla u_0 - S_\varepsilon K_\varepsilon(\Psi_{[4\varepsilon^2,2\varepsilon]}\nabla u_0)\big|
\big|\nabla\phi_\varepsilon\big| dxdt}_{I_{11}}
\text{\quad and\quad}\underbrace{\int_{\Omega_T}
(1-\Psi_{[4\varepsilon^2,2\varepsilon]})\big|\nabla u_0\big|
\big|\nabla\phi_\varepsilon\big| dxdt}_{I_{12}}.
\end{aligned}
\end{equation*}
We first handle $I_{12}$ as below
\begin{equation}\label{f:4.3}
\begin{aligned}
I_{12}
\leq C\|\nabla u_0\|_{L^2(\boxbox_{2\varepsilon})}
\|\nabla \phi_\varepsilon\|_{L^2(\boxbox_{2\varepsilon})}
\leq C\|\nabla u_0\|_{L^2(\boxbox_{2\varepsilon})}
\Big\{
\|\nabla \breve{w}_\varepsilon\|_{L^2(\Omega_T)}
+\|\nabla \phi_0\|_{L^2(\boxbox_{2\varepsilon})}
\Big\}
\end{aligned}
\end{equation}
where we replace $\phi_\varepsilon$ by $\eqref{f:4.1}$ in the last step, and use
the fact that the last two terms of $\eqref{f:4.1}$ vanish in $\boxbox_{2\varepsilon}$
since they are supported in
$\Sigma_{100\varepsilon^2,10\varepsilon}^T$. We then turn to study $I_{11}$. It also
decomposes into four parts:
\begin{equation}\label{f:4.2}
\begin{aligned}
\int_{\Omega_T}
\big|\Psi_{[4\varepsilon^2,2\varepsilon]}
\nabla u_0 - K_\varepsilon(\Psi_{[4\varepsilon^2,2\varepsilon]}\nabla u_0)\big|
&\Bigg\{\big|\nabla\breve{w}_\varepsilon(x,T-t)\big|\\
&+ \varepsilon\big|\varpi(x/\varepsilon,(T-t)/\varepsilon^2)
\nabla S_\varepsilon K_\varepsilon(\Psi_{[100\varepsilon^2,10\varepsilon]}\nabla\phi_0)\big|\\
&\underbrace{+\varepsilon^2\big|\varpi(x/\varepsilon,(T-t)/\varepsilon^2)
\nabla^2 S_\varepsilon K_\varepsilon(\Psi_{[100\varepsilon^2,10\varepsilon]}\nabla\phi_0)\big|
\Bigg\}}_{\mathcal{R}_{\varepsilon,1}}dxdt,
\end{aligned}
\end{equation}
\begin{equation}\label{f:4.5}
\begin{aligned}
\int_{\Omega_T}
\big|K_\varepsilon(\Psi_{[4\varepsilon^2,2\varepsilon]}
\nabla u_0) - S_\varepsilon K_\varepsilon(\Psi_{[4\varepsilon^2,2\varepsilon]}
\nabla u_0)\big| \mathcal{R}_{\varepsilon,1}(x,t)dxdt,
\end{aligned}
\end{equation}
and
\begin{equation}\label{f:4.4}
\begin{aligned}
\int_{\Omega_T}
&\bigg\{\big|\Psi_{[4\varepsilon^2,2\varepsilon]}
\nabla u_0 - S_\varepsilon(\Psi_{[4\varepsilon^2,2\varepsilon]}\nabla u_0)\big|
+ \big|K_\varepsilon(\Psi_{[4\varepsilon^2,2\varepsilon]}
\nabla u_0) - S_\varepsilon K_\varepsilon(\Psi_{[4\varepsilon^2,2\varepsilon]}\nabla u_0)\big|
\bigg\}\\
&\qquad\qquad\qquad\times\underbrace{\bigg\{\big|\nabla\phi_0\big|
+\big|\nabla\chi_T^*(x/\varepsilon,(T-t)/\varepsilon^2)
S_\varepsilon K_\varepsilon(\Psi_{[100\varepsilon^2,10\varepsilon]}\nabla\phi_0)\big|
\bigg\}}_{\mathcal{R}_{\varepsilon,2}} dxdt.
\end{aligned}
\end{equation}
By Cauchy's inequality, it is not hard to see that the expression $\eqref{f:4.2}$
is controlled by
\begin{equation*}
\begin{array}{l}
\big\|
\Psi_{[4\varepsilon^2,2\varepsilon]}
\nabla u_0 - K_\varepsilon(\Psi_{[4\varepsilon^2,2\varepsilon]}\nabla u_0)
\big\|_{L^2(\mathbb{R}^{d+1})}\\
\times\bigg\{
\|\nabla\breve{w}_\varepsilon\|_{L^2(\Omega_T)}
+ \varepsilon\|\varpi(y,\tau)
\nabla S_\varepsilon K_\varepsilon(\Psi_{[9\varepsilon^2,5\varepsilon]}\nabla\phi_0)
\|_{L^2(\mathbb{R}^{d+1})}
+ \varepsilon^2
\|\varpi(y,\tau)
\nabla^2 S_\varepsilon
K_\varepsilon(\Psi_{[100\varepsilon^2,10\varepsilon]}\nabla\phi_0)\|_{L^2(\mathbb{R}^{d+1})}
\bigg\}\\
\leq C\varepsilon\|\nabla(\Psi_{[4\varepsilon^2,2\varepsilon]}
\nabla u_0)\|_{L^2(\mathbb{R}^{d+1})}
\bigg\{
\|\nabla\breve{w}_\varepsilon\|_{L^2(\Omega_T)}
+ \varepsilon\|\nabla(\Psi_{[100\varepsilon^2,10\varepsilon]}\nabla\phi_0)\|_{L^2(\mathbb{R}^{d+1})}
\bigg\}
\end{array}
\end{equation*}
where we use the estimates $\eqref{pri:2.13}$ and $\eqref{pri:2.1}$ in the inequality,
and this together with $\eqref{f:4.3}$ partially produces the first term in the right-hand side of $\eqref{pri:4.1}$. Concerning $\eqref{f:4.5}$, using the same argument as in the proof of Lemma
$\ref{lemma:3.1}$ we can easily obtain
\begin{equation}
\begin{aligned}
\big\|
K_\varepsilon(\Psi_{[4\varepsilon^2,2\varepsilon]}
\nabla u_0) & - S_\varepsilon K_\varepsilon(\Psi_{[4\varepsilon^2,2\varepsilon]}\nabla u_0)
\big\|_{L^2(\mathbb{R}^{d+1})}\\
&\leq C\bigg\{\|\nabla u_0\|_{L^2(\boxbox_{2\varepsilon})}
+\varepsilon\|\nabla^2 u_0\|_{L^2(\Sigma_{4\varepsilon^2,2\varepsilon}^T)}
+\varepsilon\|\partial_t u_0\|_{L^2(\Sigma_{4\varepsilon^2,2\varepsilon}^T)}\bigg\}
\end{aligned}
\end{equation}
where we employ the estimates $\eqref{pri:2.14}$, $\eqref{pri:2.17}$. Hence, it is apparent to see that $\eqref{f:4.5}$ is governed by the first term in the right-hand side of $\eqref{pri:4.1}$.

We proceed to address $\eqref{f:4.4}$, which is dominated by
\begin{equation*}
\begin{aligned}
&\bigg\{\big\|
\Psi_{[4\varepsilon^2,2\varepsilon]}
\nabla u_0 - K_\varepsilon(\Psi_{[4\varepsilon^2,2\varepsilon]}\nabla u_0)
\big\|_{L^2(\Sigma_{4\varepsilon^2,2\varepsilon}^T;\delta)}
+ \big\|
K_\varepsilon(\Psi_{[4\varepsilon^2,2\varepsilon]}
\nabla u_0) - S_\varepsilon K_\varepsilon(\Psi_{[4\varepsilon^2,2\varepsilon]}\nabla u_0)
\big\|_{L^2(\Sigma_{4\varepsilon^2,2\varepsilon}^T;\delta)}\bigg\} \\
&\qquad\qquad\times\bigg\{
\big\|
\nabla\phi_0
\big\|_{L^2(\Sigma_{4\varepsilon^2,2\varepsilon}^T;\delta^{-1})}
+
\big\|
\varpi(y,\tau)S_\varepsilon K_\varepsilon(\Psi_{[9\varepsilon^2,5\varepsilon]}\nabla\phi_0)
\big\|_{L^2(\Sigma_{4\varepsilon^2,2\varepsilon}^T;\delta^{-1})}
\bigg\} =: \mathcal{I}.
\end{aligned}
\end{equation*}
Then applying the weighted-type inequalities $\eqref{pri:2.10}$ and $\eqref{pri:2.12}$, we obtain
\begin{equation*}
\begin{aligned}
\mathcal{I}
\leq C\varepsilon\bigg\{
\big\|\nabla(\Psi_{[4\varepsilon^2,2\varepsilon]}
\nabla u_0)\big\|_{L^2(\Sigma_{\varepsilon^2,\varepsilon}^T;\delta)}
&+ \big\|K_\varepsilon(\nabla(\Psi_{[4\varepsilon^2,2\varepsilon]}
\nabla u_0))\big\|_{L^2(\Sigma_{\varepsilon^2,\varepsilon}^T;\delta)} \\
&+ \varepsilon\big\|\partial_t K_\varepsilon(\Psi_{[4\varepsilon^2,2\varepsilon]}
\nabla u_0)\big\|_{L^2(\Sigma_{\varepsilon^2,\varepsilon}^T;\delta)}
\bigg\}\big\|\nabla\phi_0\big\|_{L^2(\Sigma_{4\varepsilon^2,2\varepsilon}^T;\delta^{-1})}.
\end{aligned}
\end{equation*}
Due to the fact $\eqref{f:3.28}$, we arrive at
\begin{equation*}
\mathcal{I}\leq
C\big\|\nabla\phi_0\big\|_{L^2(\Sigma_{4\varepsilon^2,2\varepsilon}^T;\delta^{-1})}
 \bigg\{
\|\nabla u_0\|_{L^2(\boxbox_{2\varepsilon};\delta)}
+\varepsilon\|\nabla^2 u_0\|_{L^2(\Sigma_{4\varepsilon^2,2\varepsilon}^T;\delta)}
+\varepsilon\|\partial_t u_0\|_{L^2(\Sigma_{4\varepsilon^2,2\varepsilon}^T;\delta)}\bigg\}
\end{equation*}
and this exactly gives the second term in the right-hand side of $\eqref{pri:4.1}$. Up to now, we
have completed the estimates for $I_1$. Also, the above proof actually have shown
the following estimates
\begin{equation}\label{f:4.6}
\begin{aligned}
\big\|\mathcal{R}_{\varepsilon,1}\big\|_{L^2(\Omega_T)}
\leq C\bigg\{\|\nabla\breve{w}_\varepsilon\|_{L^2(\Omega_T)}
+ \|\nabla\phi_0\|_{L^2(\boxbox_{2\varepsilon})}
+ \varepsilon\|\nabla^2\phi_0\|_{L^2(\Sigma_{4\varepsilon^2,2\varepsilon}^T)}\bigg\}
\end{aligned}
\end{equation}
and
\begin{equation}\label{f:4.7}
\begin{aligned}
\big\|\mathcal{R}_{\varepsilon,2}\big\|_{L^2(\Sigma_{4\varepsilon^2,2\varepsilon}^T;\delta^{-1})}
\leq C\big\|\nabla\phi_0\big\|_{L^2(\Sigma_{4\varepsilon^2,2\varepsilon}^T;\delta^{-1})}.
\end{aligned}
\end{equation}

Proceeding to study $I_2$ as in the proof for $I_1$, we first have
\begin{equation*}
\begin{aligned}
I_2&\leq \varepsilon\int_{\Omega_T}
\big|\varpi(x/\varepsilon,t/\varepsilon^2)
\nabla S_\varepsilon K_\varepsilon(\Psi_{[4\varepsilon^2,2\varepsilon]}\nabla u_0)\big|
\Big\{\mathcal{R}_{\varepsilon,1}+\mathcal{R}_{\varepsilon,2}\Big\} dxdt\\
&\leq \varepsilon\Bigg\{
\big\|\varpi(y,\tau)\nabla S_\varepsilon K_\varepsilon(\Psi_{[4\varepsilon^2,2\varepsilon]
}\nabla u_0)\big\|_{L^2(\mathbb{R}^{d+1})}\big\|\mathcal{R}_{\varepsilon,1}\big\|_{L^2(\Omega_T)} \\
& \qquad + \big\|\varpi(y,\tau)\nabla S_\varepsilon K_\varepsilon(\Psi_{[4\varepsilon^2,2\varepsilon]
}\nabla u_0)
\big\|_{L^2(\Sigma_{4\varepsilon^2,2\varepsilon}^T;\delta)}
\big\|\mathcal{R}_{\varepsilon,2}\big\|_{L^2(\Sigma_{4\varepsilon^2,2\varepsilon}^T;\delta^{-1})}
\Bigg\},
\end{aligned}
\end{equation*}
and then by the estimates $\eqref{f:4.6}$, $\eqref{f:4.7}$, $\eqref{pri:2.1}$,
$\eqref{pri:2.17}$,
$\eqref{pri:2.5}$ and $\eqref{pri:2.8}$, we acquire
\begin{equation*}
\begin{aligned}
I_2 &\leq C\big\|\nabla\phi_0\big\|_{L^2(\Sigma_{4\varepsilon^2,2\varepsilon}^T;\delta^{-1})} \bigg\{\|\nabla u_0\|_{L^2(\boxbox_{2\varepsilon};\delta)}
+ \varepsilon\|\nabla^2 u_0\|_{L^2(\Sigma_{4\varepsilon^2,2\varepsilon}^T;\delta)}\bigg\} \\
& + C\bigg\{\|\nabla u_0\|_{L^2(\boxbox_{2\varepsilon})}
+ \varepsilon\|\nabla^2 u_0\|_{L^2(\Sigma_{4\varepsilon^2,2\varepsilon}^T)}
\bigg\}\bigg\{\|\nabla\breve{w}_\varepsilon\|_{L^2(\Omega_T)}
+ \|\nabla\phi_0\|_{L^2(\boxbox_{2\varepsilon})}
+ \varepsilon\|\nabla^2\phi_0\|_{L^2(\Sigma_{4\varepsilon^2,2\varepsilon}^T)}\bigg\}.
\end{aligned}
\end{equation*}

To estimate $I_3$, it suffices to estimate
\begin{equation*}
\underbrace{\big\|\varpi(y,\tau)\nabla^2 S_\varepsilon K_\varepsilon(\Psi_{[4\varepsilon^2,2\varepsilon]
}\nabla u_0)\big\|_{L^2(\mathbb{R}^{d+1})}}_{I_{31}}
\quad
\text{and}
\quad
\underbrace{\big\|\varpi(y,\tau)\nabla^2 S_\varepsilon K_\varepsilon(\Psi_{[4\varepsilon^2,2\varepsilon]
}\nabla u_0)
\big\|_{L^2(\Sigma_{4\varepsilon^2,2\varepsilon}^T;\delta)}}_{I_{32}}.
\end{equation*}
Thus, it follows from the estimates $\eqref{pri:2.1}$ and $\eqref{pri:2.17}$ that
\begin{equation}\label{f:4.8}
I_{31}\leq C\varepsilon^{-1}\big\|\nabla K_\varepsilon
(\Psi_{[4\varepsilon^2,2\varepsilon]
}\nabla u_0)\big\|_{L^2(\mathbb{R}^{d+1})}
\leq C\bigg\{\varepsilon^{-2}\big\|\nabla u_0\big\|_{L^2(\boxbox_{2\varepsilon})}
+\varepsilon^{-1}\big\|\nabla^2 u_0\big\|_{L^2(\Sigma_{4\varepsilon^2,2\varepsilon}^T)}\bigg\},
\end{equation}
and from the estimates $\eqref{pri:2.5}$ and $\eqref{pri:2.8}$ that
\begin{equation}\label{f:4.9}
\begin{aligned}
I_{32}&\leq C\varepsilon^{-1}
\big\| K_\varepsilon
(\nabla(\Psi_{[4\varepsilon^2,2\varepsilon]
}\nabla u_0))\big\|_{L^2(\Sigma_{4\varepsilon^2,2\varepsilon}^T;\delta)} \\
&\leq C\bigg\{\varepsilon^{-2}\big\|\nabla u_0\big\|_{L^2(\boxbox_{2\varepsilon};\delta)}
+\varepsilon^{-1}\big\|\nabla^2 u_0\big\|_{L^2(\Sigma_{4\varepsilon^2,2\varepsilon}^T;\delta)}\bigg\}.
\end{aligned}
\end{equation}
Combining the estimates $\eqref{f:4.6}$, $\eqref{f:4.7}$, $\eqref{f:4.8}$ and $\eqref{f:4.9}$
gives the corresponding estimate for $I_3$, which partially forms the right-hand side of $\eqref{pri:4.1}$.

For $I_4$, using the same argument as before, we are ready to establish
estimates for
\begin{equation*}
\underbrace{\big\|\varpi(y,\tau)\partial_t S_\varepsilon K_\varepsilon(\Psi_{[4\varepsilon^2,2\varepsilon]
}\nabla u_0)\big\|_{L^2(\mathbb{R}^{d+1})}}_{I_{41}}
\quad
\text{and}
\quad
\underbrace{\big\|\varpi(y,\tau)\partial_t S_\varepsilon K_\varepsilon(\Psi_{[4\varepsilon^2,2\varepsilon]
}\nabla u_0)
\big\|_{L^2(\Sigma_{4\varepsilon^2,2\varepsilon}^T;\delta)}}_{I_{42}},
\end{equation*}
respectively. Recalling the crucial observation $\eqref{f:3.29}$, the
estimate for $I_{41}$ actually have been shown in the proof of Lemma $\ref{lemma:3.1}$, which
is
\begin{equation}\label{f:4.10}
I_{41} \leq C\bigg\{\varepsilon^{-2}\|\nabla u_0\|_{L^2(\boxbox_{2\varepsilon})}
+ \varepsilon^{-1}\|\partial_t u_0\|_{L^2(\Sigma_{4\varepsilon^2,2\varepsilon}^T)}\bigg\}.
\end{equation}
Regarding $I_{42}$, employing $\eqref{f:3.29}$ again, it follows from the estimates $\eqref{pri:2.5}$ and $\eqref{pri:2.8}$
that
\begin{equation}\label{f:4.11}
I_{42}
\leq C\bigg\{\varepsilon^{-2}\big\|\nabla u_0\big\|_{L^2(\boxbox_{2\varepsilon};\delta)}
+\varepsilon^{-1}\big\|\partial_t u_0\big\|_{L^2(\Sigma_{4\varepsilon^2,2\varepsilon}^T;\delta)}\bigg\}.
\end{equation}
Thus, the estimates $\eqref{f:4.6}$, $\eqref{f:4.7}$, $\eqref{f:4.10}$ and $\eqref{f:4.11}$
lead to
\begin{equation*}
\begin{aligned}
I_4
&\leq C\big\|\nabla\phi_0\big\|_{L^2(\Sigma_{4\varepsilon^2,2\varepsilon}^T;\delta^{-1})}
\bigg\{\big\|\nabla u_0\big\|_{L^2(\boxbox_{2\varepsilon};\delta)}
+\varepsilon\big\|\partial_t u_0\big\|_{L^2(\Sigma_{4\varepsilon^2,2\varepsilon}^T;\delta)}\bigg\} \\
&+ C\bigg\{\|\nabla u_0\|_{L^2(\boxbox_{2\varepsilon})}
+ \varepsilon\|\partial_t u_0\|_{L^2(\Sigma_{4\varepsilon^2,2\varepsilon}^T)}\bigg\}
\bigg\{\|\nabla\breve{w}_\varepsilon\|_{L^2(\Omega_T)}
+ \|\nabla\phi_0\|_{L^2(\boxbox_{2\varepsilon})}
+ \varepsilon\|\nabla^2\phi_0\|_{L^2(\Sigma_{4\varepsilon^2,2\varepsilon}^T)}\bigg\},
\end{aligned}
\end{equation*}
and this ends the proof.
\end{proof}

\begin{lemma}[Weighted Caccioppoli's inequality]\label{lemma:4.1}
Let $u_0\in (W^{1,1}_2(\Omega_T))^d$ be a weak solution of
$\partial_t u_0 + \mathcal{L}_0(u_0) = F$ in $\Omega_T$,
and the distance functions $\delta,\sigma$ are defined in $\eqref{eq:2.4}$ and
$\eqref{eq:2.5}$, respectively.
Then for any $\varepsilon\leq t_{k-1}<t_k<t_{k+1}\leq c_0$, we have the following
weighted-type estimate
\begin{equation}\label{pri:4.4}
\begin{aligned}
\int_{t_k^2}^{t_{k+1}^2}\int_{\Sigma_{t_k}\cap\{\delta=s^{\frac{1}{2}}\}}|\nabla u_0|^2
s^{-\frac{1}{2}}dx ds
&\leq \frac{C}{(t_k-t_{k-1})^2}\int_{t_k^2}^{t_{k+1}^2}
\int_{\Sigma_{t_{k-1}}\setminus\Sigma_{t_k}}|u_0|^2
[\emph{dist}(x,\partial\Omega)]^{-1}dxds\\
&+C\int_{t_k^2}^{t_{k+1}^2}\int_{\Sigma_{t_{k-1}}\cap\{\delta=\sigma\}}|u_0|^2
[\emph{dist}(x,\partial\Omega)]^{-3}dx ds\\
& + C\int_{t_k^2}^{t_{k+1}^2}\int_{\Sigma_{t_{k-1}}\cap\{\delta=s^{\frac{1}{2}}\}}
|u_0|^2 s^{-1}dxds \\
&+ C\Bigg\{
\int_{t_k^2}^{t_{k+1}^2}\int_{\Sigma_{t_{k-1}}}
|\partial_t u_0|^2 dxds +\int_{\Omega_T} |F|^2 dxds
\Bigg\},
\end{aligned}
\end{equation}
where $C$ depends on $\mu_1,\mu_2,d$.
\end{lemma}

\begin{proof}
We will carry out the proof
by suitable modification to the one for the estimate $\eqref{pri:3.4}$ in Lemma $\ref{lemma:3.3}$.
Set $t_0 = \varepsilon$.
In view of the estimate $\eqref{f:3.31}$, we choose $\psi = \varphi\delta^{-\frac{1}{2}}$,
where the distance function $\delta$ is defined in $\eqref{eq:2.4}$,
and $\varphi\in C_0^1(\Omega)$ is a cut-off function satisfying $\varphi = 1$ in $\Sigma_{t_k}$,
and $\varphi = 0$ outside $\Sigma_{t_{k-1}}$ with $|\nabla\varphi|\leq C/(t_k-t_{k-1})$.
Then, putting
\begin{equation*}
 |\nabla\psi|\leq |\nabla\varphi|\delta^{-\frac{1}{2}}
+ |\varphi||\nabla \delta|\delta^{-\frac{3}{2}}
\end{equation*}
into the estimate $\eqref{f:3.31}$, for any fixed $s\in[t_k^2,t_{k+1}^2]$, we have
\begin{equation*}
\begin{aligned}
\int_{\Sigma_{t_k}}|\nabla u_0|^2 \delta^{-1}dx
&\leq C
\Bigg\{\frac{1}{(t_k-t_{k-1})^2}
\int_{\Sigma_{t_k}\setminus\Sigma_{t_{k-1}}}
|u_0|^2 \delta^{-1}dx
+ \int_{\Sigma_{t_{k-1}}}|u_0|^2|\nabla \delta|^2 \delta^{-3} dx\\
& + \int_{\Sigma_{t_{k-1}}}|u_0|^2\delta^{-2}dx
+\int_{\Sigma_{t_{k-1}}}|\partial_t u_0|^2dx
+ \int_{\Omega}|F|^2dx\Bigg\}.
\end{aligned}
\end{equation*}
Then integrating
the both sides of the above inequality with respect to $s$ from $t_k^2$ to $t_{k+1}^2$, and
we have proved the following estimate:
\begin{equation}\label{f:4.27}
\begin{aligned}
\int_{t_k^2}^{t_{k+1}^2}\int_{\Sigma_{t_k}}|\nabla u_0|^2
[\delta(x,s)]^{-1}dx ds
&\leq \frac{C}{(t_k-t_{k-1})^2}\underbrace{\int_{t_k^2}^{t_{k+1}^2}
\int_{\Sigma_{t_{k-1}}\setminus\Sigma_{t_k}}|u_0|^2
[\delta(x,s)]^{-1}dxds}_{I_1}\\
&+C\underbrace{\int_{t_k^2}^{t_{k+1}^2}\int_{\Sigma_{t_{k-1}}}|u_0|^2
|\nabla\delta(x,s)|^2[\delta(x,s)]^{-3}dx ds}_{I_2}\\
& + C\underbrace{\int_{t_k^2}^{t_{k+1}^2}\int_{\Sigma_{t_{k-1}}}
|u_0|^2 [\delta(x,s)]^{-2}dxds}_{I_3} \\
&+ C\Bigg\{
\int_{t_k^2}^{t_{k+1}^2}\int_{\Sigma_{t_{k-1}}}
|\partial_t u_0|^2 dxds +\int_{\Omega_T} |F|^2 dxds
\Bigg\}.
\end{aligned}
\end{equation}
Our task now is to analyze the concrete behavior of the distance function $\delta$ in
the integrals $I_1, I_2$ and $I_3$. Since $\delta(x,s)=\text{dist}(x,\partial\Omega)$ for
any $(x,s)\in(\Sigma_{t_{k-1}}\setminus\Sigma_{t_k})\times[t_k^2,t_{k+1}^2)$, we have
\begin{equation}\label{f:4.24}
I_1 = \int_{t_k^2}^{t_{k+1}^2}
\int_{\Sigma_{t_{k-1}}\setminus\Sigma_{t_k}}|u_0|^2
[\text{dist}(x,\partial\Omega)]^{-1}dxds.
\end{equation}
By noting that
$$\Sigma_{t_{k-1}} = \big(\Sigma_{t_{k-1}}\cap\{\delta =\sigma\}\big)\bigcup
\big(\Sigma_{t_{k-1}}\cap\{\delta =s^{\frac{1}{2}}\}\big)$$
and $\nabla \delta \equiv 0$ in $\Sigma_{t_{k-1}}\cap\{\delta = s^{\frac{1}{2}}\}$,
the integral $I_2$ is just equal to
\begin{equation}\label{f:4.25}
 \int_{t_k^2}^{t_{k+1}^2}
\int_{\Sigma_{t_{k-1}}\cap\{\delta=\sigma\}}
|u_0|^2[\text{dist}(x,\partial\Omega)]^{-3}dxds,
\end{equation}
while the integral $I_3$ will produce two terms:
\begin{equation}\label{f:4.26}
\int_{t_k^2}^{t_{k+1}^2}\int_{\Sigma_{t_{k-1}}\cap\{\delta =\sigma\}}
|u_0|^2 [\text{dist}(x,\partial\Omega)]^{-2}dxds
+
\int_{t_k^2}^{t_{k+1}^2}\int_{\Sigma_{t_{k-1}}\cap\{\delta =s^{\frac{1}{2}}\}}
|u_0|^2 s^{-1}dxds.
\end{equation}
We ends the proof by
substituting the expressions $\eqref{f:4.24}$, $\eqref{f:4.25}$ and
$\eqref{f:4.26}$ for $I_1$, $I_2$ and $I_3$ in the estimate $\eqref{f:4.27}$, respectively.
Obviously, the first term of $\eqref{f:4.26}$ could be absorbed by $\eqref{f:4.25}$, and
we have completed the whole proof.
\end{proof}

\begin{lemma}[Improved lemma]\label{lemma:4.2}
Assume the same conditions as in Lemma $\ref{lemma:3.2}$. Let $u_0$ be the weak solution
of $(\emph{DP}_0)$ with $F\in (L^2(\Omega_T))^d$, $g\in ({_0H}^{1,1/2}(S_T))^d$ and
$h\in (H^1(\Omega))^d$ satisfying
\begin{equation*}
 \|F\|_{L^2(\Omega_T)}+\|g\|_{H^{1,1/2}(S_T)}
+ \|h\|_{H^1(\Omega)} = 1.
\end{equation*}
Then we have
\begin{equation}\label{pri:4.2}
\big\|\nabla u_0\big\|_{L^2(\boxbox_{2\varepsilon};\delta)}
\leq C\varepsilon,
\end{equation}
and
\begin{equation}\label{pri:4.3}
\begin{aligned}
\max\bigg\{\big\|\nabla^2 u_0\big\|_{L^2(\Sigma_{4\varepsilon^2,2\varepsilon}^T;\delta)},
~\big\|\partial_t u_0\big\|_{L^2(\Sigma_{4\varepsilon^2,2\varepsilon}^T;\delta)},
~\big\|\nabla u_0\big\|_{L^2(\Sigma_{4\varepsilon^2,2\varepsilon}^T;\delta^{-1})}\bigg\}
\leq C\big[\log_2(c_0/\varepsilon)\big]^{1/2},
\end{aligned}
\end{equation}
where $C$ depends on $\mu_1,\mu_1,d,T$ and $\Omega$.
\end{lemma}

\begin{proof}
First, we will address the estimate $\eqref{pri:4.2}$. Recalling
$\boxbox_{2\varepsilon} = \Omega_T\setminus \Sigma_{8\varepsilon^2,4\varepsilon}^T$ and the
definition of the distance function $\delta$ (see $\eqref{eq:2.4}$),
it is
not hard to see that
\begin{equation*}
\begin{aligned}
\|\nabla u_0\|_{L^2(\boxbox_{2\varepsilon;}\delta)}
&\leq C\varepsilon^{1/2}\|\nabla u_0\|_{L^2(\boxbox_{2\varepsilon})}\\
&\leq C\varepsilon^{1/2}\Bigg\{\|\nabla u_0\|_{L^2((\Omega\setminus\Sigma_{4\varepsilon})_T)}
+\sup_{\varepsilon^2<t<T}\Big(\int_{t-\varepsilon^2}^t
\int_{\Sigma_{4\varepsilon}}|\nabla u_0|^2 dxds\Big)^{\frac{1}{2}}\Bigg\}
\leq C\varepsilon,
\end{aligned}
\end{equation*}
where we employ the estimates $\eqref{pri:3.1}$ and $\eqref{pri:3.2}$ in the last step.

We now turn to study the estimate $\eqref{pri:4.3}$.
Using the same arguments as in the proof of the estimate $\eqref{pri:3.3}$ to prove the first two quantities in the left-hand side of $\eqref{pri:4.3}$, it suffices to bound
\begin{equation*}
\int_{4\varepsilon^2}^{T-4\varepsilon^2}\int_{\Sigma_{2\varepsilon}}
|\nabla^2 u_0|^2 \delta(x,t)dxdt.
\end{equation*}
by
\begin{equation*}
r_0\underbrace{\Bigg\{\int_{4\varepsilon^2}^{T-4\varepsilon^2}\int_{\Omega}
|\nabla^2 v|^2 dxdt
+\int_{4\varepsilon^2}^{T-4\varepsilon^2}\int_{\Sigma_{c_0}}
|\nabla^2 \bar{w}|^2 dxdt\Bigg\}}_{I_1}
+ \underbrace{\int_{4\varepsilon^2}^{T-4\varepsilon^2}\int_{\Sigma_{2\varepsilon}\setminus\Sigma_{c_0}}
|\nabla^2 \bar{w}|^2 \delta(x,t) dxdt}_{I_2},
\end{equation*}
recalling $u_0 = v + \bar{w}$, and $v$ together with $\bar{w}$ is shown in the proof of
Lemma $\ref{lemma:3.2}$. Obviously, it follows from the estimates $\eqref{f:3.19}$ and $\eqref{f:3.20}$ that
\begin{equation*}
I_1 \leq C.
\end{equation*}
Thus, our task is reduced to estimate $I_2$. On account of the estimate $\eqref{f:3.30}$, we obtain
\begin{equation*}
\begin{aligned}
I_2
&\leq C\int_{4\varepsilon^2}^{T-4\varepsilon}
\int_{\Sigma_{2\varepsilon}\setminus\Sigma{c_0}}
\frac{1}{\sigma(X)}
\dashint_{P(X,\sigma(X)/4)}|\nabla\bar{w}|^2 dY dX \\
&\leq C\int_0^T\int_{\partial\Omega}|(\nabla\bar{w})^*(\cdot,t)|^2 dSdt
\int_{2\varepsilon}^{c_0}\frac{dr}{r}\\
&\leq C\ln(c_0/\varepsilon)
\end{aligned}
\end{equation*}
by noting that $\delta(x)/\sigma(X)\leq 1$.
This together the estimate for $I_1$ implies the estimate
\begin{equation}\label{f:4.12}
\big\|\nabla^2 u_0\big\|_{L^2(\Sigma_{4\varepsilon^2,2\varepsilon}^T;\delta)}
\leq C[\ln(c_0/\varepsilon)]^{1/2}
\Big\{\|F\|_{L^2(\Omega_T)}+\|g\|_{H^{1,1/2}(S_T)}
+ \|h\|_{H^1(\Omega)}\Big\}.
\end{equation}
Since $\partial_t u_0 + \mathcal{L}_0(u_0) = F$  in $\Omega_T$,  to estimate
\begin{equation*}
\int_{4\varepsilon^2}^{T-4\varepsilon^2}\int_{\Sigma_{2\varepsilon}}
|\partial_t u_0|^2 \delta(x,t)dxdt.
\end{equation*}
is reduced to the estimate $\eqref{f:4.12}$.

So, the remainder thing is to investigate the quantity
\begin{equation*}
\int_{4\varepsilon^2}^{T-4\varepsilon^2}\int_{\Sigma_{2\varepsilon}}
|\nabla u_0|^2 [\delta(x,t)]^{-1}dxdt,
\end{equation*}
which will be divided by four parts:
\begin{equation*}
\underbrace{\int_{4\varepsilon^2}^{c_0^2}\int_{\Sigma_{2\varepsilon}}
|\nabla u_0|^2 [\delta(x,t)]^{-1}dxdt}_{J_1},
\qquad\quad
\underbrace{\int_{T-c_0^2}^{T-4\varepsilon^2}\int_{\Sigma_{2\varepsilon}}
|\nabla u_0|^2 [\delta(x,t)]^{-1}dxdt}_{J_2},
\end{equation*}
and
\begin{equation*}
\underbrace{\int_{c_0^2}^{T-c_0^2}\int_{\Sigma_{2\varepsilon}\setminus\Sigma_{c_0}}
|\nabla u_0|^2 [\text{dist}(x,\partial\Omega)]^{-1}dxdt}_{J_3},
\qquad
\underbrace{\int_{c_0^2}^{T-c_0^2}\int_{\Sigma_{c_0}}
|\nabla u_0|^2 [\delta(x,t)]^{-1}dxdt}_{J_4},
\end{equation*}
where we note that $\delta(x,t) = \text{dist}(x,\partial\Omega)$ in $J_3$.

Since $\delta(x,t)\geq c_0$ in $J_4$,  we can easily have
\begin{equation}\label{f:4.17}
J_4 \leq c_0^{-1}\int_{0}^{T}\int_{\Omega}
|\nabla u_0|^2dxdt
\leq C
\end{equation}
from the estimate $\eqref{pri:3.5}$. To estimate $J_3$, we first define the radical maximal function as
\begin{equation}
\mathcal{M}(\nabla u_0)(x_0,t) = \sup\big\{|\nabla u_0(\Lambda_r(x_0),t)|:0\leq r\leq c_0\big\}
\end{equation}
for any $x_0\in\partial\Omega$ and $t\in(0,T)$, where $\Lambda_r:\partial\Omega\to \partial\Sigma_r =S_r$ are a family of bi-Lipschitz maps (see for example \cite{SZW2,QX2}).
Hence, by co-area formula, it is not hard to derive
\begin{equation}\label{f:4.13}
\begin{aligned}
J_3
&\leq C\int_{c^2}^{T-c^2}\int_{\partial\Omega}|\mathcal{M}(\nabla u_0)(\cdot,t)|^2dSdt\int_{2\varepsilon}^{c_0}\frac{dr}{r}\\
&\leq C\ln(c_0/\varepsilon)\bigg\{
\int_{0}^T\int_{\partial\Omega}|\mathcal{M}(\nabla v)|^2dSdt
+ \int_{0}^T\int_{\partial\Omega}|\mathcal{M}(\nabla \bar{w})|^2dSdt\bigg\}\\
&\leq C\ln(c_0/\varepsilon)\Bigg\{
\int_{0}^T\int_{\Omega\setminus\Sigma_{c_0}}\big(|\nabla^2 v|^2
+|\nabla v|^2\big)dxdt
+ \int_{0}^T\int_{\partial\Omega}|(\nabla \bar{w})^*|^2dSdt\Bigg\}\\
&\leq C\ln(c_0/\varepsilon),
\end{aligned}
\end{equation}
where we use the fact that $u_0 = v+\bar{w}$ in the second inequality, and the third one
involves two important things: one is
\begin{equation*}
\begin{aligned}
\int_{\partial\Omega}|\mathcal{M}(\nabla v)(\cdot,t)|^2dS
\leq \int_{\Omega\setminus\Sigma_{c_0}}\big(|\nabla^2 v(\cdot,t)|^2
+|\nabla v(\cdot,t)|^2\big)dx
\qquad\forall t\in[0,T]
\end{aligned}
\end{equation*}
(see \cite[Lemma 2.24]{QXS1}),
the other is the observation that $\mathcal{M}(\nabla \bar{w})\leq (\nabla\bar{w})^*$ on $S_T$.
The last inequality of $\eqref{f:4.13}$ follows from the estimates $\eqref{f:3.7}$ and
$\eqref{f:3.13}$.

We now turn to $J_1$ and $J_2$. In fact,
by changing variable, the study on $J_2$ may be reduced to investigate $J_1$. Thus we focus
our minds on $J_1$. First of all, it is better to slip $J_1$ in two parts:
\begin{equation*}
\underbrace{\int_{4\varepsilon^2}^{c_0^2}
\int_{\Sigma_{2\varepsilon}\cap\{\delta = t^{\frac{1}{2}}\}}|\nabla u_0|^2
t^{-\frac{1}{2}}dxdt}_{J_{11}}
\qquad\text{and}\qquad
\underbrace{\int_{4\varepsilon^2}^{c_0^2}
\int_{\Sigma_{2\varepsilon}\cap\{\delta = \sigma\}}|\nabla u_0|^2 [\text{dist}(x,\partial\Omega)]^{-1}dxdt}_{J_{12}}.
\end{equation*}
It is clear to see that we can adopt the same arguments used in $J_3$ to derive
\begin{equation}\label{f:4.22}
 J_{12} \leq C\ln(c_0/\varepsilon),
\end{equation}
and we do not reproduce the details here. Concerning $J_{11}$, set $t_k = 2^k\varepsilon$ and
$N_0 = \log_2(c_0/\varepsilon)$, which follows from $2^k\varepsilon = c_0$. Hence, we obtain
\begin{equation}\label{f:4.23}
 J_{11} \leq \sum_{k=1}^{N_0}\int_{t_k^2}^{t_{k+1}^2}
\int_{\Sigma_{t_k}\cap\{\delta=s^{\frac{1}{2}}\}}
|\nabla u_0|^2 s^{-\frac{1}{2}}dxds :=\sum_{k=1}^{N_0} \mathcal{K}_k
\leq C\log_2(c_0/\varepsilon)
\end{equation}
provided that
\begin{equation}\label{f:4.21}
 \sup_{k\geq 1}\mathcal{K}_k\leq C.
\end{equation}
Hence, the problem is reduced to estimate $\eqref{f:4.21}$. It follows from
the weighted Caccioppoli's inequality $\eqref{pri:4.4}$ that
\begin{equation}\label{f:4.28}
\begin{aligned}
 \mathcal{K}_{k}
 &= \int_{t_k^2}^{t_{k+1}^2}
 \int_{\Sigma_{t_k}\cap\{\delta=s^{\frac{1}{2}}\}}|\nabla u_0|^2 s^{-\frac{1}{2}} dx ds \\
&\leq Ct_{k-1}^{-2}\int_{t_k^2}^{t_{k+1}^2}
\int_{\Sigma_{t_{k-1}}\setminus\Sigma_{t_{k}}}|u_0|^2 [\text{dist}(x,\partial\Omega)]^{-1}dxds
+ Ct_{k-1}^{-3}\int_{t_k^2}^{t_{k+1}^2}
\int_{\Sigma_{t_{k-1}}\cap\{\delta=\sigma\}}|u_0|^2 dxds\\
& + C\Bigg\{\int_{t_k^2}^{t_{k+1}^2}\int_{\Sigma_{t_{k-1}}\cap\{\delta=s^{\frac{1}{2}}\}}
|u_0|^2 s^{-1}dxds
+ \int_{t_k^2}^{t_{k+1}^2}\int_{\Sigma_{\varepsilon}}
|\partial_t u_0|^2 dxds
+ \int_{\Omega_T}
|F|^2 dxds
\Bigg\},
\end{aligned}
\end{equation}
by noting that $\sigma\geq t_{k-1}$
in $(\Sigma_{t_{k-1}}\cap\{\delta=\sigma\})\times[t_k^2,t_{k+1}^2)$. Since
there holds
\begin{equation*}
 \Sigma_{t_{k-1}}\setminus \Sigma_{t_{k}}\subset \Sigma_{t_{k-1}}\cap \{\delta=\sigma\}
\subset \Sigma_{t_{k-1}}\setminus \Sigma_{t_{k+1}}
\end{equation*}
for any $s\in[t_k,t_{k+1})$, the second line of $\eqref{f:4.28}$ is reduced to estimate
the quantity
\begin{equation*}
 t_{k-1}^{-3}\int_{t_k^2}^{t_{k+1}^2}
\int_{\Sigma_{t_{k-1}}\setminus\Sigma_{t_{k+1}}}|u_0|^2 dxds.
\end{equation*}
Hence, following the methods employed in the estimate $\eqref{f:3.17}$ we arrive at
\begin{equation*}
\begin{aligned}
t_{k-1}^{-3}\int_{t_k^2}^{t_{k+1}^2}
\int_{\Sigma_{t_{k-1}}\setminus\Sigma_{t_{k+1}}}|u_0|^2 dxds
&\leq Ct_{k-1}^{-1}
\sup_{0\leq s\leq T}\int_{\Sigma_{t_{k-1}}\setminus\Sigma_{t_{k+1}}}|u_0(\cdot,s)|^2 dx \\
&\leq C\sup_{0\leq s\leq T}\int_{\mathbb{R}^d}(|v|^2 + |\nabla v|^2) dx +
C\int_{\partial\Omega}|(\bar{w})^*(\cdot,\xi_k)|^2 dS
\end{aligned}
\end{equation*}
where $\xi_k\in[t_k,t_{k+1})$. Integrating the both sides of the above inequality with respect to
$\xi_k$ from $0$ to $T$, we consequently reach
\begin{equation}\label{f:4.18}
 t_{k-1}^{-3}\int_{t_k^2}^{t_{k+1}^2}
\int_{\Sigma_{t_{k-1}}\setminus\Sigma_{t_{k+1}}}|u_0|^2 dxds
\leq C
\end{equation}
through the estimates $\eqref{f:3.7}$ and $\eqref{f:3.13}$.

An argument similar to the one used in the estimate $\eqref{f:3.17}$ shows that
\begin{equation}\label{f:4.16}
\begin{aligned}
\int_{t_k^2}^{t_{k+1}^2}
\int_{\Sigma_{t_{k-1}}\cap\{\delta=s^{\frac{1}{2}}\}}
|u_0|^2 s^{-1}dxds
&\leq \sup_{0\leq s\leq T}
\int_{\Omega}|u_0|^2dx\int_{t_k^2}^{t_{k+1}^2}\frac{ds}{s}
\leq C,
\end{aligned}
\end{equation}
where we note that $t_{k+1}/t_k = 2$. Besides,
adopting the idea similar to the one used in the estimate $\eqref{f:3.18}$, it is not hard to derive \begin{equation}\label{f:4.20}
\Big(\int_{t_k^2}^{t_{k+1}^2}\int_{\Sigma_{2\varepsilon}}
|\partial_t u_0|^2 dxds\Big)^{\frac{1}{2}}
\leq C.
\end{equation}

Thus, collecting the estimates $\eqref{f:4.20}$,$\eqref{f:4.16}$ and $\eqref{f:4.18}$ leads to the desired estimate $\eqref{f:4.21}$. Finally, the estimates
$\eqref{f:4.22}$ and $\eqref{f:4.23}$ give
\begin{equation*}
  J_1 \leq C\log_2(c_0/\varepsilon),
\end{equation*}
and this together with the estimates $\eqref{f:4.17}$ and  $\eqref{f:4.13}$ implies
\begin{equation*}
\big\|\nabla u_0\big\|_{L^2(\Sigma_{4\varepsilon^2,2\varepsilon}^T;\delta^{-1})}
\leq
C\big[\log_2(c_0/\varepsilon)\big]^{\frac{1}{2}}.
\end{equation*}

Up to now, we have proved  the estimate $\eqref{pri:4.3}$ and completed the whole proof.
\end{proof}

\begin{flushleft}
\textbf{Proof of Theorem $\ref{thm:1.1}$.}
First of all, we may assume
$\|F\|_{L^2(\Omega_T)}+\|h\|_{H^1(\Omega)}
+ \|g\|_{H^{1,1/2}(S_T)} = 1$.
On account of Lemmas $\ref{lemma:4.3}$ and $\ref{lemma:4.2}$ and
$\ref{lemma:3.2}$, for any $\Phi\in L^2(\Omega_T)$, it
is not hard to derive from
the estimates $\eqref{pri:4.1}$, $\eqref{pri:4.2}$ and $\eqref{pri:4.3}$ together
with $\eqref{pri:3.1}$, $\eqref{pri:3.2}$ that
\end{flushleft}
\begin{equation}
\begin{aligned}
\bigg|\int_{\Omega_T} w_\varepsilon
\Phi dxdt\bigg|
&\leq C\Big\{\varepsilon^{\frac{1}{2}}
+ \varepsilon^{\frac{1}{2}} + \varepsilon^{\frac{1}{2}}\Big\}\cdot\Big\{\varepsilon^{\frac{1}{2}}
+\varepsilon^{\frac{1}{2}}
+\varepsilon^{\frac{1}{2}}\Big\}\|\Phi\|_{L^2(\Omega_T)} \\
&+ C\big[\log_2(c_0/\varepsilon)\big]^{\frac{1}{2}}
\cdot\Big\{\varepsilon
+ \varepsilon\big[\log_2(c_0/\varepsilon)\big]^{\frac{1}{2}}
+ \varepsilon\big[\log_2(c_0/\varepsilon)\big]^{\frac{1}{2}}\Big\}\|\Phi\|_{L^2(\Omega_T)}
\end{aligned}
\end{equation}
where we also employed the estimate $\eqref{pri:4.5}$ in the computation, which shows
that
\begin{equation}\label{f:4.29}
\|w_\varepsilon\|_{L^2(\Omega_T)}\leq C\varepsilon\log_2(c_0/\varepsilon).
\end{equation}
Since
\begin{equation*}
\begin{aligned}
\big\|u_\varepsilon - u_0\big\|_{L^2(\Omega_T)}
&\leq \big\|w_\varepsilon\big\|_{L^2(\Omega_T)}
+ \varepsilon\big\|\chi_j(\cdot/\varepsilon,\cdot/\varepsilon^2)
S_\varepsilon K_\varepsilon(\Psi_{[4\varepsilon^2,2\varepsilon]}\nabla_j u_0)\big\|_{L^2(\Omega_T)}\\
&+ \varepsilon^2
\big\|E_{l(d+1)j}(\cdot/\varepsilon,\cdot/\varepsilon^2)\nabla_l
S_\varepsilon K_\varepsilon(\Psi_{[4\varepsilon^2,2\varepsilon]}\nabla_j u_0)\big\|_{L^2(\Omega_T)},
\end{aligned}
\end{equation*}
it suffices to estimate
\begin{equation*}
\big\|\chi_j(\cdot/\varepsilon,\cdot/\varepsilon^2)
S_\varepsilon K_\varepsilon(\Psi_{[4\varepsilon^2,2\varepsilon]}\nabla_j u_0)
\big\|_{L^2(\mathbb{R}^{d+1})}
\quad\text{and}\quad
\varepsilon\big\|E_{l(d+1)j}(\cdot/\varepsilon,\cdot/\varepsilon^2)\nabla_l
S_\varepsilon K_\varepsilon(\Psi_{[4\varepsilon^2,2\varepsilon]}\nabla_j u_0)
\big\|_{L^2(\mathbb{R}^{d+1})}.
\end{equation*}
Form the estimates $\eqref{pri:2.1}$ and $\eqref{pri:2.17}$, we can assert that
the above two quantities will be determined by
$\|\nabla u_0\|_{L^2(\Omega_T)}$.
This together with the estimate $\eqref{pri:3.5}$ and $\eqref{f:4.29}$ implies
\begin{equation*}
 \|u_\varepsilon - u_0\|_{L^2(\Omega_T)} \leq C\varepsilon\log_2(c_0/\varepsilon).
\end{equation*}
We have completed all the proof.
\qed

\begin{center}
\textbf{Acknowledgements}
\end{center}

The first author wants to express his sincere appreciation to Professor Zhongwei Shen for his constant guidance and encouragement. The first author was supported by the National Natural Science Foundation of China (Grant NO.11471147). The second author was supported by the National Natural Science Foundation of China (Grant NO.11571020).

\end{document}